\documentclass[a4paper,10pt,reqno,draft]{amsart}

\usepackage{amsthm, amssymb, amsmath, latexsym}
\usepackage{amsfonts, mathrsfs, color}
\usepackage{a4wide}
\usepackage{listings}
\usepackage{color}
\usepackage[usenames,dvipsnames]{xcolor}
\usepackage{rotating}
\usepackage{hhline}
\usepackage{subfigure}
\usepackage{multirow}
\usepackage{enumerate}
\usepackage[bookmarks]{}
\usepackage{amsfonts}   
\usepackage{dsfont}
\usepackage{tikz}

 \theoremstyle{plain}
 \begingroup
 \newtheorem{thm}{Theorem}[section]
 \newtheorem{lem}[thm]{Lemma}
 
 \newtheorem{cor}[thm]{Corollary}
 \endgroup

 \theoremstyle{definition}
 \begingroup
 \newtheorem{defn}[thm]{Definition}
 \newtheorem{example}[thm]{Example}
 \endgroup

 \theoremstyle{remark}
 \begingroup
 \newtheorem{rem}[thm]{Remark}
 \endgroup

 \numberwithin{equation}{section}

\mathchardef\emptyset="001F

\newcommand{\hs}{{\mathcal H}}

\newcommand{\ds}{\displaystyle}

\newcommand{\G}{{\mathcal G}}
\newcommand{\dx}{\,dx}

\newcommand{\R}{{\mathbb R}}

\newcommand{\Q}{{\mathbb Q}}
\newcommand{\N}{{\mathbb N}}
\newcommand{\Sph}{{\mathbb S}}

\newcommand{\e}{\varepsilon}

\newcommand{\ie}{{; \it i.e., }}

\newcommand{\A}{\mathscr{A}}

\title[]
{$\mathbf\Gamma$-Convergence of Free-discontinuity Problems}

\author[]{FILIPPO CAGNETTI}
\address[]{Department of Mathematics,
University of Sussex, Brighton, United Kingdom}
\email[]{F.Cagnetti@sussex.ac.uk}
\author[]{GIANNI DAL MASO}
\address[]{SISSA, Trieste, Italy}
\email[]{dalmaso@sissa.it}
\author[]{LUCIA SCARDIA}
\address[]{Department of Mathematical Sciences, University of Bath, Bath, United Kingdom}
\email[]{L.Scardia@bath.ac.uk}
\author[]{CATERINA IDA ZEPPIERI}
\address[]{Angewandte Mathematik, WWU M\"unster, Germany}
\email[]{caterina.zeppieri@uni-muenster.de}

\begin{document}


\begin{abstract}
We study the $\Gamma$-convergence of sequences of free-discontinuity functionals depending on vector-valued functions $u$
which can be discontinuous across hypersurfaces whose shape and
location are not known a priori. 
The main novelty of our result is that we work under very general assumptions on the integrands which, in particular, 
are not required to be periodic in the space variable.
Further, we consider the case of surface integrands which are not bounded 
from below by the amplitude of the jump of~$u$. 

We obtain three main results: compactness with respect to $\Gamma$-convergence, representation of the $\Gamma$-limit in an integral form 
and identification of its integrands, and homogenisation formulas without periodicity assumptions. 
In particular, the classical case of periodic homogenisation follows as a by-product of our analysis.
Moreover, our result covers also the case of stochastic homogenisation, 
as we will show in a forthcoming paper.
\end{abstract}

\maketitle

{\small

\bigskip

\section{Introduction}

In this paper we study the $\Gamma$-convergence, as $k\to+\infty$, of sequences of free-discontinuity functionals of the form
\begin{equation}\label{intro functionals}
E_k(u,A)=\int_A f_k(x, \nabla u(x)) \dx+ \int_{S_u\cap A}g_k(x,[u](x),\nu_u(x))d \mathcal{H}^{n-1}(x),
\end{equation}
where $A\subset \R^n$ is a bounded open set, $u\colon A\to\R^m$ is a generalised special function of bounded variation,  $\nabla u$ is its approximate gradient, $S_u$ is the jump set of $u$ and $[u]$ is its jump on $S_u$, while $\nu_u$ is the approximate normal to $S_u$ and $\mathcal{H}^{n-1}$ denotes the $(n-1)$-dimensional Hausdorff measure. 

Functionals of the form \eqref{intro functionals} appear naturally in the study of quasistatic crack growth in nonlinear elasticity (see \cite{FM, DM-T, Cha, FL, DMFT} and the monograph \cite{Bou-Fra-Mar}), and represent the energy associated to a deformation $u$ of an elastic body with cracks. The parameter $k$ may have different meanings: it may represent the scale of a regularisation of the energy, the size of a microstructure, or the ratio of the contrasting values of the mechanical response of the material in different parts of the body. For example, for a high-contrast medium $f_k$ and $g_k$ represent the strength and the toughness of the material, respectively, and may have a very different behaviour in each component. In this case taking the limit of $E_k$, in the sense of $\Gamma$-convergence, corresponds to computing the \textit{effective} energy of the material. 

\subsection{A brief literature review.} The classical case of \textit{periodic} homogenisation, namely where $f_k(x,\xi) = f(x/\e_k,\xi)$, $g_k(x,\zeta,\nu) = g(x/\e_k,\zeta,\nu)$, with 
$f$ and $g$ periodic in the first variable, and $\e_k\to 0+$ as $k\to +\infty$, is well studied. In this case, the limit behaviour of $E_k$ is also 
of free-discontinuity type, under mild assumptions on $f$ and $g$. Moreover, assuming that 
\begin{equation}\label{intro estimates}
c_1|\xi|^p\leq f(x,\xi)\le c_2(1+|\xi|^p) \quad \textrm{and} \quad c_4(1+|\zeta|)\le g(x,\zeta,\nu)\le c_5(1+|\zeta|),
\end{equation}
for $p>1$ and constants $0<c_1\le c_2, c_4\le c_5<+\infty$, it was proved in \cite{BDfV} that the $\Gamma$-limit of $E_k$ with respect to $L^1$-convergence is obtained by the simple superposition of the limit  behaviours of its volume and surface parts. Note that in \cite{BDfV} it is natural to study the $\Gamma$-convergence of $E_k$ in $L^1$ since the assumptions \eqref{intro estimates} on $f$ and $g$ guarantee that sequences $(u_k)$ with bounded energy $E_k$ are bounded in $BV$. 

Under coercivity conditions weaker than \eqref{intro estimates} for $f$ and $g$, however, it is not guaranteed that the volume and surface terms do ``not mix'' in the limit. For example, if $f$ and $g$ satisfy ``degenerate'' coercivity conditions, the two terms in $E_k$ can stay separate (see \cite{BF, CS, FGP}), or interact (see \cite{BDM, BLZ, DMZ, PSZ18, LS1, LS2}) and produce rather complex limit effects.

\smallskip

The case of general functionals $E_k$ as in \eqref{intro functionals} with \textit{non-periodic} integrands $f_k$ and $g_k$ is less studied. In the work \cite{GP}, the authors consider the case of $u$ scalar ($m=1$) and assume that $f_k$ and $g_k$ satisfy
\begin{equation}\label{intro estimates 2}
c_1|\xi|^p\leq f_k(x,\xi)\le c_2(1+|\xi|^p) \quad \textrm{and} \quad  c_4\le g_k(x,\nu)\le c_5,
\end{equation}
for suitable, $k$-independent constants $0<c_1\le c_2, c_4\le c_5<+\infty$. Note that $g_k$ in \eqref{intro estimates 2} is independent of $\zeta$, which, together with the restriction $m=1$, introduces lots of simplifications in the analysis. In particular, these simplifications guarantee that sequences $(u_k)$ with bounded energy $E_k$ are bounded in $BV$, up to a truncation, and hence also in \cite{GP} it is natural to study the $\Gamma$-convergence of $E_k$ in $L^1$. By using the abstract integral representation result in \cite{BFLM}, it is shown in \cite{GP} that the $\Gamma$-limit of $E_k$ is a free-discontinuity functional of the same type, and that also in this case no interaction occurs between the bulk and the surface part of the functionals in the  $\Gamma$-convergence process.

Therefore, the volume and surface terms decouple in the limit both in the periodic case - for vector-valued $u$ and with dependence of the surface densities on $[u]$, under strong coercivity assumptions - and in the non-periodic case - for scalar $u$ and with no dependence on $[u]$. This raises the question of determining general assumptions for $f_k$ and $g_k$ guaranteeing the decoupling.

\subsection{The main result: Method of proof and comparison with previous works.} In this paper we study the $\Gamma$-convergence of \eqref{intro functionals} in the vector-valued case ($m\ge 1$) \textit{without any periodicity assumptions}, and under the assumption that $(f_k) \subset \mathcal{F}$ (see $(f1)$-$(f4)$ in Definition \ref{volume integrands}) 
and $(g_k) \subset \mathcal{G}$ (see $(g1)$-$(g7)$ in Definition \ref{volume integrands}). In particular, we assume that $f_k$ and $g_k$ satisfy \textit{the more general growth conditions} 
\begin{equation}\label{intro estimates 3}
c_1|\xi|^p\leq f_k(x,\xi)\le c_2(1+|\xi|^p) \quad \textrm{and} \quad c_4 \le g_k(x,\zeta,\nu)\le c_5(1+|\zeta|),
\end{equation}
which include both \eqref{intro estimates} and \eqref{intro estimates 2}.

\medskip

We prove three main results. The first one, Theorem~\ref{thm:joint}, 
is a compactness result with respect to $\Gamma$-convergence. Namely, we show that for every sequence $(E_k)$
with $(f_k) \subset \mathcal{F}$ and $(g_k) \subset \mathcal{G}$ there exists a subsequence, not relabelled, 
such that,  for every bounded open set 
$A\subset\R^n$, $E_k(\cdot,A)$ $\Gamma$-converges to a functional $E_\infty(\cdot,A)$, which can be written in the form \eqref{intro functionals} for suitable functions $f_\infty \in \mathcal{F}$ and~$g_\infty \in \mathcal{G}$. In the proof of Theorem~\ref{thm:joint} we rely on the compactness by $\Gamma$-convergence in \cite{BDfV} and on the integral representation in \cite{BFLM}. These results, however, are not applied directly to the functionals $E_k$, due to  the weak coercivity of $g_k$ (see \eqref{intro estimates 3}), but to perturbed functionals $E_k(u,A) + \e\int_{S_u\cap A}|[u]|d\mathcal{H}^{n-1}$, for $\e>0$. Dealing with perturbed functionals introduces some technicalities, which are resolved in Lemma~\ref{estimate truncations}, Lemma~\ref{truncations Gamma} and Theorem~\ref{perturb}. These technical results are therefore not needed if $g_k$ satisfies the stronger lower bound in \eqref{intro estimates}.
\smallskip

The second result, Theorem~\ref{G-convE},  
identifies the $\Gamma$-limit $E_\infty(\cdot,A)$. That is, it provides a connection between the functions $f_k$ and $g_k$, used to define $E_k$, and the functions $f_\infty$ and $g_\infty$, which appear in the integral representation of $E_\infty$. More precisely, set 
\begin{equation} \label{i:min-f}
m_{F_k}^{1,p}(\ell_\xi, Q_\rho(x)):= \inf \int_{Q_\rho(x)} f_k(y,\nabla u(y))dy,
\end{equation}
where the infimum is taken among the functions $u \in W^{1,p}(Q_\rho(x),\R^m)$ with $u (y) = \xi \cdot y$ 
near $\partial Q_\rho(x)$, and $Q_\rho(x) := x + (-\rho/2, \rho/2)^n$, and
\begin{equation}\label{i:min-g}
m_{G_k}^{\mathrm{pc}}(u_{x,\zeta,\nu}, Q^{\nu}_\rho(x)):= \inf \int_{S_u\cap Q^\nu_\rho(x)} g_k(y,[u](y),\nu_u(y)) d \mathcal{H}^{n-1}(y),
\end{equation}
where the cube $Q^\nu_\rho(x)$ is a suitable rotation of $Q_\rho(x)$ (see item (l) of Section~\ref{Notation}), 
and the infimum is taken among all the functions $u \in SBV(Q^\nu_\rho(x),\R^m)$
with $\nabla u = 0$ $\mathcal{L}^n$-a.e.\ in  $Q^\nu_\rho(x)$ and that near $\partial Q^\nu_\rho(x)$ agree  with the pure-jump function $u_{x,\zeta,\nu}$
(see item (n) of Section~\ref{Notation}).

Roughly speaking, we show that if 
\begin{equation}\label{lsup=linf_f}
\limsup_{\rho \to 0+} \liminf_{k \to + \infty} \frac{
m_{F_k}^{1,p}(\ell_\xi, Q_\rho(x))}{\rho^n} = \limsup_{\rho \to 0+} \limsup_{k \to + \infty} \frac{
m_{F_k}^{1,p}(\ell_\xi, Q_\rho(x))}{\rho^n},
\end{equation}
and 
\begin{equation}\label{lsup=linf_g}
\limsup_{\rho \to 0+} \liminf_{k \to + \infty} \frac{ 
m_{G_k}^{\mathrm{pc}}(u_{x,\zeta,\nu}, Q^{\nu}_\rho(x))}{\rho^{n-1}} = \limsup_{\rho \to 0+} \limsup_{k \to + \infty} \frac{ 
m_{G_k}^{\mathrm{pc}}(u_{x,\zeta,\nu}, Q^{\nu}_\rho(x))}{\rho^{n-1}},
\end{equation}
then $E_k$ $\Gamma$-converges to $E_\infty$, the limit volume density $f_\infty(x,\xi)$ coincides with the common value in \eqref{lsup=linf_f}, and the limit surface energy 
$g_\infty(x,\zeta,\nu)$
coincides with the common value in \eqref{lsup=linf_g}.

This result shows, in particular, that the problems for the volume and surface integrals are decoupled in the limit\ie $f_\infty$ depends only on the sequence $(f_k)$, while $g_\infty$ depends only on the sequence $(g_k)$. Moreover, 
the equalities \eqref{lsup=linf_f} and \eqref{lsup=linf_g} are not only sufficient for $\Gamma$-convergence, but also, in some sense, necessary: Theorem \ref{G-convE2} states that if $E_k$ $\Gamma$-converges to $E_\infty$, 
then the limit densities can be characterised by formulas as in \eqref{lsup=linf_f} and \eqref{lsup=linf_g}, but where the limits in $k$ are taken along a subsequence.
\smallskip

The third result (Theorem~\ref{Homogenisation}) deals with the case of (non-periodic) homogenisation, that is $f_k(x,\xi)=f(x/\e_k, \xi)$ and $g_k(x,\zeta,\nu)=g(x/\e_k,\zeta,\nu)$ for a sequence $\e_k\to0+$
as $k \to + \infty$.  In this case, for given $x$, $\xi$, $\zeta$, and $\nu$, a natural change of variables in \eqref{i:min-f} and \eqref{i:min-g} leads to consider, for every $r > 0$, the two rescaled minimisation problems
\begin{equation}\label{i:f-hom}
\frac{1}{r^n}\inf \bigg\{\int_{Q_r(rx)} f(y,\nabla u(y))dy \colon u \in W^{1,p}(Q_r(rx),\R^m),\; 
u (y)= \xi \cdot y \textrm{ near } \partial Q_r(rx) \bigg\},
\end{equation}
and 
\begin{equation} \label{i:g-hom}
\frac{1}{r^{n-1}}\inf \int_{S_u\cap Q^\nu_r(rx)} g(y,[u](y),\nu_u(y))d \mathcal{H}^{n-1}(y).
\end{equation}
In the last formula, the infimum is taken among all the functions $u \in SBV(Q^\nu_r(rx),\R^m)$
with $\nabla u = 0$ $\mathcal{L}^n$-a.e.\ in $Q^\nu_r(rx)$ and that near $\partial Q^\nu_r(rx)$ agree with the pure jump function $u=u_{rx,\zeta,\nu}$
(see item (n) of Section~\ref{Notation}). Assume that the limits as $r\to 0+$ of the expressions in \eqref{i:f-hom} and \eqref{i:g-hom} 
exist and are independent of $x$, and denote them by $f_{\mathrm{hom}}(\xi)$ 
and $g_{\mathrm{hom}}(\zeta,\nu)$, respectively (see \eqref{def fhom} and \eqref{def ghom}). 
Then, we prove that for every bounded open set $A\subset\R^n$ the sequence $E_k(\cdot,A)$ 
with integrands $f(x/\e_k, \xi)$ and $g(x/\e_k,\zeta,\nu)$  $\Gamma$-converges to the functional $E_{\mathrm{hom}}(\cdot,A)$ with integrands $f_{\mathrm{hom}}(\xi)$ and $g_{\mathrm{hom}}(\zeta,\nu)$.

In particular, we recover the case where $f(x,\xi)$ and $g(x,\zeta,\nu)$ are periodic with respect to $x$, which was previously studied in \cite{BDfV} assuming \eqref{intro estimates} for $g$. In the forthcoming paper \cite{ergodic} we shall prove that, under our more general assumptions \eqref{intro estimates 2}, the existence of these limits and their independence of $x$ can be proved even in the more general context of stochastic homogenisation. Therefore Theorem~\ref{Homogenisation} of the present paper will be a key ingredient in the proof of the results on stochastic homogenisation for free-discontinuity problems.

\smallskip

In this paper, unlike in \cite{BDfV} and \cite{GP}, the natural topology for the $\Gamma$-convergence of $E_k$ is not $L^1$. Indeed, unlike \eqref{intro estimates}, assumption \eqref{intro estimates 3} 
does not guarantee a bound in $BV(A,\R^m)$ for sequences $(u_k)$ with bounded energy $E_k(u_k,A)$. 
Moreover, unlike in the scalar case considered in \cite{GP}, 
in the vector-valued case an estimate for $\|u_k\|_{L^\infty(A,\R^m)}$ cannot be easily obtained by a standard truncation procedure. For these reasons, in our setting sequences $(u_k)$ with bounded energy $E_k(u_k,A)$ are, in general, 
not relatively compact in $L^1(A,\R^m)$. Therefore, we study the $\Gamma$-convergence in the larger space $L^0(A,\R^m)$ of all $\mathcal{L}^n$-measurable functions $u\colon A\to \R^m$, endowed with the metrisable topology of convergence in measure.
This is the natural choice of convergence in our case: using compactness theorems for free-discontinuity functionals, it is indeed possible to prove that sequences $(u_k)$ with equi-bounded energy $E_k(u_k,A)$ are relatively compact in $L^0(A,\R^m)$, under a very weak integral bound on $(u_k)$. 
Therefore, $\Gamma$-convergence of $(E_k(\cdot,A))$ in $L^0(A,\R^m)$ implies convergence of the solutions of some associated minimisation problems obtained, for instance, by adding a lower order term to $E_k$ (see Corollary~\ref{conv min}).

\medskip

\subsection{Outline of the paper.} The paper is organised as follows. In Section~\ref{Notation} we fix the notation and give the references for the background material used in the paper. In Section~\ref{Statement} we list the general hypotheses on the integrands $f_k$ and $g_k$ and state our main results. We also prove that the result on homogenisation follows, through a change of variables, from the result on the identification of the $\Gamma$-limit. 

In Section~\ref{compactness for perturbed} we prove a compactness theorem for the perturbed functionals obtained by adding to $E_k(u,A)$ the regularising term $\e\int_{S_u\cap A}|[u]|d\mathcal{H}^{n-1}$, which allows us to use the results of \cite{BDfV}. This section contains also some technical lemmas on smooth truncations that are used throughout the paper. 

In Section~\ref{Gamma free} we begin the proof of Theorem~\ref{thm:joint}, 
which gives the compactness of sequences of functionals of the form \eqref{intro functionals} with respect to $\Gamma$-convergence. The main tool is the analysis of the limit as $\e\to0+$ 
of the $\Gamma$-limits of the perturbed functionals of Section~\ref{compactness for perturbed}.
The conclusion of the proof is based on Theorem~\ref{6.3}, where 
the integrands of the functional obtained in this way are
compared with \eqref{i:min-f} and \eqref{i:min-g}. The proof of this theorem is very technical and is given in Sections~\ref{proof volume} and~\ref{estimate surface}.

In Section~\ref{appl conv min} we prove the identification result
for the $\Gamma$-limit (Theorem~\ref{G-convE}) 
using Theorem~\ref{6.3}. Moreover we show that, for some minimisation problems involving an $L^p(A,\R^m)$-perturbation of the functionals \eqref{intro functionals}, $\Gamma$-convergence in $L^0(A,\R^m)$ implies convergence of the minimum values and, for a subsequence, convergence in $L^p(A,\R^m)$ of the minimum points.

In Sections~\ref{proof volume} and~\ref{estimate surface} we prove the statements of Theorem~\ref{6.3} concerning the volume and the surface integrals, respectively. 

The final section is an appendix which collects some technical results used in the paper.

\section{Preliminaries and notation}\label{Notation}

In this section we give a brief account of the mathematical tools that will be needed in the paper.

For the general notions on $BV$, $SBV$, and $GSBV$ functions and their fine properties we refer to \cite{AFP} (see also \cite{EvansGariepy,Giusti}). 
For $u \in BV$, $Du$ and $D^s u$ denote the distributional derivative of $u$
and its singular part with respect to the Lebesgue measure, respectively, 
while $\nabla u$ stands for the density of the absolutely continuous part of $Du$ 
with respect to the Lebesgue measure.  $\nabla u$ coincides with the approximate gradient of $u$, 
which makes sense also for $u\in GSBV$.
Moreover, $S_u$ denotes the set of approximate discontinuity points of $u$, and $\nu_u$ the measure theoretic normal to $S_u$. The symbols $u^\pm$ denote the one-sided approximate limits of $u$ at a point of $S_u$, from the side of $\pm\nu_u$.

For the general theory of $\Gamma$-convergence we refer to the monograph \cite{DM93}. Other results on this subject can be found in \cite{Braides} and \cite{BDf}.

We introduce now some notation that will be used throughout the paper.
\begin{itemize}
\item[(a)] $m$ and $n$ are fixed positive integers, $\R$ is the set of real numbers, and $\R^m_0:=\R^m\setminus\{0\}$.
\item[(b)]$\Sph^{n-1}:=\{x=(x_1,\dots,x_n)\in \mathbb{R}^{n}: x_1^2+\dots+x_n^2=1\}$ and
$\widehat\Sph^{n-1}_\pm:=\{x\in \Sph^{n-1}: \pm x_{i(x)}>0\}$,
where $i(x)$ is the largest $i\in\{1,\dots,n\}$ such that $x_i\neq 0$.
\item[(c)] $\mathcal{L}^n$ denotes the Lebesgue measure on $\R^n$ and $\mathcal{H}^{n-1}$ the $(n-1)$-dimensional Hausdorff measure on $\R^n$. 
\item[(d)] $\A$ denotes the collection of all bounded open subsets of $\R^n$; if $A$, $B\in \A$, by
$A\subset\subset B$ we mean that $A$ is relatively compact in $B$.
\item[(e)] For $u\in GSBV(A,\R^m)$, with  $A \in \A$, the jump of $u$ across $S_u$ is defined by 
$[u]:=u^+-u^-$.
\item[(f)] For $A \in \A$ we define
$$
SBV_{\mathrm{pc}}(A,\R^m):=\{u\in SBV(A,\R^m): \nabla u=0 \,\, \mathcal{L}^{n}\textrm{-a.e.}, \,\mathcal{H}^{n-1}(S_u)<+\infty\};
$$
it is known (see \cite[Theorem 4.23]{AFP}) that every $u$ in $SBV_{\mathrm{pc}}(A,\R^m)\cap L^\infty(A,\R^m)$ is piecewise constant in the sense of \cite[Definition 4.21]{AFP}, namely 
there exists a Caccioppoli partition $(E_i)$ of $A$ such that $u$ is constant $\mathcal{L}^n$-a.e.\ in each set $E_i$. We note that same result holds for $u\in SBV_{\mathrm{pc}}(A,\R^m)$, however this property will never be used in the paper.
\item[(g)] For $A \in \A$ and $p>1$ we define
$$
SBV^p(A,\R^m):=\{u\in SBV(A,\R^m): \nabla u \in L^p(A, \R^{m{\times}n}), \,\mathcal{H}^{n-1}(S_u)<+\infty\}.
$$
\item[(h)] For $A \in \A$ and $p>1$ we define
$$
GSBV^p(A,\R^m):=\{u\in GSBV(A,\R^m): \nabla u  \in L^p(A ,\R^{m{\times}n}), \,\mathcal{H}^{n-1}(S_u)<+\infty\};
$$
it is known that $GSBV^p(A,\R^m)$ is a vector space and that $\psi(u)\in SBV^p(A,\R^m)\cap L^\infty(A,\R^m)$ for every $u\in GSBV^p(A,\R^m)$ and for every $\psi\in C^1_c(\R^m,\R^m)$ (see, e.g., \cite[page 172]{DMFT}).
\item[(i)] For every $\mathcal{L}^n$-measurable set $A \subset \R^n$ let 
$L^0(A,\R^m)$ be the space of all $\mathcal{L}^n$-measurable functions $u\colon A\to \R^m$,
endowed with the topology of convergence in measure on bounded subsets of $A$;
we observe that this topology is metrisable and separable.
\item[(j)] For $x\in \mathbb{R}^n$ and $\rho>0$ we define 
\begin{eqnarray*}
& B_\rho(x):=\{y\in \R^n: \,|y-x|<\rho\},
\\
& Q_\rho(x):= \{y\in \R^n: \,|(y-x)\cdot e_i | < \rho/2 \quad\hbox{for } i=1,\dots,n \},
\end{eqnarray*}
where $|\cdot|$ is the Euclidean norm in $\R^n$, $e_1,\dots, e_n$ is the canonical basis of $\R^n$, and $\cdot$ denotes the Euclidean scalar product; we omit the subscript $\rho$ when $\rho=1$ ($|\cdot|$ denotes the absolute value in $\R$ or the Euclidean norm in $\R^n$, $\R^m$, 
or $\R^{m{\times}n}$, depending on the context).

\item[(k)] For every $\nu\in \Sph^{n-1}$ let $R_\nu$ be an orthogonal $n{\times} n$ matrix such that $R_\nu e_n=\nu$; we assume that the restrictions of the function $\nu\mapsto R_\nu$ to the sets $\widehat\Sph^{n-1}_\pm$ defined in (b) are continuous and that $R_{-\nu}Q(0)=R_\nu Q(0)$ for every $\nu\in \Sph^{n-1}$;  a map $\nu\mapsto R_\nu$ satisfying these properties is provided in Example \ref{Rnu}  in the Appendix.

\item[(l)] For $x\in \mathbb{R}^n$, $\rho>0$, and $\nu\in \Sph^{n-1}$ we set
$$
Q^\nu_\rho(x):= 
R_\nu Q_\rho(0) + x;
$$
we omit the subscript $\rho$ when $\rho=1$.
\item[(m)] For $\xi\in \R^{m{\times}n}$, the linear function from $\R^n$ to $\R^m$ with gradient $\xi$ is denoted by $\ell_\xi$\ie $\ell_\xi(x):=\xi x$, where $x$ is considered as an $n{\times}1$ matrix.
\item[(n)] For $x\in \mathbb{R}^n$, $\zeta\in \R^m_0$, and $\nu \in \Sph^{n-1}$ we define the function $u_{x,\zeta,\nu}$ as
\begin{equation*}
u_{x,\zeta,\nu}(y):=\begin{cases}
\zeta \quad & \mbox{if} \,\, (y-x)\cdot \nu \geq 0,\\
0 \quad & \mbox{if} \,\, (y-x)\cdot \nu < 0.
\end{cases}
\end{equation*}
\item[(o)] For  $x\in \mathbb{R}^n$ and $\nu \in \Sph^{n-1}$, we set
$$
\Pi^{\nu}_0:= \{y\in \R^n: y\cdot \nu = 0\}\quad\text{and}\quad\Pi^{\nu}_{x}:= \{y\in \R^n: (y-x)\cdot \nu = 0\}.
$$ 
\end{itemize}

\section{Statement of the main results} \label{Statement}
\noindent Throughout the paper we fix six constants $p, c_1,\dots, c_5$, with  $1<p<+\infty$, $0<c_1\leq c_2<+\infty$, $1\le c_3<+\infty$, and  $0<c_4\leq c_5<+\infty$, and
two nondecreasing continuous functions $\sigma_1$, $\sigma_2\colon  [0,+\infty) \to [0,+\infty)$ such that $\sigma_1(0)=\sigma_2(0)=0$.

\begin{defn}[Volume and surface  integrands]\label{volume integrands} Let $\mathcal{F}=\mathcal{F}(p,c_1,c_2,\sigma_1)$ be the collection of all functions $f\colon \R^n{\times} \R^{m{\times}n}\to [0,+\infty)$ satisfying the following conditions:
\begin{itemize}
\item[$(f1)$] (measurability) $f$ is Borel measurable on $\R^n{\times} \R^{m{\times}n}$;
\item[$(f2)$] (continuity in $\xi$) for every $x \in \R^n$ we have
\begin{equation*}
|f(x,\xi_1)-f(x,\xi_2)| \leq \sigma_1(|\xi_1-\xi_2|)\big(1+f(x,\xi_1)+f(x,\xi_2)\big)
\end{equation*}
for every $\xi_1$, $\xi_2 \in \R^{m{\times}n}$;
\item[$(f3)$] (lower bound) for every $x \in \R^n$ and every $\xi \in \R^{m{\times}n}$
$$
c_1 |\xi|^p \leq f(x,\xi);
$$
\item[$(f4)$] (upper bound) for every $x \in \R^n$ and every $\xi \in \R^{m{\times}n}$
$$
f(x,\xi) \leq c_2(1+|\xi|^p).
$$
\end{itemize}

Let $\mathcal{G}=\mathcal{G}(c_3, c_4,c_5, \sigma_2)$ be the collection of all functions 
$g\colon \R^n{\times}\R^m_0{\times} \Sph^{n-1} \to [0,+\infty)$ satisfying the following conditions:
\begin{itemize}
\item[$(g1)$] (measurability) $g$ is Borel measurable on $\R^n{\times}\R^m_0{\times} \Sph^{n-1}$;
\item[$(g2)$] (continuity in $\zeta$) for every $x\in \R^n$ and every $\nu \in \Sph^{n-1}$ we have
\begin{equation*}
|g(x,\zeta_2,\nu)-g(x,\zeta_1,\nu)|\leq \sigma_2(|\zeta_1-\zeta_2|)\big(g(x,\zeta_1,\nu)+g(x,\zeta_2,\nu)\big)
\end{equation*}
for every $\zeta_1$, $\zeta_2\in \R^m_0$;
\item[$(g3)$] (estimate for $|\zeta_1|\le|\zeta_2|$) for every $x\in \R^n$ and every $\nu \in \Sph^{n-1}$
we have
$$
g(x,\zeta_1,\nu) \leq c_3 \,g(x,\zeta_2,\nu)
$$
for every $\zeta_1$, $\zeta_2 \in \R^m_0$ with $|\zeta_1|\le |\zeta_2|$;
\item[$(g4)$] (estimate for $c_3|\zeta_1|\le|\zeta_2|$) for every $x\in \R^n$ and every $\nu \in \Sph^{n-1}$
we have
\begin{equation*}
g(x,\zeta_1,\nu) \leq \,g(x,\zeta_2,\nu)
\end{equation*}
for every $\zeta_1$, $\zeta_2\in \R^m_0$ with $c_3|\zeta_1|\le|\zeta_2|$;
\item[$(g5)$] (lower bound) for every $x\in \R^n$, $\zeta\in \R^m_0$, and $\nu \in \Sph^{n-1}$
\begin{equation*}
c_4 \leq g(x,\zeta,\nu);
\end{equation*}
\item[$(g6)$] (upper bound) for every $x\in \R^n$, $\zeta\in \R^m_0$, and $\nu \in \Sph^{n-1}$
\begin{equation*}
g(x,\zeta,\nu) \leq c_5 (1+|\zeta|);
\end{equation*}
\item[$(g7)$] (symmetry) for every $x\in \R^n$, $\zeta\in \R^m_0$, and $\nu \in \Sph^{n-1}$
$$g(x,\zeta,\nu) = g(x,-\zeta,-\nu).$$ 
\end{itemize}
\end{defn}

\begin{rem}[Assumptions $(g3)$ and $(g4)$] Let $g\colon \R^n{\times}\R^m_0{\times} \Sph^{n-1} \to [0,+\infty)$ be a function satisfying the following ``monotonicity" condition: for every $x\in \R^n$ and every $\nu \in \Sph^{n-1}$
$$
g(x,\zeta_1,\nu) \leq  g(x,\zeta_2,\nu)
$$
for every $\zeta_1$, $\zeta_2 \in \R^m_0$ with $|\zeta_1|\le |\zeta_2|$; then it is immediate to verify that $g$ satisfies $(g3)$ and $(g4)$.

On the other hand $(g3)$ and $(g4)$ are weaker than monotonicity in $|\zeta|$. For instance, the function $g(x,\zeta,\nu):= \hat{g}(|\zeta|)$, with $\hat g:[0,+\infty)\to [0,+\infty)$ given by
$$
\hat g(t) = 
\begin{cases}
t \quad &\textrm{if } t\in [0,1],\\
\in \left[\frac{t}{c_3},1\right] \quad &\textrm{if } t\in [1,c_3],\\
\frac{t}{c_3} \quad &\textrm{if } t \geq c_3,
\end{cases}
$$
satisfies $(g3)$ and $(g4)$, but its behaviour in $[1,c_3]$ can be chosen quite freely, in particular it can be nonmonotone.
\end{rem}

\begin{rem}
We remark that assumptions $(g3)$ and $(g4)$ on the surface integrand $g$ will be crucial to prove that the functional $E$ defined in \eqref{en:vs1-loc2} decreases by smooth truncations up to an error term (see \eqref{estimate for E} and the proof of Lemma \ref{estimate truncations}). We also notice that $(g3)$ and $(g4)$ could be omitted if assumption $(g5)$ were replaced by the stronger lower bound 
\begin{equation}\label{strong-bb}
c(1+|\zeta|)  \leq g(x,\zeta,\nu) \quad \text{for every } (x,\zeta,\nu) \in \R^n\times \R^m_0\times \mathbb{S}^{n-1}
\end{equation}
for some $c>0$ (see, e.g., the proof of \cite[Lemma 3.5]{BDfV}). However, a lower bound as in \eqref{strong-bb} would rule out, for instance, functionals of Mumford-Shah type, which we would like to cover in our analysis. For this reason we prefer to work under the weaker growth condition $(g5)$ on $g$ and under the additional ``monotonicity" assumptions $(g3)$ and $(g4)$.    
\end{rem}

Given $f\in \mathcal{F}$ and $g\in \mathcal{G}$, we consider the integral functionals 
$F$, $G$, $E \colon L^0(\R^n,\R^m){\times} \A \longrightarrow [0,+\infty]$ defined as
\begin{align}
F(u,A)&:=
\begin{cases}
\ds\int_A f(x, \nabla u) \dx & \text{if}\; u|_A\in GSBV^p(A,\R^m),\cr
+\infty & \text{otherwise in}\; L^0(\R^n,\R^m).
\end{cases} \label{Effe} \\
G(u,A) &:= 
\begin{cases}
\ds \int_{S_u\cap A}g(x,[u],\nu_u)d \mathcal{H}^{n-1} &\text{if} \; u|_A\in GSBV^p(A,\R^m),\cr
+\infty \quad & \mbox{otherwise in}\; L^0(\R^n,\R^m), \label{surf}
\end{cases} \\[7pt]
\vphantom{F_{g_{K_{K_K}}}}
E(u,A) &:= F(u,A) +  G(u,A). \label{en:vs1-loc2}
\end{align}
We also consider the integral functional $E^p \colon L^p_{\mathrm{loc}}(\R^n,\R^m){\times} \A \longrightarrow [0,+\infty]$, defined as the restriction of $E$ to $L^p_{\mathrm{loc}}(\R^n,\R^m){\times} \A$.

\begin{rem}\label{well defined}
Since $[u]$ is reversed when the orientation of $\nu_u$ is reversed, the functional $G$ is well defined thanks to $(g7)$.
\end{rem}

The following compactness theorem, with respect to $\Gamma$-convergence, 
is one of the main results of this paper.
\begin{thm}[Compactness for $\Gamma$-convergence]\label{thm:joint}
Let $(f_k)$ be a sequence in $\mathcal{F}$, let $(g_k)$ be a sequence in $\mathcal{G}$,
let $E_k\colon L^0(\R^n,\R^m){\times} \A \to [0,+\infty]$
be the integral functionals defined by \eqref{en:vs1-loc2} corresponding to
$f_k$ and $g_k$, and let $E^p_k\colon L^p_{\mathrm{loc}}(\R^n,\R^m){\times} \A \to [0,+\infty]$ be their restrictions to $L^p_{\mathrm{loc}}(\R^n,\R^m){\times} \A$.
Then there exist a subsequence, not relabelled, and two functions $f \in \mathcal{F}$ and $g \in \mathcal{G}$
such that for every $A\in \A$ 
\begin{align*}
E_k(\cdot,A)\ &\Gamma\hbox{-converges to }E(\cdot,A)\hbox{ in }L^0(\R^n,\R^m),
\\
E^p_k(\cdot,A)\ &\Gamma\hbox{-converges to }E^p(\cdot,A)\hbox{ in }L^p_{\mathrm{loc}}(\R^n,\R^m),
\end{align*}
where the integral functional 
$E \colon L^0(\R^n,\R^m){\times} \A \to [0,+\infty]$ is given by \eqref{en:vs1-loc2} and $E^p$ is its restriction to $L^p_{\mathrm{loc}}(\R^n,\R^m){\times} \A$.
\end{thm}

\begin{rem}[The strongly coercive case] Theorem \ref{thm:joint} above states that the class of free-discontinuity functionals $E_k$, with $f_k \in \mathcal F$ and $g_k\in \G$, is compact by $\Gamma$-convergence\ie up to a subsequence, $E_k$ $\Gamma$-converge to a free-discontinuity functional $E$ with integrands $f$ and $g$ satisfy $f\in \mathcal F$ and $g\in \mathcal G$ (and similarly for its restriction to $L^p_{\rm loc}$).
Note that if the surface integrands $g_k$ satisfy the stronger coercivity condition \eqref{strong-bb} uniformly in $k$, then the domain of the $\Gamma$-limit is $SBV^p$, and the existence of a free-discontinuity functional $E^p$ such that $E^p_k$ $\Gamma$-converges to $E^p$ is an easy consequence of \cite[Proposition 3.3]{BDfV} and \cite[Theorem 1]{BFLM}. The analysis carried out in \cite{BDfV, BFLM}, however, does not provide immediately the detailed information on the regularity of the limit integrands $f$ and $g$, which will be used later. Hence, even in the coercive case the closure of the class of functionals $E$ defined in \eqref{en:vs1-loc2} requires a proof.
\end{rem}


Let $X$ be a subspace of $L^0(\R^n,\R^m)$. For every $H \colon X{\times} \A \to [0,+\infty]$, $A\in\A$, and 
$w \in L^0(\R^n,\R^m)$, we set
\begin{align} \label{emme}
m^{1,p}_H (w,A) &:= \inf \left\{H(u,A):u\in X,\ u|_A \in W^{1,p}(A,\R^m),\; u=w \textrm{ near } \partial A \right\},
\\
\label{emmeG}
m^{\mathrm{pc}}_H (w,A) &:= \inf\left\{H(u,A):  u\in X,\ u|_A\in SBV_{\mathrm{pc}}(A,\R^m),\ u=w \textrm{ near }\partial A \right\},
\\
\label{mE0eta}
m_{H}(w,A) &:= \inf\left\{H(u,A) : u\in X,\ u|_A\in SBV^p(A,\R^m),\ u=w \textrm{ near }\partial A \right\},
\end{align}
with the standard convention $\inf\emptyset=+\infty$.
In all the formulas above, by ``$u=w$ near $\partial A$'' we mean that there exists a neighbourhood $U$ of $\partial A$ in $\R^n$ such that $u=w$ $\mathcal{L}^n$-a.e.\ in $U \cap A$.

Let $(f_k)$ be a sequence in $\mathcal{F}$ and let $(g_k)$ be a sequence in $\mathcal{G}$.
For every $k$, we consider the integral functionals $F_k,G_k, E_k\colon L^0(\R^n,\R^m){\times} \A \to [0,+\infty]$ defined by  \eqref{Effe}, \eqref{surf}, and \eqref{en:vs1-loc2} corresponding to $f_k$ and $g_k$.
For every $x \in \R^n$, $\xi \in\R^{m{\times}n}$, $\zeta \in \R_0^m$, and $\nu \in \mathbb{S}^{n-1}$
we define 
\begin{align}
f' (x, \xi) &: = \limsup_{\rho \to 0+} \liminf_{k \to + \infty} \frac{
m_{F_k}^{1,p}(\ell_\xi, Q_\rho(x))}{\rho^n} , \label{f'} \\
f'' (x, \xi) &: = \limsup_{\rho \to 0+} \limsup_{k \to + \infty} \frac{
m_{F_k}^{1,p}(\ell_\xi, Q_\rho(x))}{\rho^n} ,  \label{f''}\\
g' (x, \zeta ,  \nu) &: = \limsup_{\rho \to 0+} \liminf_{k \to + \infty} \frac{
m_{G_k}^{\mathrm{pc}}(u_{x,\zeta,\nu}, Q^{\nu}_\rho(x))}{\rho^{n-1}} ,  \label{g'}\\
g'' (x, \zeta ,  \nu) &: = \limsup_{\rho \to 0+} \limsup_{k \to + \infty} \frac{ 
m_{G_k}^{\mathrm{pc}}(u_{x,\zeta,\nu}, Q^{\nu}_\rho(x))}{\rho^{n-1}}. \label{g''}
\end{align}

\begin{rem}\label{f'' in F and g'' in G}
It turns out that $f'$, $f'' \in \mathcal{F}$ (see Lemma~\ref{f'' g' in F and G}), and $g'$, $g'' \in \mathcal{G}$  (see Lemma \ref{g' in G}).
\end{rem}

The second main result of this paper is the identification of the $\Gamma$-limit.

\begin{thm}[Identification of the $\Gamma$-limit] \label{G-convE}
Let $(f_k)$, $(g_k)$, $(E_k)$, and $(E^p_k)$ be as in Theorem~\ref{thm:joint},
let $f_\infty \in \mathcal{F}$ and $g_\infty \in \mathcal{G}$, let $E_\infty$ be defined as in \eqref{en:vs1-loc2}  with $f_\infty$ and $g_\infty$, and let $E^p_\infty$ be its restriction to $L^p_{\mathrm{loc}}(\R^n,\R^m){\times} \A$. Assume that
the following equalities are satisfied:
\begin{itemize}
\item[(a1)] for $\mathcal{L}^n$-a.e. $x \in \R^n$ we have
\begin{equation*}
f_\infty (x, \xi) = f' (x, \xi) = f'' (x, \xi)\quad\hbox{for every }\xi \in\R^{m{\times}n};
\end{equation*}
\item[(a2)] for every $A\in\A$, for every $u\in GSBV^p(A,\R^m)$, and for $\hs^{n-1}$-a.e. $x\in S_u$ we have
\begin{equation*}
g_\infty(x,[u](x),\nu_u(x))
=  g'(x,[u](x),\nu_u(x)) =  g''(x,[u](x),\nu_u(x)).
\end{equation*}
\end{itemize}
Then 
\begin{align}
E_k(\cdot,A)\ &\Gamma\hbox{-converges to }E_\infty (\cdot,A)\hbox{ in }L^0(\R^n,\R^m),
\label{Gamma L0}
\\
E^p_k(\cdot,A)\ &\Gamma\hbox{-converges to }E^p_\infty (\cdot,A)\hbox{ in }L^p_{\mathrm{loc}}(\R^n,\R^m),
\label{Gamma Lp}
\end{align}
for every $A\in \A$.
\end{thm} 

The next theorem is a sort of `vice-versa' of Theorem~\ref{G-convE}; Theorem~\ref{G-convE} and Theorem~\ref{G-convE2} together give an `almost equivalence' between the $\Gamma$-convergence of $E_k$ and the equalities (a1) and (a2). More precisely, we have the following result.

\begin{thm} \label{G-convE2}
Let $(f_k)$, $(g_k)$, and $(E_k)$ be as in Theorem~\ref{thm:joint},
let $f_\infty \in \mathcal{F}$ and $g_\infty \in \mathcal{G}$, and let $E_\infty$ be defined as in \eqref{en:vs1-loc2}  with $f_\infty$ and $g_\infty$. Assume that
\begin{align*}
E_k(\cdot,A)\ &\Gamma\hbox{-converges to }E_\infty (\cdot,A)\hbox{ in }L^0(\R^n,\R^m),
\end{align*}
for every $A\in \A$. Then there exists a subsequence $(k_j)$ such that the following equalities are satisfied:
\begin{itemize}
\item[($\widetilde{\textrm a1}$)] for $\mathcal{L}^n$-a.e. $x \in \R^n$ we have
\begin{equation*}
f_\infty (x, \xi) = \tilde{f}' (x, \xi) = \tilde{f}'' (x, \xi)\quad\hbox{for every }\xi \in\R^{m{\times}n};
\end{equation*}
\item[($\widetilde{\textrm a2}$)] for every $A\in\A$, for every $u\in GSBV^p(A,\R^m)$, and for $\hs^{n-1}$-a.e. $x\in S_u$ we have
\begin{equation*}
g_\infty(x,[u](x),\nu_u(x))
=  \tilde{g}'(x,[u](x),\nu_u(x)) =  \tilde{g}''(x,[u](x),\nu_u(x)),
\end{equation*}
\end{itemize}
where $\tilde{f}'$, $\tilde{f}''$, $\tilde{g}'$ and $\tilde{g}''$ are defined as in \eqref{f'}, \eqref{f''}, \eqref{g'} and \eqref{g''} respectively, for the subsequence $(k_j)$.
\end{thm} 

\begin{rem}
Theorem \ref{G-convE2} does not say that $f'=f''=f_\infty$ and $g'=g''=g_\infty$ for the original sequence. We only have 
$$
f' \leq \tilde{f}' =f_\infty= \tilde{f}'' \leq f'' \quad \text{ and } \quad 
g' \leq \tilde{g}' =g_\infty = \tilde{g}'' \leq g''.
$$
\end{rem}


The third main result of the paper concerns the case of homogenisation, where $f_k(x,\xi):=f(x/\e_k, \xi)$ and $g_k(x,\zeta,\nu):=g(x/\e_k,\zeta,\nu)$ for a sequence $\e_k\to0+$.

\begin{thm}[Homogenisation] \label{Homogenisation}
Let $f\in \mathcal{F}$ and $g\in \mathcal{G}$, and let $F$ and $G$ be the functionals defined as in \eqref{Effe} and \eqref{surf}, respectively. Assume that 
for every $x\in \R^n$, $\xi\in\R^{m{\times}n}$, $\zeta\in \R^m_0$, and $\nu\in \Sph^{n-1}$
the limits
\begin{align}\label{def fhom}
\lim_{r\to +\infty} \frac{m^{1,p}_{F}(\ell_\xi,Q_r(rx))}{r^n}&=:f_{\mathrm{hom}}(\xi),
\\
\label{def ghom}
\lim_{r\to +\infty} \frac{m^{\mathrm pc}_{G}(u_{r x,\zeta,\nu},Q^\nu_r(rx))}{r^{n-1}}&=:g_{\mathrm{hom}}(\zeta,\nu)
\end{align}
exist and are independent of $x$. Then $f_{\mathrm{hom}}\in \mathcal{F}$ and $g_{\mathrm{hom}}\in\mathcal{G}$.

Let $(\e_k)$ be a sequence of positive real numbers converging to $0$,  let $f_k$ and $g_k$ be defined by
\begin{equation}\nonumber
f_k(x,\xi):=f(x/\e_k, \xi)\quad\hbox{and}\quad g_k(x,\zeta,\nu):=g(x/\e_k,\zeta,\nu),
\end{equation}
let  $E_k$ be defined as in \eqref{en:vs1-loc2} with $f_k$ and $g_k$, let $E_{\mathrm{hom}}$ be defined as in \eqref{en:vs1-loc2} with $f_{\mathrm{hom}}$ and $g_{\mathrm{hom}}$, and let $E^p_k$ and $E^p_{\mathrm{hom}}$ be their restrictions to $L^p_{\mathrm{loc}}(\R^n,\R^m){\times} \A$. Then 
\begin{align*}
E_k(\cdot,A)\ &\Gamma\hbox{-converges to }E_{\mathrm{hom}} (\cdot,A)\hbox{ in }L^0(\R^n,\R^m),
\\
E^p_k(\cdot,A)\ &\Gamma\hbox{-converges to }E^p_{\mathrm{hom}} (\cdot,A)\hbox{ in }L^p_{\mathrm{loc}}(\R^n,\R^m),
\end{align*}
for every $A\in \A$.
\end{thm}

Arguing as in \cite{BDfV} (see also \cite{BDf} for the volume part) one can prove that \eqref{def fhom} and \eqref{def ghom} are always satisfied when
$f$ and $g$ are periodic of period $1$ with respect to the space coordinates $x_1, \dots, x_n$. We omit here the proof of this property, since in \cite{ergodic} we shall prove that \eqref{def fhom} and \eqref{def ghom} are satisfied almost surely under the natural assumptions of stochastic homogenisation, which include, in particular, the case of deterministic periodic homogenisation.

The complete proofs of Theorems \ref{thm:joint}  and \ref{G-convE} require several intermediate results which will be established in the next sections. Theorem \ref{Homogenisation} instead follows easily from Remark \ref{f'' in F and g'' in G} and from Theorem \ref{G-convE} by means of a natural change of variables, as we show below.

\begin{proof}[Proof of Theorem \ref{Homogenisation}] By Theorem~\ref{G-convE}
it is enough to show that
\begin{equation}\label{f'=f''=fhom}
f'(x,\xi)=f''(x,\xi)=f_{\mathrm{hom}}(\xi)\quad\hbox{and}\quad g'(x,\zeta,\nu)=g''(x,\zeta,\nu)=g_{\mathrm{hom}}(\zeta,\nu)
\end{equation}
 for every $x\in \R^n$, $\xi\in\R^{m{\times}n}$, $\zeta\in \R_0^m$, and $\nu\in \Sph^{n-1}$. Indeed, if these equalities are satisfied, then $f_{\mathrm{hom}}\in \mathcal{F}$ and $g_{\mathrm{hom}}\in\mathcal{G}$ by Remark \ref{f'' in F and g'' in G}, and the $\Gamma$-convergence follows from Theorem \ref{G-convE} applied with $f_\infty=f_{\mathrm{hom}}$ and  $g_\infty=g_{\mathrm{hom}}$.
 
 To prove the first equality in \eqref{f'=f''=fhom} we fix $x\in \R^n$, $\xi\in\R^{m{\times}n}$, $\rho>0$, and $k\in\N$. Given $u\in W^{1,p}(Q_\rho(x), \R^m)$, let $u_k\in W^{1,p}(Q_{\rho/\e_k}(x/\e_k), \R^m)$ be defined by $u_k(z)=u(\e_k z)/\e_k$ for every $z\in Q_{\rho/\e_k}(x/\e_k)$. By the change of variables $z=y/\e_k$ we obtain
 $F_k(u,Q_\rho(x))=\e_k^nF(u_k, Q_{\rho/\e_k}(x/\e_k))$. Since $u=\ell_\xi$ near $\partial Q_\rho(x)$
if and only if $u_k=\ell_\xi$ near
$\partial Q_{\rho/\e_k}(x/\e_k)$, we deduce that
$m^{1,p}_{F_k}(\ell_\xi, Q_\rho(x))\allowbreak=\e_k^n\, m^{1,p}_F(\ell_\xi, Q_{\rho/\e_k}(x/\e_k))=
(\rho^n/r_k^n)\, m^{1,p}_F(\ell_\xi, Q_{r_k}(r_k x/\rho))$, where $r_k:=\rho/\e_k$. By applying \eqref{def fhom} with $x$ replaced by $x/\rho$ we obtain
$$
\lim_{k\to+\infty}\frac1{\rho^n} m^{1,p}_{F_k}(\ell_\xi, Q_\rho(x))=f_{\mathrm{hom}}(\xi).
$$
By \eqref{f'} and \eqref{f''} this implies that $f'(x,\xi)=f''(x,\xi)=f_{\mathrm{hom}}(\xi)$.

To prove the second equality in \eqref{f'=f''=fhom} we fix $x\in \R^n$, $\zeta \in \R^m_0$, $\nu\in \Sph^{n-1}$, $\rho>0$, and $k\in\N$. Given $v\in SBV_{\mathrm{pc}}(Q^\nu_\rho(x), \R^m)$, let $v_k\in SBV_{\mathrm{pc}}(Q^\nu_{\rho/\e_k}(x/\e_k), \R^m)$ be defined by $v_k(z)=v(\e_k z)$ for every $z\in Q^\nu_{\rho/\e_k}(x/\e_k)$. Then $S_{v_k}=(1/\e_k)S_v$ and, thanks to $(g7)$, we may assume that $[v_k](z)=[v](\e_k z)$ for $\hs^{n-1}$-a.e.\ 
 $z\in S_{v_k}$.  By the change of variables $z=y/\e_k$ we obtain
 $G_k(v,Q^\nu_\rho(x))=\e_k^{n-1}G(v_k, Q^\nu_{\rho/\e_k}(x/\e_k))$. From the fact that $v=u_{x,\zeta,\nu}$ near $\partial Q^\nu_\rho(x)$
if and only if $v_k=u_{x/\e_k,\zeta,\nu}$ near
$\partial Q^\nu_{\rho/\e_k}(x/\e_k)$, we deduce that
$m^{\mathrm{pc}}_{G_k}(u_{x,\zeta,\nu}, Q^\nu_\rho(x))=\e_k^{n-1}\, m^{\mathrm{pc}}_G(u_{x/\e_k,\zeta,\nu}, Q^\nu_{\rho/\e_k}(x/\e_k))=
(\rho^{n-1}/r_k^{n-1})\, m^{\mathrm{pc}}_G(u_{r_k x/\rho,\zeta,\nu}, Q^\nu_{r_k}(r_k x/\rho))$, where $r_k:=\rho/\e_k$. By applying \eqref{def ghom} with $x$ replaced by $x/\rho$ we obtain
$$
\lim_{k\to+\infty}\frac1{\rho^{n-1}} m^{\mathrm{pc}}_{G_k}(u_{x,\zeta,\nu}, Q^\nu_\rho(x))=g_{\mathrm{hom}}(\zeta,\nu).
$$
By \eqref{g'} and \eqref{g''} this implies that
$g'(x,\zeta,\nu)=g''(x,\zeta,\nu)=g_{\mathrm{hom}}(\zeta,\nu)$.
 \end{proof}

\section{Compactness result for perturbed functionals}\label{compactness for perturbed}

In this section we prove a compactness result, Theorem~\ref{Gamma e}, for the perturbed functionals obtained by adding to $E_k^p(u,A)$ the regularising term $\e\int_{S_u\cap A}|[u]|d\mathcal{H}^{n-1}$, with $\e>0$. Theorem~\ref{Gamma e} will then be pivotal to prove our main compactness result, Theorem~\ref{thm:joint}.

\bigskip

In order to prove Theorem~\ref{Gamma e} we need some technical tools.

We start with a result (Lemma~\ref{estimate truncations}) establishing the existence of smooth truncations of $u$ by which the functionals $F$ and $E$ ``almost decrease''  (see \eqref{estimate for F} and \eqref{estimate for E} below). Similar truncation results can be found in \cite[proof of Proposition 2.6]{CDM} and \cite[Lemma 3.5]{BDfV}.

In what follows we use the shorthand
 $\{|u|>\lambda\}:=
\{x\in \R^n:|u(x)|>\lambda\}$, where $u\in 
L^0(\R^n,\R^m)$ and $\lambda>0$. 

\medskip

\noindent \textbf{Smooth truncations.} Let $\varphi\in C^\infty(\R)$ be fixed and such that $\varphi(t)=t$ for every $t\le 1$, $\varphi(t)=0$ for every $t\ge 3$, while $\varphi(t)\ge 0$ and $|\varphi'(t)|\le 1$ for every $t \geq 0$. We define $\psi\in C^\infty_c(\R^m,\R^m)$ by
\begin{equation*}
\psi(\zeta):=\begin{cases}
\varphi(|\zeta|)\zeta/|\zeta|&\text{if }\zeta\neq0,
\\
0&\text{if }\zeta=0.
\end{cases}
\end{equation*}
Then $\psi(\zeta)=\zeta$ for every $|\zeta|\le 1$, $\psi(\zeta)=0$ for every $|\zeta|\ge 3$, and $|\psi(\zeta)|\le 2$ for every $\zeta\in\R^m$.
Moreover for every $\eta$, $\tilde\eta\in \R^m$ we have 
$$
\partial_\eta\psi(\zeta){\cdot\,}\tilde\eta=(\zeta{\,\cdot\,}\eta) \, (\zeta{\,\cdot\,}\tilde\eta)\,\varphi'(|\zeta|)/|\zeta|^2+(\eta{\,\cdot\,}\tilde\eta) \,
\varphi(|\zeta|)/|\zeta|-(\zeta{\,\cdot\,}\eta)\,(\zeta{\,\cdot\,}\tilde\eta)\,\varphi(|\zeta|)/|\zeta|^3.
$$
Let $\eta^{||}$ and $\tilde\eta^{||}$ be the orthogonal projections of $\eta$ and $\tilde\eta$ onto the one-dimensional space generated by $\zeta$, and let $\eta^\perp$ and $\tilde\eta^\perp$ be the orthogonal projections of $\eta$ and $\tilde\eta$ onto the space orthogonal to $\zeta$. Then
\begin{align*}
\partial_\eta\psi(\zeta){\cdot\,}\tilde\eta&=(\eta^{||}{\,\cdot\,}\tilde\eta^{||}) \, \varphi'(|\zeta|)+(\eta{\,\cdot\,}\tilde\eta) \,
\varphi(|\zeta|)/|\zeta|-(\eta^{||}{\,\cdot\,}\tilde\eta^{||})\,\varphi(|\zeta|)/|\zeta|
\\
&=(\eta^{||}{\,\cdot\,}\tilde\eta^{||}) \, \varphi'(|\zeta|)+(\eta^\perp{\,\cdot\,}\tilde\eta^\perp) \,
\varphi(|\zeta|)/|\zeta|.
\end{align*}
Since $|\varphi'(t)|\le 1$ and $0\le \varphi(t)/t\le 1$ for every $t \in \mathbb{R}$, we obtain that
$$
\partial_\eta\psi(\zeta){\cdot\,}\tilde\eta\le |\eta^{||}{\,\cdot\,}\tilde\eta^{||}|+
|\eta^\perp{\,\cdot\,}\tilde\eta^\perp|\le |\eta|\,|\tilde\eta|.
$$
Since $\tilde\eta$ is arbitrary, this implies that
$|\partial_\eta\psi(\zeta)|\le |\eta|$ for every $\eta\in \R^m$. By the mean value theorem this inequality gives
$
|\psi(\zeta_2)-\psi(\zeta_1)|\le |\zeta_2-\zeta_1|
$
for every $\zeta_1$, $\zeta_2\in\R^m$. 

For every $\lambda>0$ we set
\begin{equation}\label{psi lambda}
\psi^\lambda(\zeta):=\lambda\psi(\zeta/\lambda).
\end{equation}
Then $\psi^\lambda \in C^\infty_c(\R^m,\R^m)$ and
\begin{align}
\psi^\lambda(\zeta)&=\zeta\quad \hbox{for every} \; \zeta\in \R^m \colon |\zeta|\le \lambda,
\label{psi=zeta}
\\
|\psi^\lambda(\zeta)|&\le 2 \lambda\quad \hbox{for every}\; \zeta\in\R^m,
\label{psi<=2lambda}
\\
\psi^\lambda(\zeta)&=0\quad \hbox{for every}\; \zeta\in \R^m \colon |\zeta|\ge 3\lambda,
\label{psi=0}
\\
|\psi^\lambda(\zeta_2)-\psi^\lambda(\zeta_1)|&\le |\zeta_2-\zeta_1|\quad \hbox{for every}\; \zeta_1, \zeta_2\in\R^m.
\label{lip 1}
\end{align}
From \eqref{psi=zeta} and \eqref{lip 1} it follows that
\begin{equation}\label{psi<=zeta}
|\psi^\lambda(\zeta)|\le|\zeta|\quad \hbox{for every}\; \zeta\in\R^m.
\end{equation}

\begin{lem}\label{estimate truncations}
Let $\eta>0$ and let $h\in \N$, $h\ge 1$, be such that 
\begin{equation}\label{property h}
c_2/(c_1 h)<\eta \quad\hbox{and}\quad 2c_3/h<\eta,
\end{equation}
where $c_1,c_2$, and $c_3$ are as in Definition \ref{volume integrands}.  
Let moreover $\alpha\ge 3$ be such that $\alpha-1\ge c_3$. Given $\lambda>0$, let $\lambda_1,\dots,\lambda_{h+1}\in\R$ be such that
\begin{align}
\lambda_1&\ge \lambda\label{lambda 1}
\\
\lambda_{i+1}&\ge \alpha \lambda_i \quad\hbox{for }i=1,\dots,h \label{lambda i+1}.
\end{align}
We set $\mu:=\lambda_{h+1}$ and, for $i=1,\dots,h$, we define $\psi_i:=\psi^{\lambda_i}$, where $\psi^{\lambda_i}$ is given by \eqref{psi lambda}. 
Then for every $i=1,\dots,h$ we have $\psi_i\in C^\infty_c(\R^m,\R^m)$, 
\begin{align}
|\psi_i(\zeta)|&\le \mu\quad\hbox{for every }\zeta\in\R^m,
\label{bound mu}
\\
\psi_i(\zeta)&=\zeta\quad\hbox{for every }\zeta\in\R^m\hbox{ with } |\zeta|\le \lambda.
\label{equality lambda}
\end{align}
Moreover, the following property holds: if the function $f\colon \R^n{\times} \R^{m{\times}n}\to [0,+\infty)$ satisfies $(f1)$, $(f3)$, $(f4)$, and the function $g\colon \R^n{\times}\R^m_0{\times} \Sph^{n-1} \to [0,+\infty)$ satisfies $(g1)$, $(g3)$, $(g4)$, $(g7)$, then for every  $u\in 
L^0(\R^n,\R^m)$ and every $A\in\A$ there exist $\hat\imath$, $\hat\jmath\in\{1,\dots,h\}$ (depending also on $f$, $g$, $u$, and $A$) such that 
\begin{align}
&F(\psi_{\hat\imath}(u),A)\le (1+\eta) F(u,A) + c_2\mathcal{L}^n(A\cap\{|u|\ge\lambda\}),
\label{estimate for F}
\\
&E(\psi_{\hat\jmath}(u),A)\le (1+\eta) E(u,A) + c_2\mathcal{L}^n(A\cap\{|u|\ge\lambda\}),
\label{estimate for E}
\end{align}
where $F$ and $E$ are as in \eqref{Effe} and \eqref{en:vs1-loc2}, respectively.
\end{lem}

\begin{proof} Since $\alpha\ge 3$, inequalities \eqref{bound mu} and \eqref{equality lambda} follow from \eqref{psi=zeta},
\eqref{psi<=2lambda}, \eqref{lambda 1}, and \eqref{lambda i+1}.

Let $f$, $g$, $u$, $A$, be as in the statement. To prove \eqref{estimate for F} and \eqref{estimate for E} it is enough to consider the case $u|_A\in GSBV^p(A,\R^m)$.
For every $i=1,\dots,h$ let $v_i:=\psi_i(u)$. Then $v_i=u$ $\mathcal{L}^n$-a.e.\ in $\{|u|\le\lambda_i\}$ by \eqref{psi=zeta} and $v_i=0$ $\mathcal{L}^n$-a.e.\ in $\{|u|\ge\lambda_{i+1}\}$ by \eqref{psi=0} and \eqref{lambda i+1}. Moreover \eqref{lip 1} gives $|\nabla v_i|\le |\nabla u|$ 
$\mathcal{L}^n$-a.e.\ in $A$. Therefore $(f3)$, $(f4)$, \eqref{lambda 1}, and \eqref{lambda i+1} yield
\begin{align}\nonumber
F(v_i,A)&\leq\int_{A\cap\{|u|\le\lambda_i\}} \!\!\!\!\!  \!\!\!\!\!  \!\!\!\!\!  \!\!\!\!\!  f(x, \nabla u)\dx+
c_2\mathcal{L}^n(A\cap\{|u|\ge \lambda_{i+1} \}) 
+ c_2\int_{A\cap\{\lambda_i<|u|<\lambda_{i+1}\}} \!\!\!\!\!  \!\!\!\!\! \!\!\!\!\! \!\!\!\!\!  \!\!\!\!\!  \!\!\!\!\!   |\nabla u|^p\dx
\\
&\le \int_A f(x, \nabla u)\dx + c_2\mathcal{L}^n(A\cap\{|u|\ge\lambda\}) + \frac{c_2}{c_1}\int_{A\cap\{\lambda_i<|u|<\lambda_{i+1}\}} \!\!\!\!\!  \!\!\!\!\!  \!\!\!\!\!  \!\!\!\!\!  \!\!\!\!\!  \!\!\!\!\!  f(x, \nabla u)\dx.
\label{F vi}
\end{align}
Since
$$
\sum_{i=1}^{h} \int_{A\cap\{\lambda_i<|u|<\lambda_{i+1}\}} \!\!\!\!\!  \!\!\!\!\!  \!\!\!\!\!  \!\!\!\!\!  \!\!\!\!\!   \!\!\!\!\!  f(x, \nabla u)\dx 
\le \int_A f(x, \nabla u)\dx,
$$
there exists $\hat\imath\in\{1,\dots,h\}$ such that
$$
\int_{A\cap\{\lambda_{\hat\imath}<|u|<\lambda_{\hat\imath+1}\}} \!\!\!\!\!  \!\!\!\!\!  \!\!\!\!\!  \!\!\!\!\!  \!\!\!\!\!  \!\!\!\!\!  f(x, \nabla u)\dx 
 \le \frac{1}{h} \int_A f(x, \nabla u)\dx.
$$
By \eqref{F vi} this implies
\begin{equation}\nonumber
F(v_{\hat\imath},A)
 \le \Big(1+\frac{c_2}{c_1 h} \Big)F( u,A) + c_2\mathcal{L}^n(A\cap\{|u|\ge\lambda\}),
\end{equation}
which gives \eqref{estimate for F} thanks to \eqref{property h}. 

To estimate $G(v_i,A)$ we use the inclusion $S_{v_i}\subset S_u\cap\big(\{|u^+|<\lambda_{i+1}\}\cup\{|u^-|<\lambda_{i+1}\}\big)$. Moreover, thanks to $(g7)$, we can choose the orientation of  $\nu_{v_i}$ so that $\nu_{v_i}=\nu_u$
$\hs^{n-1}$-a.e.\ in $S_{v_i}$. This leads to $v_i^\pm=\psi_i(u^\pm)$ $\hs^{n-1}$-a.e.\ in $S_{v_i}$.
By \eqref{lip 1} this implies that
\begin{equation}\label{jump vi}
|[v_i]|\le |[u]|\quad \hs^{n-1}\hbox{-a.e.\ on }S_{v_i}.
\end{equation}
Therefore we have
\begin{align}\nonumber
G(v_i,A)&\le\int_{S_u\cap A\cap\{|u^+|\le\lambda_i\}\cap\{|u^-|\le\lambda_i\}} \!\!\!\!\!  \!\!\!\!\! \!\!\!\!\! \!\!\!\!\!  \!\!\!\!\! \!\!\!\!\! \!\!\!\!\! \!\!\!\!\!  \!\!\!\!\!  \!\!\!\!\!  \!\!\!\!\!  g(x, [u], \nu_u)\,d\hs^{n-1}+
\int_{S_u\cap A\cap\{\lambda_i<|u^+|<\lambda_{i+1}\}}  \!\!\!\!\!  \!\!\!\!\!  \!\!\!\!\! \!\!\!\!\! \!\!\!\!\! \!\!\!\!\!  \!\!\!\!\!  \!\!\!\!\!  \!\!\!\!\!  g(x, [v_i], \nu_u)\,d\hs^{n-1} + \int_{S_u\cap A\cap\{\lambda_i<|u^-|<\lambda_{i+1}\}}  \!\!\!\!\!  \!\!\!\!\!  \!\!\!\!\! \!\!\!\!\! \!\!\!\!\! \!\!\!\!\!  \!\!\!\!\!  \!\!\!\!\!  \!\!\!\!\!  g(x, [v_i], \nu_u)\,d\hs^{n-1}
\\
&\quad+\int_{S_u\cap A\cap\{|u^+|\ge\lambda_{i+1}\}\cap\{|u^-|\le\lambda_i\}} \!\!\!\!\!  \!\!\!\!\!    \!\!\!\!\! \!\!\!\!\! \!\!\!\!\!  \!\!\!\!\! \!\!\!\!\! \!\!\!\!\! \!\!\!\!\!  \!\!\!\!\!  \!\!\!\!\!  \!\!\!\!\!  g(x, [v_i], \nu_u)\,d\hs^{n-1}\qquad
+\int_{S_u\cap A\cap\{|u^+|\le\lambda_i\}\cap\{|u^-|\ge\lambda_{i+1}\}}  \!\!\!\!\!  \!\!\!\!\!   \!\!\!\!\! \!\!\!\!\! \!\!\!\!\!  \!\!\!\!\! \!\!\!\!\! \!\!\!\!\! \!\!\!\!\!  \!\!\!\!\!  \!\!\!\!\!  \!\!\!\!\!  g(x, [v_i], \nu_u)\,d\hs^{n-1}.
\label{G vi}
\end{align}
For $\hs^{n-1}$-a.e.\ point of $\{|u^+|\ge\lambda_{i+1}\}\cap\{|u^-|\le\lambda_i\}$ we have $[v_i]=-u^-$, hence $|[v_i]|\le \lambda_i$, while \eqref{lambda i+1} implies that 
$$
|[u]|=|u^+-u^-|\ge|u^+|-|u^-|\ge \lambda_{i+1}-\lambda_i\ge (\alpha-1)\lambda_i\ge c_3 \lambda_i,
$$ 
hence $c_3 |[v_i]|\le |[u]|$. By $(g4)$ this implies
$$
g(x, [v_i], \nu_u)\le g(x, [u], \nu_u) \quad  \hs^{n-1} \textrm{-a.e.\ on } \quad   \{|u^+|\ge\lambda_{i+1}\}\cap\{|u^-|\le\lambda_i\}.
$$ 
The same inequality holds $\hs^{n-1}$-a.e.\ on $\{|u^+|\le\lambda_i\}\cap\{|u^-|\ge\lambda_{i+1}\}$. Therefore, from \eqref{jump vi}, \eqref{G vi}, and $(g3)$ we obtain
\begin{equation}\label{G vi 2}
G(v_i,A)\le \int_{S_u\cap A}  \!\!\!\!\! g(x, [u], \nu_u)\,d\hs^{n-1}+c_3
\int_{S_u\cap A\cap\{\lambda_i<|u^+|<\lambda_{i+1}\}} \!\!\!\!\!  \!\!\!\!\!  \!\!\!\!\! \!\!\!\!\! \!\!\!\!\! \!\!\!\!\!  \!\!\!\!\!  \!\!\!\!\!  \!\!\!\!\!  g(x, [u], \nu_u)\,d\hs^{n-1} +c_3\int_{S_u\cap A\cap\{\lambda_i<|u^-|<\lambda_{i+1}\}}  \!\!\!\!\! \!\!\!\!\! \!\!\!\!\! \!\!\!\!\!  \!\!\!\!\!  \!\!\!\!\!  \!\!\!\!\!  \!\!\!\!\!  \!\!\!\!\!  g(x, [u], \nu_u)\,d\hs^{n-1}.
\end{equation}
Since
\begin{align}\nonumber
\sum_{i=1}^h\Big(\frac{c_2}{c_1}&\int_{A\cap\{\lambda_i<|u|<\lambda_{i+1}\}} \!\!\!\!\!  \!\!\!\!\!  \!\!\!\!\!  \!\!\!\!\!  \!\!\!\!\!  \!\!\!\!\!   \!\!\!\!\!  f(x, \nabla u)\dx 
+c_3
\int_{S_u\cap A\cap\{\lambda_i<|u^+|<\lambda_{i+1}\}}  \!\!\!\!\!   \!\!\!\!\!   \!\!\!\!\! \!\!\!\!\! \!\!\!\!\! \!\!\!\!\!  \!\!\!\!\!  \!\!\!\!\!  \!\!\!\!\!  g(x, [u], \nu_u)\,d\hs^{n-1} +c_3\int_{S_u\cap A\cap\{\lambda_i<|u^-|<\lambda_{i+1}\}}  \!\!\!\!\! \!\!\!\!\! \!\!\!\!\! \!\!\!\!\!  \!\!\!\!\!  \!\!\!\!\!  \!\!\!\!\!  \!\!\!\!\!   \!\!\!\!\!  g(x, [u], \nu_u)\,d\hs^{n-1}\Big)
\\
&\le \frac{c_2}{c_1}\int_A f(x, \nabla u)\dx + 2c_3\int_{S_u\cap A}g(x, [u], \nu_u)\,d\hs^{n-1},
\nonumber
\end{align}
there exists $\hat\jmath\in\{1,\dots,h\}$ such that
\begin{align}\nonumber
\frac{c_2}{c_1}&\int_{A\cap\{\lambda_{\hat\jmath}<|u|<\lambda_{\hat\jmath+1}\}} \!\!\!\!\!  \!\!\!\!\!  \!\!\!\!\!  \!\!\!\!\!  \!\!\!\!\!  \!\!\!\!\!  f(x, \nabla u)\dx 
 +c_3
\int_{S_u\cap A\cap\{\lambda_{\hat\jmath}<|u^+|<\lambda_{\hat\jmath+1}\}}   \!\!\!\!\!   \!\!\!\!\!  \!\!\!\!\! \!\!\!\!\! \!\!\!\!\! \!\!\!\!\!  \!\!\!\!\!  \!\!\!\!\!  \!\!\!\!\!  g(x, [u], \nu_u)\,d\hs^{n-1} +c_3\int_{S_u\cap A\cap\{\lambda_{\hat\jmath}<|u^-|<\lambda_{\hat\jmath+1}\}} \!\!\!\!\!   \!\!\!\!\!   \!\!\!\!\! \!\!\!\!\! \!\!\!\!\! \!\!\!\!\!  \!\!\!\!\!  \!\!\!\!\!  \!\!\!\!\!  g(x, [u], \nu_u)\,d\hs^{n-1}
\\
&\le \frac{c_2}{c_1h} \int_A f(x, \nabla u)\dx + \frac{2c_3}h \int_{S_u\cap A}g(x, [u], \nu_u)\,d\hs^{n-1}.
\label{trick}
\end{align}
Inequality \eqref{estimate  for E} follows then from \eqref{property h}, \eqref{F vi}, \eqref{G vi 2}, and \eqref{trick}.
\end{proof}

The estimate in the previous lemma can be extended to the $\Gamma$-liminf, as the following result shows.

\begin{lem}\label{truncations Gamma}
Let $f_k$ and $g_k$ be as in Theorem~\ref{thm:joint}, let $E_k$ be as in \eqref{en:vs1-loc2}, with integrands $f_k$ and $g_k$, and let $E^p_k$ be the restriction of $E_k$ to $L^p_{\rm loc}(\R^n,\R^m)$. Finally, let $E'\colon L^0(\R^n,\R^m){\times} \A \to [0,+\infty]$ and $E^{\prime p}\colon L^p_{\mathrm{loc}}(\R^n,\R^m){\times} \A \to [0,+\infty]$ be defined as
$$
E'(\cdot,A) :=\Gamma\hbox{-}\liminf_{k\to +\infty} E_{k}(\cdot,A)\quad\hbox{and}\quad E^{\prime p}(\cdot,A) :=\Gamma\hbox{-}\liminf_{k\to +\infty} E^p_k(\cdot,A),
$$
where for $E'$ we use the topology of $L^0(\R^n,\R^m)$, while for $E^{\prime p}$ we use the topology of $L^p_{\mathrm{loc}}(\R^n,\R^m)$.
Under the assumptions of Lemma~\ref{estimate truncations} the 
following property holds: for every  $u\in 
L^0(\R^n,\R^m)$, $v\in 
L^p_{\mathrm{loc}}(\R^n,\R^m)$ and $A\in\A$, there exist $\hat\imath$, $\hat\jmath\in\{1,\dots,h\}$ (depending also on $u$, $v$, and $A$) such that 
\begin{align}
E'(\psi_{\hat\imath}(u),A)&\le (1+\eta) E'(u,A) + c_2\mathcal{L}^n(A\cap\{|u|\ge \lambda\}),
\label{estimate for E'}
\\
E^{\prime p}(\psi_{\hat\jmath}(u),A)&\le (1+\eta) E^{\prime p}(u,A) + c_2\mathcal{L}^n(A\cap\{|u|\ge \lambda\}).
\label{estimate for E'p}
\end{align}
\end{lem}

\begin{proof} Let $u\in 
L^0(\R^n,\R^m)$ and $A\in\A$ be fixed. Let $(u_k)$ be a sequence in $L^0(\R^n,\R^m)$ converging to $u$ in measure on bounded sets and such that
$$
E'(u,A)=\liminf_{k\to+\infty}E_k(u_k,A).
$$
There exists a subsequence $(u_{k_j})$ such that 
\begin{equation}\label{limit kj}
E'(u,A)=\lim_{j\to+\infty} E_{k_j}(u_{k_j},A).
\end{equation}
By Lemma \ref{estimate truncations} for every $j$ there exists $i_j\in\{1,\dots,h\}$ such that
$$
E_{k_j}(\psi_{i_j}(u_{k_j}),A)\le (1+\eta) E_{k_j}(u_{k_j},A) + c_2\mathcal{L}^n(A\cap\{|u_{k_j}|\ge\lambda\}).
$$
Therefore there exist $\hat\imath\in\{1,\dots,h\}$ and a sequence $j_\ell\to+\infty$ such that $i_{j_\ell}=\hat\imath$ for every $\ell$. This implies that 
$$
E_{k_{j_\ell}}(\psi_{\hat\imath}(u_{k_{j_\ell}}),A)\le (1+\eta) E_{k_{j_\ell}}(u_{k_{j_\ell}},A) + c_2\mathcal{L}^n(A\cap\{|u_{k_{j_\ell}}|\ge\lambda\}).
$$
Since $u_{k_{j_\ell}}\to u$ and $\psi_{\hat\imath}(u_{k_{j_\ell}})\to \psi_{\hat\imath}(u)$ in measure on bounded sets, taking the limit as $\ell\to+\infty$ and using \eqref{limit kj} we obtain \eqref{estimate for E'}. The same argument, with obvious changes, also proves \eqref{estimate for E'p}.
\end{proof}

We are now ready to prove the $\Gamma$-convergence of the perturbed functionals 
$E_k^{\e,p}$, which are defined on $L^p_{\mathrm{loc}}(\R^n,\R^m){\times}\A\to[0,+\infty]$ by
\begin{equation}\label{Eepsdelta}
E_k^{\e,p}(u,A):= 
\begin{cases}
\ds \int_A\!\! f_k(x,\nabla u)\dx +\int_{S_u\cap A}\!\!\!\!\! \!\!\!\!\!g^\e_k(x,[u],\nu_u)d \mathcal{H}^{n-1} &\text{if} \; u|_A\in SBV^p(A,\R^m),\cr
+\infty \quad & \mbox{otherwise in}\; L^p_{\mathrm{loc}}(\R^n,\R^m),
\end{cases}
\end{equation}
where
\begin{equation}\label{gek}
g^\e_k(x,\zeta,\nu):= g_k(x,\zeta,\nu)+\e|\zeta|.
\end{equation}

\begin{thm}\label{Gamma e}
Under the assumptions of Theorem~\ref{thm:joint}, for every $\e>0$ there exist a subsequence, not relabelled, and a functional $E^{\e,p}\colon L^p_{\mathrm{loc}}(\R^n,\R^m){\times}\A\to[0,+\infty]$ such that   for every $A \in \A$ the sequence $E_k^{\e,p}(\cdot,A)$ defined in \eqref{Eepsdelta} $\Gamma$-converges to  $E^{\e,p}(\cdot,A)$ in $L^p_{\mathrm{loc}}(\R^n,\R^m)$. Let $f^{\e,p}\colon\R^n{\times}\R^{m{\times}n}\to[0,+\infty]$ and $g^{\e,p}\colon\R^n{\times}\R^m_0{\times}\Sph^{n-1}\to[0,+\infty]$ be the functions defined by
\begin{align}
f^{\e,p}(x,\xi) &= \limsup_{\rho\to 0+} \frac{m_{E^{\e,p}}(\ell_\xi, Q_\rho(x))}{\rho^{n}},
\label{f^eta}\\
g^{\e,p}(x,\zeta,\nu) &=  \limsup_{\rho\to 0+} \frac{m_{E^{\e,p}}(u_{x,\zeta,\nu}, Q_\rho^\nu(x))}{\rho^{n-1}}.\label{g^eta}
\end{align}
Then $f^{\e,p}\in \mathcal{F}$, $g^{\e,p}$ satisfies $(g1)$, $(g3)$, $(g4)$, and $(g7)$, with $c_3$ replaced by $\hat c_3:=\max\{c_2/c_1,c_3\}$, and
\begin{equation}\label{int rep Ee}
E^{\e,p}(u,A)=
\begin{cases}
\displaystyle
\int_A f^{\e,p}(x, \nabla u) \dx + \int_{S_u\cap A}g^{\e,p}(x,[u],\nu_u)d \mathcal{H}^{n-1} \quad &\text{if } u|_A\in SBV^p(A,\R^m),\\
+\infty \quad  &\text{otherwise in } L^p_{\mathrm{loc}}(\R^n,\R^m),
\end{cases}
\end{equation}
for every $A\in \A$.
\end{thm}

\begin{proof}

For fixed $\e>0$ by $(f3)$, $(f4)$, \eqref{gek}, $(g5)$, and $(g6)$, for every $A\in \A$, we have
\begin{align}
&c_1\int_A|\nabla u|^p\dx+\int_{S_u\cap A}\big(c_4+\e|[u]|\big)d\hs^{n-1}\le E_k^{\e,p}(u,A)
\nonumber
\\
&\le c_2\int_A\big(1+|\nabla u|^p\big)\dx+(c_5+\e)\int_{S_u\cap A}\big(1+|[u]|\big)d\hs^{n-1}
\label{estimates Eepk}
\end{align}
if $u|_A\in SBV^p(A,\R^m)$, while $E_k^{\e,p}(u,A)=+\infty$ if $u|_A\notin SBV^p(A,\R^m)$.

Since the functionals $E_k^{\e,p}$ satisfy all assumptions of  \cite[Proposition 3.3]{BDfV}, there exist a subsequence, not relabelled, and a functional $E^{\e,p}\colon L^p_{\mathrm{loc}}(\R^n,\R^m){\times}\A\to[0,+\infty]$ such that   for every $A \in \A$ the sequence $E_k^{\e,p}(\cdot,A)$ $\Gamma$-converges  to  $E^{\e,p}(\cdot,A)$ in $L^p(A,\R^m)$.

Let $\Phi^\e\colon L^p_{\mathrm{loc}}(\R^n,\R^m){\times}\A\to[0,+\infty]$ be defined by
\begin{equation}\label{Psi}
\Phi^\e(u,A):=
\begin{cases}
\ds c_1\int_A|\nabla u|^p\dx+\int_{S_u\cap A}\big(c_4+\e|[u]|\big)d\hs^{n-1}&\text{if }u|_A\in SBV^p(A,\R^m),
\\
+\infty&\text{otherwise.}
\end{cases}
\end{equation}
Since $\Phi^\e(\cdot,A)$ is lower semicontinuous in $L^p_{\mathrm{loc}}(\R^n,\R^m)$ (see \cite[Theorems 2.2 and 3.7]{Amb1} or
\cite[Theorem 4.5 and Remark 4.6]{Amb2}), from \eqref{estimates Eepk} we deduce that for every $u\in L^p_{\mathrm{loc}}(\R^n,\R^m)$ and every $A\in\A$ it holds
\begin{align}
&c_1\int_A|\nabla u|^p\dx+\int_{S_u\cap A}\big(c_4+\e|[u]|\big)d\hs^{n-1}\le E^{\e,p}(u,A)
\nonumber
\\
&\le c_2\int_A\big(1+|\nabla u|^p\big)\dx+(c_5+\e)\int_{S_u\cap A}\big(1+|[u]|\big)d\hs^{n-1}
\label{estimates Eep}
\end{align}
if $u|_A\in SBV^p(A,\R^m)$, while $E^{\e,p}(u,A)=+\infty$ if $u|_A\notin SBV^p(A,\R^m)$.

In order to apply the integral representation result \cite[Theorem 1]{BFLM} we need a functional defined on $SBV^p_{\mathrm{loc}}(\R^n,\R^m){\times}\A$. Since $E^{\e,p}(u,A)$ is not defined in $SBV^p_{\mathrm{loc}}(\R^n,\R^m)\setminus L^p_{\mathrm{loc}}(\R^n,\R^m)$, we now introduce the functional $E^\e\colon SBV^p_{\mathrm{loc}}(\R^n,\R^m){\times}\A\to[0,+\infty)$ defined by
\begin{equation}\label{E lim lambda}
E^\e(u,A):=
\lim_{\lambda\to+\infty} E^{\e,p}(u^\lambda,A),
\end{equation}
where $u^\lambda:=\psi^\lambda(u)$ and $\psi^\lambda$ is as in \eqref{psi lambda}. 
\smallskip

\textit{Step 1: $E^\e$ is well defined and $E^\e = E^{\e,p}$ on $(SBV^p_{\mathrm{loc}}(\R^n,\R^m)\cap L^p_{\mathrm{loc}}(\R^n,\R^m))\times \A$.} We start by proving that $E^\e$ is well defined\ie that the limit in \eqref{E lim lambda} exists. We prove it by contradiction. Namely, if the limit in \eqref{E lim lambda} does not exist  
we can find $u\in SBV^p_{\mathrm{loc}}(\R^n,\R^m)$, $A\in\A$, $a<b$, $\lambda_j\to+\infty$, and $\mu_j\to+\infty$ such that
\begin{equation}\label{E lambda j mu j}
E^{\e,p}(u^{\lambda_j},A)> b\quad\text{and}\quad E^{\e,p}(u^{\mu_j},A)< a.
\end{equation}
Fix $\eta$, $h$, $\alpha$ as in Lemma~\ref{estimate truncations}, with $(1+\eta)a+\eta<b$. By possibly removing a finite number of terms in these sequences, it is not restrictive to assume that
\begin{equation}\label{level set}
c_2\mathcal{L}^n(A\cap\{|u|\ge \lambda_1\})<\eta,
\end{equation}
and that $\lambda_{i+1}\ge \alpha\lambda_i$ for $i=1,\dots,h$. Then by Lemma~\ref{truncations Gamma} for every $j$ there exists $i_j\in\{1,\dots,h\}$ such that
\begin{equation}\label{psi lambda mu}
E^{\e,p}(\psi_{i_j} \!(u^{\mu_j}),A)\le (1+\eta)E^{\e,p}(u^{\mu_j},A)+c_2\mathcal{L}^n(A\cap\{|u^{\mu_j}|\ge \lambda_1\}),
\end{equation}
where, here and below, we use the shorthand $\psi_k$ for $\psi^{\lambda_k}$.
Therefore there exist $\hat\imath\in\{1,\dots,h\}$ and a sequence $j_\ell\to+\infty$ such that $i_{j_\ell}=\hat\imath$ for every $\ell$.
Since $u^{\mu_{j_\ell}}\to u$ in measure on bounded sets we have that $\limsup_\ell\mathcal{L}^n(A\cap\{|u^{\mu_{j_\ell}}|\ge \lambda_1\})\le \mathcal{L}^n(A\cap\{|u|\ge \lambda_1\})$. Moreover $\psi_{\hat\imath} (u^{\mu_{j_\ell}})\to \psi_{\hat\imath} (u)$ in $L^p_{\mathrm{loc}}(\R^n,\R^m)$ as $\ell\to+\infty$. By the lower semicontinuity of the $\Gamma$-limits, from \eqref{psi lambda mu} we obtain
\begin{equation}\label{used later}
E^{\e,p}(\psi_{\hat\imath} (u),A)\le (1+\eta) \limsup_{\ell\to+\infty}E^{\e,p}(u^{\mu_{j_\ell}},A) + c_2\mathcal{L}^n(A\cap\{|u|\ge \lambda_1\}).
\end{equation}
By \eqref{E lambda j mu j} and \eqref{level set} this implies that
$$
b<E^{\e,p}(\psi_{\hat\imath} (u),A)\le (1+\eta)a +\eta,
$$
which contradicts the inequality $(1+\eta)a+\eta<b$ and hence yields the existence of the limit in \eqref{E lim lambda}. 

We note that \eqref{E lim lambda} and \eqref{used later} imply that, under the assumptions of Lemma~\ref{estimate truncations},  for every $u\in SBV^p_{\mathrm{loc}}(\R^n,\R^m)$ and every $A\in\A$, there exists $\hat\imath\in\{1,\dots,h\}$ such that
\begin{equation}\label{used later 2}
E^{\e,p}(\psi_{\hat\imath} (u),A)\le (1+\eta) E^\e(u,A) + c_2\mathcal{L}^n(A\cap\{|u|\ge \lambda_1\}).
\end{equation}

We now show that
\begin{equation}\label{Ee=Eep}
E^\e(u,A)=E^{\e,p}(u,A) \quad \forall \, (u,A) \in \big(SBV^p_{\mathrm{loc}}(\R^n,\R^m)\cap L^p_{\mathrm{loc}}(\R^n,\R^m)\big)\times \A.
\end{equation}
Fix $u$ and $A$; since $u^\lambda\to u$ in $L^p_{\mathrm{loc}}(\R^n,\R^m)$ as $\lambda\to+\infty$ by \eqref{psi=zeta}  and \eqref{psi<=zeta}, by the lower semicontinuity of the $\Gamma$-limits we have
$$
E^{\e,p}(u,A)\le \liminf_{\lambda\to+\infty} E^{\e,p}(u^\lambda,A)=E^\e(u,A).
$$
To prove the opposite inequality we fix $\eta$, $h$, and $\alpha$ as in Lemma~\ref{estimate truncations} and we consider a 
sequence $(\lambda_i)$, $\lambda_i\to +\infty$ as $i\to +\infty$, such that $\lambda_{i+1}\ge \alpha \lambda_i$ for every $i$. We now apply Lemma~\ref{truncations Gamma} to $\lambda_{i+1},\dots,\lambda_{i+h}$ and obtain that for every $i$ there exists $j_i\in\{i+1,\dots, i+h\}$ such that
$$
E^{\e,p}(u^{\lambda_{j_i}},A)\le (1+\eta) E^{\e,p}(u,A)+ c_2\mathcal{L}^n(A\cap\{|u|\ge \lambda_i\}).
$$
Taking the limit as $i\to+\infty$, by \eqref{E lim lambda} we get
$$
E^\e(u,A)\le (1+\eta) E^{\e,p}(u,A), 
$$
and taking the limit as $\eta\to0+$ we obtain
$$
E^\e(u,A)\le E^{\e,p}(u,A),
$$
which concludes the proof of \eqref{Ee=Eep}.
\smallskip

\textit{Step 2: Lower semicontinuity of $E^\e$ with respect to the strong convergence in $L^1_{\mathrm{loc}}$.} For fixed $A\in\A$ we now prove that $E^\e(\cdot,A)$ is lower semicontinuous on $SBV^p_{\mathrm{loc}}(\R^n,\R^m)$ with respect to the strong convergence in $L^1_{\mathrm{loc}}(\R^n,\R^m)$. Let us fix $u\in SBV^p_{\mathrm{loc}}(\R^n,\R^m)$ and a sequence $(u_k)$ in $SBV^p_{\mathrm{loc}}(\R^n,\R^m)$ converging to $u$ in $L^1_{\mathrm{loc}}(\R^n,\R^m)$ and such that $\lim_k E^\e(u_k,A)$ exists. Let $\eta$, $h$, $\alpha$, and $(\lambda_i)$ be as in the previous step. We now apply \eqref{used later 2} to $\lambda_{i+1},\dots,\lambda_{i+h}$ and obtain that for every $i$ and every $k$ there exists $j_{i,k}\in\{i+1,\dots, i+h\}$ such that
$$
E^{\e,p}(\psi_{j_{i,k}} (u_k),A)\le (1+\eta) E^\e(u_k,A)+ c_2\mathcal{L}^n(A\cap\{|u_k|\ge \lambda_i\}).
$$
For every $i$ there exist $N_i\in\{i+1,\dots, i+h\}$ and sequence $k_\ell^{i}\to+\infty$ as $\ell\to+\infty$ E such that $j_{i,k^i_\ell}=N_i$ for every $\ell$. Since $\psi_{N_i} (u_{k^i_\ell})$ converges to $\psi_{N_i} (u)$ in $L^p_{\mathrm{loc}}(\R^n,\R^m)$ as $\ell\to+\infty$, by the lower semicontinuity of the $\Gamma$-limits we obtain
\begin{align*}
E^{\e,p}(&\psi_{N_i} (u),A)\le \liminf_{\ell\to+\infty} E^{\e,p}(\psi_{N_i} (u_{k^i_\ell}),A)
\\
&\le (1+\eta) \lim_{\ell\to+\infty} E^\e(u_{k^i_\ell},A)+ c_2\mathcal{L}^n(A\cap\{|u|\ge \lambda_i\})
\\
&=(1+\eta) \lim_{k\to+\infty} E^\e(u_k,A)+ c_2\mathcal{L}^n(A\cap\{|u|\ge \lambda_i\}).
\end{align*}
Taking the limit first as $i\to+\infty$ and then as $\eta\to0+$, from \eqref{E lim lambda} and from the previous inequalities we obtain
$$
E^\e(u,A)\le \lim_{k\to+\infty} E^\e(u_k,A),
$$
which proves the lower semicontinuity of $E^\e(\cdot,A)$.
\smallskip

\textit{Step 3: Integral representation of $E^{\e,p}$.} By \cite[Proposition 3.3]{BDfV} for every $u\in SBV^p_{\mathrm{loc}}(\R^n,\R^m)\cap L^p_{\mathrm{loc}}(\R^n,\R^m))$ the function $A\mapsto E^{\e,p}(u,A)$ is the restriction to $\A$ of a measure defined on the $\sigma$-algebra of all Borel subsets of $\R^n$. By \eqref{estimates Eep} and \eqref{E lim lambda}, this implies that for every $u\in SBV^p_{\mathrm{loc}}(\R^n,\R^m)$ the function $A\mapsto E^\e(u,A)$ is the restriction to $\A$ of a measure defined on the Borel $\sigma$-algebra of $\R^n$ (see, e.g., \cite[Th\'eor\`eme~5.7]{DGL}). 

It follows from the definition that $E^{\e,p}$ is local\ie if $u$, $v\in L^p_{\mathrm{loc}}(\R^n,\R^m)$, $A\in\A$, and $u=v$ $\mathcal{L}^n$-a.e.\ in $A$, then $E^{\e,p}(u,A)=E^{\e,p}(v,A)$. By \eqref{E lim lambda}, this property immediately extends to $E^\e$\ie for every $u$, $v\in SBV^p_{\mathrm{loc}}(\R^n,\R^m)$, $A\in\A$, with $u=v$ $\mathcal{L}^n$-a.e.\ in $A$, we have $E^\e(u,A)=E^\e(u,A)$. Moreover, by \eqref{lip 1} we have $|\nabla u^\lambda|\le |\nabla u|$ $\mathcal{L}^n$-a.e.\ in $A$ and $|[u^\lambda]|\le |[u]|$ $\hs^{n-1}$-a.e.\ in $S_{u^\lambda}\cap A\subset S_u\cap A$. Taking into account the lower semicontinuity of $\Phi^\e$ defined in \eqref{Psi}, these inequalities, together with
 \eqref{estimates Eep} and \eqref{E lim lambda}, yield
 \begin{align*}
&c_1\int_A|\nabla u|^p\dx+\int_{S_u\cap A}\big(c_4+\e|[u]|\big)d\hs^{n-1}\le E^\e(u,A)
\\
&\le c_2\int_A\big(1+|\nabla u|^p\big)\dx+(c_5+\e)\int_{S_u\cap A}\big(1+|[u]|\big)d\hs^{n-1}
\end{align*}
for every $u\in SBV^p_{\mathrm{loc}}(\R^n,\R^m)$ and every $A\in\A$.

Therefore $E^{\e,p}$ satisfies all the assumptions of  the integral representation result \cite[Theorem 1]{BFLM}. Consequently, using also \eqref{Ee=Eep}, for every $u\in SBV^p_{\mathrm{loc}}(\R^n,\R^m)\cap L^p_{\mathrm{loc}}(\R^n,\R^m)$  and every $A\in\A$ we have the integral representation \eqref{int rep Ee} with $f^{\e,p}$ and $g^{\e,p}$ defined by \eqref{f^eta} and \eqref{g^eta}. Indeed, it is easy to deduce from \eqref{mE0eta}, \eqref{psi=zeta},  \eqref{E lim lambda}, and \eqref{Ee=Eep} that for every $x\in\R^n$, $\xi\in\R^{m{\times}n}$, $\zeta\in\R^m_0$, $\nu\in \Sph^{n-1}$, and $\rho>0$ we have
\begin{align*}
m_{E^{\e,p}}(\ell_\xi,Q_\rho(x))&=\inf\,\{E^\e(u,Q_\rho(x)): u\in SBV^p_{\mathrm{loc}}(\R^n,\R^m),\ u=\ell_\xi \hbox{ near }\partial Q_\rho(x) \},
\\
m_{E^{\e,p}}(u_{x,\zeta,\nu},Q^\nu_\rho(x))&=\inf\,\{E^\e(u,Q^\nu_\rho(x)): u\in SBV^p_{\mathrm{loc}}(\R^n,\R^m),\ u=u_{x,\zeta,\nu}\hbox{ near }\partial Q^\nu_\rho(x) \},
\end{align*}
which coincide with the definitions used in \cite{BFLM}. By locality and inner regularity, formula \eqref{int rep Ee} holds also for every
 $u\in L^p_{\mathrm{loc}}(\R^n,\R^m)$ and every $A\in\A$ such that $u|_A\in SBV^p(A,\R^m)$.

The Borel measurability of $f^{\e,p}$ and $g^{\e,p}$ are then proved in Lemma \ref{measurability}.

\smallskip

\textit{Step 4: $f^{\e,p}$ satisfies $(f2)$, $(f3)$ and $(f4)$.} We now show that $f^{\e,p}$ satisfies $(f2)$. Since $(f2)$ holds for $f_k$, for every $A\in \A$ we have
$$
E^{\e,p}_k(u+\ell_\xi,A)\le E^{\e,p}_k(u,A)+\sigma_1(|\xi|)\big(\mathcal{L}^n(A)+E^{\e,p}_k(u+\ell_\xi,A)+E^{\e,p}_k(u,A)\big)
$$
for every $\xi\in\R^{m{\times}n}$ and for every $u\in L^p_{\mathrm{loc}}(\R^n,\R^m)$. We have
\begin{equation}\label{new *} 
(1-\sigma_1(|\xi|))E^{\e,p}_k(u+\ell_\xi,A)\le (1+\sigma_1(|\xi|))E^{\e,p}_k(u,A)+\sigma_1(|\xi|)\mathcal{L}^n(A),
\end{equation}
thus if $\sigma_1(|\xi|)<1$ taking the $\Gamma$-limit gives  
$$
(1-\sigma_1(|\xi|))E^{\e,p}(u+\ell_\xi,A)\le (1+\sigma_1(|\xi|))E^{\e,p}(u,A)+\sigma_1(|\xi|)\mathcal{L}^n(A).
$$
This implies that 
\begin{equation}\label{new **}
(1-\sigma_1(|\xi_2-\xi_1|))\,m_{E^{\e,p}}(\ell_{\xi_2},Q_\rho(x))\le (1+\sigma_1(|\xi_2-\xi_1|))\,m_{E^{\e,p}}(\ell_{\xi_1},Q_\rho(x))
+\sigma_1(|\xi_2-\xi_1|)\rho^n
\end{equation}
for every $\rho>0$, $x\in\R^n$, and $\xi_1$, $\xi_2\in\R^{m{\times}n}$ with $\sigma_1(|\xi_2-\xi_1|)< 1$. Dividing by $\rho^n$ and taking the limsup as $\rho\to0+$ we obtain from \eqref{f^eta} and \eqref{new **}
$$
(1-\sigma_1(|\xi_2-\xi_1|))f^{\e,p}(x,\xi_2)\le(1+\sigma_1(|\xi_2-\xi_1|)) f^{\e,p}(x,\xi_1)+\sigma_1(|\xi_2-\xi_1|).
$$
which implies
$$
f^{\e,p}(x,\xi_2) \leq f^{\e,p}(x,\xi_1)+ \sigma_1(|\xi_2-\xi_1|)(1+f^{\e,p}(x,\xi_1)+f^{\e,p}(x,\xi_2)).
$$
This inequality is trivial if $\sigma_1(|\xi_2-\xi_1|)\ge 1$.
Exchanging the roles of $\xi_1$ and $\xi_2$ we obtain $(f2)$ for~$f^{\e,p}$. 

Let us prove that $f^{\e,p}$ satisfies $(f3)$. By \eqref{estimates Eepk} for every $u\in L^p_{\mathrm{loc}}(\R^n,\R^m)$ and every $A\in\A$ we have that 
$E_k^{\e,p}(u,A)\ge \Phi^\e(u,A)$ for every $k$, where $\Phi^\e$ is defined by \eqref{Psi}. 
Since $\Phi^\e(\cdot,A)$ is lower semicontinuous in $L^p_{\mathrm{loc}}(\R^n,\R^m)$, this inequality is preserved in the $\Gamma$-limit and hence we get
\begin{equation}\label{Ee>=Psi}
E^{\e,p}(u,A)\ge \Phi^\e(u,A)
\end{equation}
for every $u\in L^p_{\mathrm{loc}}(\R^n,\R^m)$ and every $A\in\A$. 

Let $\phi^\e\colon\R^n{\times}\R^{m{\times}n}\to[0,+\infty]$ be defined by
\begin{equation}
\phi^\e(x,\xi) := \limsup_{\rho\to 0+} \frac{m_{\Phi^\e}(\ell_\xi, Q_\rho(x))}{\rho^{n}}.
\label{psi2}
\end{equation}
Note that, by translation invariance, $\phi^\e(x,\xi)=\phi^\e(0,\xi)$ for every $x\in\R^n$ and every $\xi\in\R^{m{\times}n}$.
We can now apply  the integral representation result \cite[Theorem 1]{BFLM} to $\Phi^\e$ and, taking $u=\ell_\xi$ and $A=Q(0)$, we obtain
$$
c_1|\xi|^p=\Phi^\e(\ell_\xi,Q(0))=\int_{Q(0)}\phi^\e(y,\xi)\,dy=\phi^\e(0,\xi)=\phi^\e(x,\xi)
$$
for every $x\in\R^n$ and every $\xi\in\R^{m{\times}n}$. Together with \eqref{f^eta}, \eqref{Ee>=Psi}, and \eqref{psi2}, this gives the lower bound $(f3)$ for $f^{\e,p}$.

To prove the upper bound $(f4)$, we observe that $E_k^{\e,p}(\ell_\xi,Q_\rho(x))\le c_2(1+|\xi|^p)\rho^n$ for every $x\in \R^n$, $\xi\in \R^{m{\times}n}$, $\rho>0$ and $k$. This implies that $E^{\e,p}(\ell_\xi,Q_\rho(x))\le c_2(1+|\xi|^p)\rho^n$, hence $m_{E^{\e,p}}(\ell_\xi,Q_\rho(x))\le c_2(1+|\xi|^p)\rho^n$. The upper bound $(f4)$ for $f^{\e,p}$ follows from \eqref{f^eta}.
\smallskip

\textit{Step 5: $g^{\e,p}$ satisfies $(g3)$, $(g4)$ and $(g7)$.}
To prove $(g3)$ we fix $\zeta_1$, $\zeta_2\in\R^m_0$, with $|\zeta_1|\le |\zeta_2|$, and  a rotation $R$ on $\R^m$ such that $aR\zeta_2=\zeta_1$, where $a:= |\zeta_1|/|\zeta_2|\le 1$. Since $f_k$ and $g^\e_k$ (see \eqref{gek}) satisfy $(f3)$, $(f4)$, and $(g3)$, for every $A\in\A$ and every $u\in L^p_{\mathrm{loc}}(\R^n,\R^m)$, with $u|_A\in SBV^p(A,\R^m)$, we have
\begin{align*}
E_k^{\e,p}(aRu,A)&=\int_Af_k(x,aR\nabla u)\dx + \int_{S_u\cap A}g^\e_k(x, aR[u], \nu_u)d\hs^{n-1}
\\
&\le c_2\mathcal{L}^n(A)+c_2\int_A|\nabla u|^p\dx+ c_3 \int_{S_u\cap A}g^\e_k(x, [u], \nu_u)d\hs^{n-1} 
\\
&\le c_2\mathcal{L}^n(A)+\frac{c_2}{c_1}\int_Af_k(x,\nabla u)\dx+ c_3 \int_{S_u\cap A}g^\e_k(x, [u], \nu_u)d\hs^{n-1}.
\end{align*}
Passing to the $\Gamma$-limit, we obtain
$
E^{\e,p}(aRu,A)\le c_2\mathcal{L}^n(A)+ \hat c_3 E^{\e,p}(u,A)
$,
with $\hat c_3=\max\{c_2/c_1,c_3\}$.  This implies that
$
m_{E^{\e,p}}(u_{x,aR\zeta_2,\nu},Q^\nu_\rho(x))\le c_2\rho^n+ \hat c_3 m_{E^{\e,p}}(u_{x,\zeta_2,\nu},Q^\nu_\rho(x))
$
for every $x\in\R^n$, $\nu\in\Sph^{n-1}$, and $\rho>0$. Since $aR\zeta_2=\zeta_1$, using \eqref{g^eta} we obtain
$
g^{\e,p}(x,\zeta_1,\nu)\le  \hat c_3\, g^{\e,p}(x,\zeta_2,\nu)
$,
which proves $(g3)$, with $c_3$ replaced by $\hat c_3$.

To prove $(g4)$ we fix $\zeta_1$, $\zeta_2\in\R^m_0$, with $\hat c_3|\zeta_1|\le |\zeta_2|$, and  a rotation $R$ on $\R^m$ such that $aR\zeta_2=\zeta_1$, where $a:= |\zeta_1|/|\zeta_2|\le 1/\hat c_3\le 1$.
Since $f_k$ and $g^\e_k$ satisfy $(f3)$, $(f4)$, and $(g4)$, the inequalities $c_3a\le \hat c_3a\le1$ imply that for every $A\in\A$ and every $u\in L^p_{\mathrm{loc}}(\R^n,\R^m)$, with $u|_A\in SBV^p(A,\R^m)$, we have
\begin{align*}
E_k^{\e,p}(aRu,A )&=\int_Af_k(x,aR\nabla u)\dx + \int_{S_u\cap A}g^\e_k(x, aR[u], \nu_u)d\hs^{n-1}
\\
&\le c_2\mathcal{L}^n(A)+c_2a^p\int_A|\nabla u|^p\dx+ \int_{S_u\cap A}g^\e_k(x, [u], \nu_u)d\hs^{n-1} 
\\
&\le c_2\mathcal{L}^n(A)+\frac{c_2a^p}{c_1}\int_Af_k(x,\nabla u)\dx+ \int_{S_u\cap A}g^\e_k(x, [u], \nu_u)d\hs^{n-1}.
\end{align*}
Since $a\le 1$ and $\hat c_3 a\le 1$, we have $c_2a^p/c_1\le c_2a/c_1\le \hat c_3 a\le 1$. Therefore $E_k^{\e,p}(aRu,A)\le  c_2\mathcal{L}^n(A)+ E_k^{\e,p}(u,A)$.
Passing to the $\Gamma$-limit, we obtain
$
E^{\e,p}(aRu,A)\le c_2\mathcal{L}^n(A)+  E^{\e,p}(u,A)
$.
This implies that
$
m_{E^{\e,p}}(u_{x,aR\zeta_2,\nu},Q^\nu_\rho(x))\le c_2\rho^n+ m_{E^{\e,p}}(u_{x,\zeta_2,\nu},Q^\nu_\rho(x))
$
for every $x\in\R^n$,  $\nu\in\Sph^{n-1}$, and  $\rho>0$. Since $aR\zeta_2=\zeta_1$, using \eqref{g^eta} we obtain
$
g^{\e,p}(x,\zeta_1,\nu)\le g^{\e,p}(x,\zeta_2,\nu)
$,
which proves $(g4)$, with $c_3$ replaced by $\hat c_3$.

To prove the symmetry condition $(g7)$ for $g^{\e,p}$, we observe that $u_{x,-\zeta,-\nu}=u_{x,\zeta,\nu}-\zeta$
for every $x\in\R^n$, $\zeta\in \R_0^m$, and $\nu\in \Sph^{n-1}$. Therefore $u\in SBV^p(Q_\rho^\nu(x),\R^m)\cap L^p(Q_\rho ^\nu(x),\R^m)$ satisfies $u=u_{x,-\zeta,-\nu}$ in a neighbourhood of $\partial Q_\rho ^\nu(x)$ if and only if $u=v-\zeta$ for some $v\in SBV^p(Q_\rho ^\nu(x),\R^m)\cap L^p(Q_\rho^\nu(x),\R^m)$ satisfying $v=u_{x,\zeta,\nu}$ in a neighbourhood of $\partial Q_\rho ^\nu(x)$. Since $Q_\rho ^{-\nu}(x)=Q_\rho ^{\nu}(x)$ by (k) in Section \ref{Notation}, it follows that $m_{E^{\e,p}}(u_{x,-\zeta,-\nu},Q_\rho ^{-\nu}(x))=m_{E^{\e,p}}(u_{x,\zeta,\nu},Q_\rho ^\nu(x))$. By \eqref{g^eta} this implies that $g^{\e,p}(x,\zeta,\nu)=g^{\e,p}(x,-\zeta,-\nu)$, which proves $(g7)$ for $g^{\e,p}$.
\end{proof}

\section{Proof of the compactness result}\label{Gamma free}

In this section we begin the proof of the compactness result with respect to $\Gamma$-convergence, Theorem~\ref{thm:joint}. We start with the following perturbation result, which, together with Theorem~\ref{Gamma e}, provides a slightly weaker version of Theorem~\ref{thm:joint}. Indeed it does not establish that the surface integrand $g^0$, defined in \eqref{def g} below, satisfies properties $(g2)$, $(g5)$, and $(g6)$.

\begin{thm}[Perturbation result]\label{perturb} Under the hypotheses of Theorem~\ref{thm:joint},
let $D$ be a countable subset of $(0,+\infty)$ with $0\in\overline D$. 
Assume that for every $\e\in D$ there exists a functional
$E^{\e,p}\colon L^p_{\mathrm{loc}}(\R^n,\R^m){\times}\A\to[0,+\infty]$ such that for every $A \in \A$ the sequence $E_k^{\e,p}(\cdot,A)$ defined in \eqref{Eepsdelta} $\Gamma$-converges to  $E^{\e,p}(\cdot,A)$ in $L^p_{\mathrm{loc}}(\R^n,\R^m)$. Let $f^{\e,p}$ and $g^{\e,p}$ be the functions defined by \eqref{f^eta} and \eqref{g^eta}, and let $f^0\colon\R^n{\times}\R^{m{\times}n}\to[0,+\infty]$ and $g^0\colon\R^n{\times}\R^m_0{\times}\Sph^{n-1}\to[0,+\infty]$ be the functions defined by
\begin{align}\label{def f}
f^0(x,\xi)&:=\inf_{\e\in D}f^{\e,p}(x,\xi)=\lim_{\substack{\e\to0+\\\e\in D}}f^{\e,p}(x,\xi),
\\
\label{def g}
g^0(x,\zeta,\nu)&:=\inf_{\e\in D}g^{\e,p}(x,\zeta,\nu)=\lim_{\substack{\e\to0+\\\e\in D}}g^{\e,p}(x,\zeta,\nu).
\end{align}
Then $f^0\in\mathcal{F}$ and $g^0$ satisfies $(g1)$, $(g3)$, $(g4)$, and $(g7)$, with $c_3$ replaced by $\hat c_3:=\max\{c_2/c_1,c_3\}$. 

Let $E^0$ and $E_k$ be as in  \eqref{en:vs1-loc2}, with $f$ and $g$ replaced by $f^0$ and $g^0$ and by $f_k$ and $g_k$, respectively, and let $E^{0,p}$ and $E^p_k$ be the corresponding restrictions to $L^p_{\mathrm{loc}}(\R^n,\R^m){\times}\A$. Then 
\begin{align*}
E_k(\cdot,A)\ &\Gamma\hbox{-converges to }E^0(\cdot,A)\hbox{ in }L^0(\R^n,\R^m),
\\
E^p_k(\cdot,A)\ &\Gamma\hbox{-converges to }E^{0,p}(\cdot,A)\hbox{ in }L^p_{\mathrm{loc}}(\R^n,\R^m),
\end{align*}
for every $A\in \A$.
\end{thm}

\begin{proof}
By Theorem \ref{Gamma e} $E^{\e,p}$ can be written in integral form as in \eqref{int rep Ee},  where $f^{\e,p}$ and $g^{\e,p}$ are defined by \eqref{f^eta} and \eqref{g^eta} and satisfy
$(f1)$-$(f4)$ and $(g1)$, $(g3)$, $(g4)$, $(g7)$. 
It follows from \eqref{f^eta} and \eqref{g^eta}  that  $f^{\e_1,p}\le f^{\e_2,p}$ and $g^{\e_1,p}\le g^{\e_2,p}$ for $0<\e_1<\e_2$.

Properties  $(f1)$-$(f4)$ for $f^0$ and properties $(g1)$, $(g3)$, $(g4)$, $(g7)$ for $g^0$ follow from \eqref{def f} and \eqref{def g} and from the corresponding properties for $f^{\e,p}$ and $g^{\e,p}$.

By the Monotone Convergence Theorem we have
\begin{equation}\label{lim Epe}
E^{0,p}(u,A)= \lim_{\substack{\e\to0+\\\e\in D}}E^{\e,p}(u,A)
\end{equation}
for every $A\in\A$ and every $u\in L^p_{\mathrm{loc}}(\R^n,\R^m)$ with $u|_A\in SBV^p(A,\R^m)$.

Let $E'$, $E''\colon L^0(\R^n,\R^m){\times} \A \to [0,+\infty]$ and $E^{\prime p}$, $E^{\prime\prime p}\colon L^p_{\mathrm{loc}}(\R^n,\R^m){\times} \A \to [0,+\infty]$ be defined by
\begin{align*}
E'(\cdot,A) :=\Gamma\hbox{-}\liminf_{k\to +\infty} E_{k}(\cdot,A)\quad&\text{and}\quad E''(\cdot,A):=\Gamma\hbox{-}\limsup_{k\to +\infty} E_{k}(\cdot,A),
\\
E^{\prime p}(\cdot,A) :=\Gamma\hbox{-}\liminf_{k\to +\infty} E^p_k(\cdot,A)\quad&\text{and}\quad E^{\prime\prime p}(\cdot,A):=\Gamma\hbox{-}\limsup_{k\to +\infty} E^p_k(\cdot,A),
\end{align*}
where for  $E'$ and $E''$ we use the topology of $L^0(\R^n,\R^m)$, while for $E^{\prime p}$ and $E^{\prime\prime p}$ we use the topology of $L^p_{\mathrm{loc}}(\R^n,\R^m)$.

Then for every $u\in L^p_{\mathrm{loc}}(\R^n,\R^m)$ and for every $\e\in D$ we have
$E''(u,A)\le E^{\prime\prime p}(u,A)\le E^{\e,p}(u,A)$, thus by \eqref{lim Epe}
\begin{equation}\label{E''<=E SBVp}
E''(u,A)\le E^{\prime\prime p}(u,A)\le E^{0,p}(u,A)=E^0(u,A)
\end{equation}
for every $A\in\A$ and $u\in L^p_{\mathrm{loc}}(\R^n,\R^m)$ with $u|_A\in SBV^p(A,\R^m)$.

We claim that
\begin{equation}\label{E<=E' Linfty}
E^0(u,A)=E^{0,p}(u,A)\le E'(u,A)\le E^{\prime p}(u,A) 
\end{equation}
for every $A \in \A$ and every $u\in L^\infty(\R^n,\R^m)$.
Let us fix $A$ and $u$. The inequality $E'(u,A)\le E^{\prime p}(u,A)$ is trivial. By $\Gamma$-convergence there exists a sequence $(u_k)$ converging to $u$ in $L^0(\R^n,\R^m)$ such that
\begin{equation}\label{recovery E'}
E'(u,A)=\liminf_{k\to+\infty}E_k(u_k,A).
\end{equation}
Let us fix $\lambda>\|u\|_{L^\infty(\R^n\!,\,\R^m)}$ and $\e>0$.  
By Lemma \ref{estimate truncations} there exist $\mu>\lambda$, independent of $k$, and a sequence $(v_k)\subset L^\infty(\R^n,\R^m)$, converging to $u$ in measure on bounded sets, such that for every $k$ we have
\begin{align}
&\|v_k\|_{L^\infty(\R^n\!,\,\R^m)}\le \mu,
\label{Linfty bound k}
\\
&v_k=u_k \quad\mathcal{L}^n\hbox{-a.e.\ in }\{|u_k|\le\lambda\},
\label{equality for k}
\\
&E_k(v_k,A)\le (1+\e) E_k(u_k,A) + c_2\mathcal{L}^n(A\cap\{|u_k|\ge\lambda\}).
\label{estimate for Ek}
\end{align}
 It follows from \eqref{Linfty bound k} that $v_k\to u$ also in $L^p_{\mathrm{loc}}(\R^n,\R^m)$. If $E_k(u_k,A)<+\infty$, by $(f3)$, $(g5)$, and \eqref{estimate for Ek} the function $v_k$ belongs to $GSBV^p(A,\R^m)$  and 

\begin{equation}\label{c:est-on-svk}
\mathcal{H}^{n-1}(S_{v_k}\cap A)\le  (1/c_4)(1+\e)E_k(u_k,A)+  (c_2/c_4)\mathcal{L}^n(A\cap\{|u_k|\ge\lambda\}).
\end{equation}
By \eqref{Eepsdelta} and \eqref{Linfty bound k} this implies that 
$$
E^{\e,p}_k(v_k,A)\le E_k(v_k,A)+2\e\mu\mathcal{H}^{n-1}(S_{v_k}\cap A),
$$ 
which, in its turn, by \eqref{estimate for Ek} and \eqref{c:est-on-svk}, leads to
$$
E^{\e,p}_k(v_k,A)\le (1+\e) (1 +(2\e\mu/c_4)) E_k(u_k,A) + c_2(1 +(2\e\mu/c_4))\mathcal{L}^n(A\cap\{|u_k|\ge\lambda\}).
$$
Clearly this inequality holds also when $E_k(u_k,A)=+\infty$.
Therefore, using  \eqref{recovery E'} and the inequality $\|u\|_{L^\infty(\R^n\!,\,\R^m)}<\lambda$, by $\Gamma$-convergence we get
$$
E^{\e,p}(u,A)\le (1+\e) (1 +(2\e\mu/c_4)) E'(u,A)
$$
for every $\e\in D$. By \eqref{lim Epe}, passing to the limit as $\e\to0+$ we obtain \eqref{E<=E' Linfty} whenever $u\in L^\infty(\R^n,\R^m)$.

We now prove that
\begin{equation}\label{E''<=E}
E''(u,A)\le E^0(u,A) \quad\hbox{for every }u\in L^0(\R^n,\R^m)\hbox{ and every }A\in \A.
\end{equation}
Let us fix $u$ and $A$. It is enough to prove the inequality when $u|_A\in GSBV^p(A,\R^m)$. By Lemma \ref{estimate truncations} for every $\e>0$ and for every integer $k\ge 1$ there exists $u_k \in  L^\infty(\R^n,\R^m)$, with $u_k|_A\in SBV^p(A,\R^m)$, such that $u_k=u$ $\mathcal{L}^n$-a.e.\ in $\{|u|\le k\}$ and 
$$
E^0(u_k,A)\le (1+\e) E^0(u,A) + c_2\mathcal{L}^n(A\cap\{|u|\ge k\}).
$$
By \eqref{E''<=E SBVp} we have
$
E''(u_k,A)\le E^0(u_k,A)
$, hence
$$
E''(u_k,A)\le (1+\e) E^0(u,A) + c_2\mathcal{L}^n(A\cap\{|u|\ge k\}).
$$
Since  $u_k\to u$ in measure on bounded sets, passing to the limit as $k\to +\infty$, by the lower semicontinuity of the $\Gamma$-limsup we deduce
$$
E''(u,A)\le (1+\e) E^0(u,A).
$$
Hence letting $\e\to0+$ we obtain \eqref{E''<=E}. The same proof shows that
\begin{equation}\label{E''p<=E}
E^{\prime\prime p}(u,A)\le E^{0,p}(u,A) \quad\hbox{for every }u\in L^p_{\mathrm{loc}}(\R^n,\R^m)\hbox{ and every }A\in \A.
\end{equation}

We now prove that
\begin{equation}\label{E<=E'}
E^0(u,A)\le E'(u,A) \quad\hbox{for every }u\in L^0(\R^n,\R^m)\hbox{ and every }A\in \A.
\end{equation}
Let us fix $u$ and $A$. It is enough to prove the inequality when $u|_A\in GSBV^p(A,\R^m)$, since otherwise $E'(u,A)=+\infty$ due to the lower bounds $(f3)$ and $(g5)$. By Lemma \ref{truncations Gamma} for every $\e>0$ and every integer $k\ge 1$ there exists $u_k \in  L^\infty(\R^n,\R^m)$, with $u_k|_A\in SBV^p(A,\R^m)$, such that $u_k=u$ $\mathcal{L}^n$-a.e.\ in $\{|u|\le k\}$,  $u^\pm_k=u^\pm$ $\hs^{n-1}$-a.e.\ in $S_u\cap \{|u^\pm|\le k\}$, and 
$$
E'(u_k,A)\le (1+\e) E'(u,A) + c_2\mathcal{L}^n(A\cap\{|u|\ge k\}).
$$
By \eqref{E<=E' Linfty} we have $E^0(u_k,A)\le E'(u_k,A)$, hence
$$
\int_{A\cap \{|u|\le k\}}\!\!\!\!\! \!\!\!\!\! \!\!\!\!\! \!\!\!\!\! f^0(x,\nabla u)\dx +
\int_{S_u\cap A\cap\{|u^+|\le k\}\cap \{|u^-|\le k\}} \!\!\!\!\! \!\!\!\!\!  \!\!\!\!\! \!\!\!\!\! \!\!\!\!\! \!\!\!\!\! \!\!\!\!\! \!\!\!\!\! \!\!\!\!\! \!\!\!\!\! 
g^0(x,[u],\nu_u)\,d\hs^{n-1} \le E^0(u_k,A) \le (1+\e) E'(u,A) + c_2\mathcal{L}^n(A\cap\{|u|\ge k\}).
$$
As $k\to+\infty$ we get
$$
E^0(u,A)= \int_A f^0(x,\nabla u)\dx +
\int_{S_u\cap A}
g^0(x,[u],\nu_u)\,d\hs^{n-1}\le (1+\e) E'(u,A),
$$
and as $\e\to0+$ we obtain \eqref{E<=E'}. 
Since $E'(u,A)\le E^{\prime p}(u,A)$ for every $u\in L^p_{\mathrm{loc}}(\R^n,\R^m)$, from \eqref{E<=E'} we   also get
\begin{equation}\label{E<=E'p}
E^{0,p}(u,A)\le E^{\prime p}(u,A) \quad\hbox{for every }u\in L^p_{\mathrm{loc}}(\R^n,\R^m)\hbox{ and every }A\in \A.
\end{equation}

The $\Gamma$-convergence of $E_k(\cdot,A)$ to $E^0(\cdot,A)$ in  $L^0(\R^n,\R^m)$ follows from 
\eqref{E''<=E} and \eqref{E<=E'}, while the $\Gamma$-convergence of $E^p_k(\cdot,A)$ to $E^{0,p}(\cdot,A)$ in  $L^p_{\mathrm{loc}}(\R^n,\R^m)$ follows from
\eqref{E''p<=E} and \eqref{E<=E'p}.
\end{proof}

To conclude the proof of Theorem \ref{thm:joint} and to prepare the proof of Theorem~\ref{G-convE}, we now establish some relations between the functions $f^0$
and $g^0$ introduced in Theorem \ref{perturb} and the functions $f'$, $f''$, $g'$, and $g''$ defined in \eqref{f'}-\eqref{g''}.

\begin{thm}\label{6.3}
Under the assumptions of Theorems~\ref{thm:joint} and~\ref{perturb}, let $f^0$ and $g^0$ be defined by \eqref{def f} and \eqref{def g} and let
 $f'$, $f''$, $g'$, and $g''$ be defined by \eqref{f'}-\eqref{g''}. Then  
\begin{itemize}
\item[(a)] for every $x\in\R^n$ and every $\xi\in \R^{m{\times}n}$ we have 
$
f^0(x,\xi)\le f'(x,\xi)
$; 
\item[(b)] for $\mathcal{L}^n$-a.e.\ $x\in\R^n$ we have 
$
f''(x,\xi)\le  f^0(x,\xi)
$
for every $\xi\in\R^{m{\times}n}$;
\item[(c)] for every $x\in \R^n$, every $\zeta\in \R^m_0$, and every $\nu\in\Sph^{n-1}$ we have
$
g^0(x,\zeta,\nu)\le g'(x,\zeta,\nu)
$;
\item[(d)] for every $A\in\A$ and every $u\in GSBV^p(A,\R^m)$ we have
\begin{equation}\label{g0>=ghom}
g''(x,[u](x),\nu_u(x)) \le g^0(x,[u](x),\nu_u(x))
\end{equation}
for $\hs^{n-1}$-a.e.\ $x\in S_u\cap A$. 
\end{itemize}
\end{thm}

The proof of Theorem~\ref{6.3} is postponed to Sections~\ref{proof volume}
and~\ref{estimate surface}.

\begin{rem}\label{c:facile-oss}
Since by definition $f'\leq f''$ and $g'\leq g''$, Theorem \ref{6.3} implies that for $\mathcal{L}^n$-a.e.\ $x\in\R^n$ we have 
$
f'(x,\xi)=f''(x,\xi)=  f^0(x,\xi)
$
for every $\xi\in\R^{m{\times}n}$, and that for every $A\in\A$ and every $u\in GSBV^p(A,\R^m)$ we have
$$
g'(x,[u](x),\nu_u(x))= g''(x,[u](x),\nu_u(x)) = g^0(x,[u](x),\nu_u(x)),
$$
for $\hs^{n-1}$-a.e.\ $x\in S_u\cap A$.
\end{rem}

Appealing to Theorem \ref{6.3} we can now conclude the proof of the compactness result, Theorem~\ref{thm:joint}.

\begin{proof}[Proof of Theorem \ref{thm:joint}]
By combining Theorem \ref{Gamma e} and a diagonal argument, we obtain a subsequence, not relabelled, and, for every $\e\in D$,  a functional $E^{\e,p}\colon L^p_{\mathrm{loc}}(\R^n,\R^m){\times}\A\to[0,+\infty]$, such that   for every $A \in \A$ the sequence $E_k^{\e,p}(\cdot,A)$ $\Gamma$-converges in $L^p_{\mathrm{loc}}(\R^n,\R^m)$ to  $E^{\e,p}(\cdot,A)$ for every $\e\in D$. 
By Theorem~\ref{perturb}
$E_{k} (\cdot, A)$ $\Gamma$-converges to $E^0 (\cdot, A)$ in $L^0(\R^n,\R^m)$
for every $A \in \A$, and $E^0$ can be written as 
$$
E^0(u,A)=\int_A f^0(x,\nabla u)\,dx+ \int_{S_u\cap A} g^0 (x,[u],\nu_u)\,d\mathcal H^{n-1},
$$
where $f^0$ and $g^0$ are defined as in \eqref{def f}
and \eqref{def g} (note that $f^0$ and $g^0$ depend on the chosen subsequence). Note that $f^0\in\mathcal{F}$, but $g^0$ only satisfies $(g1)$, $(g3)$, $(g4)$, and $(g7)$, with $c_3$ replaced by $\hat c_3:=\max\{c_2/c_1,c_3\}$. 
To conclude the proof it remains to show that there exists $g\in\mathcal{G}$, possibly different from $g^0$, such that $E^0$ can still be represented as in \eqref{en:vs1-loc2} using  $f^0$ and $g$.

Let now $g'$ be defined as in \eqref{g'}
(note that also this function depends on the chosen subsequence). 
We can now apply Theorem~\ref{6.3} and Remark \ref{c:facile-oss} to obtain
\begin{align*}
E^0(u,A) &= \int_A f^0 (x,\nabla u)\,dx+ \int_{S_u\cap A} g^0 (x,[u],\nu_u)\,d\mathcal H^{n-1} \\
&= \int_A  f^0 (x,\nabla u)\,dx+ \int_{S_u\cap A} g' 
 (x,[u],\nu_u)\,d\mathcal H^{n-1}.
\end{align*}
Since $g' \in \mathcal{G}$ by Lemma~\ref{g' in G}, the theorem is proved.  \end{proof}

\section{Identification of the $\Gamma$-limit and related results} \label{appl conv min}

In this section we prove Theorem \ref{G-convE} using Theorem \ref{6.3}, which will be proved in Sections \ref{proof volume} and \ref{estimate surface}. We also prove a result on the convergence of minimisers. 

\begin{proof}[Proof of Theorem \ref{G-convE}] To prove that  (a1) and (a2) imply \eqref{Gamma L0}, we observe that, 
by the Urysohn property of $\Gamma$-convergence \cite[Proposition 8.3]{DM93},  the sequence $E_k(\cdot,A)$ $\Gamma$-converges to $E_\infty(\cdot,A)$ in $L^0(\R^n,\R^m)$ for every $A\in\A$  if and only if for every $A\in\A$ every subsequence of  $E_k(\cdot,A)$ has a sub-subsequence $\Gamma$-converging to $E_\infty(\cdot,A)$ in $L^0(\R^n,\R^m)$. 

Let $D$ be a countable subset of $(0,+\infty)$ with $0\in\overline D$. By Theorem \ref{Gamma e}, using a diagonal argument, 
for every subsequence of $(E_k)$ we obtain a sub-subsequence $(E_{k_j})$ which satisfies the assumptions of Theorem~\ref{perturb}. Let $f^0$, $g^0$, and $E^0$ be defined as in Theorem~\ref{perturb}, corresponding to the subsequence $(E_{k_j})$.
Then $E_{k_j}\! (\cdot,A)$ $\Gamma$-converges to $E^0(\cdot,A)$ for every $A\in\A$. 
Thus, proving \eqref{Gamma L0} is equivalent to showing that
\begin{equation}\label{E0=Ehom}
E^0(u,A)=E_\infty(u,A) \quad\hbox{for every } u\in L^0(\R^n,\R^m) \hbox{ and every } A\in\A.
\end{equation}

Let $\tilde{f}', \tilde{f}'', \tilde{g}', \tilde{g}''$ be the functions defined 
as in \eqref{f'}-\eqref{g''}, corresponding to the subsequences $F_{k_j}$ and $G_{k_j}$.
Since
$$
f' \leq \tilde{f}' \leq \tilde{f}'' \leq f'' \quad \text{ and } \quad 
g' \leq \tilde{g}' \leq \tilde{g}'' \leq g'',
$$
equalities (a1) give 
$$
f_\infty (x, \xi) = \tilde{f}' (x, \xi) = \tilde{f}'' (x, \xi)
\text{ for   $\mathcal{L}^n$-a.e. $x \in \R^n$ and every $\xi \in\R^{m{\times}n}$},
$$
while (a2) implies that for every $A\in\A$ and every $u\in GSBV^p(A,\R^m)$ we have
$$
g_\infty (x,[u](x),\nu_u(x))
=\tilde{g}' (x,[u](x),\nu_u(x))
=  \tilde{g}'' (x,[u](x),\nu_u(x))
$$
for $\hs^{n-1}$-a.e. $x\in S_u\cap A$.

By Theorem \ref{6.3} and Remark \ref{c:facile-oss} we have
\begin{align*} 
&f^0(x, \xi) = \tilde{f}' (x, \xi) = \tilde{f}'' (x, \xi)
\text{ for   $\mathcal{L}^n$-a.e. $x \in \R^n$ and every $\xi \in\R^{m{\times}n}$},
\\
&\int_{S_u\cap A} g^0 (x,[u],\nu_u)d\hs^{n-1}
= \int_{S_u\cap A} \tilde{g}' (x,[u],\nu_u)d\hs^{n-1}
= \int_{S_u\cap A} \tilde{g}'' (x,[u],\nu_u)d\hs^{n-1}
\end{align*}
for every $A\in\A$ and every $u\in GSBV^p(A,\R^m)$.

Therefore
\begin{align*}
&f^0(x, \xi) = f_\infty (x, \xi) \text{ for   
$\mathcal{L}^n$-a.e. $x \in \R^n$ and every $\xi \in\R^{m{\times}n}$}, \\
&\int_{S_u\cap A} g^0 (x,[u],\nu_u)d\hs^{n-1} 
= \int_{S_u\cap A} g_\infty(x,[u],\nu_u)d\hs^{n-1} 
\end{align*}
for every $A\in\A$ and every $u\in GSBV^p(A,\R^m)$. By the definition of $E_\infty$ this
implies \eqref{E0=Ehom}, and hence \eqref{Gamma L0}.

The same arguments also give \eqref{Gamma Lp}.
\end{proof}

The proof of Theorem \ref{G-convE2} follows by similar arguments.

\begin{proof}[Proof of Theorem \ref{G-convE2}] 
Let $D$ be a countable subset of $(0,+\infty)$ with $0\in\overline D$, and for every $\e\in D$ let  $(E_k^{\e,p})$ be the perturbed functionals defined in \eqref{Eepsdelta}. By Theorem \ref{Gamma e}, using a diagonal argument, we can obtain a subsequence $(E^{\e,p}_{k_j})$ and a functional $\tilde E^{\e,p}$ such that for every $\e\in D$ and every $A \in \A$ the subsequence $E_{k_j}^{\e,p}(\cdot,A)$ $\Gamma$-converges to  $\tilde E^{\e,p}(\cdot,A)$. Let $\tilde f^{\e,p}$ and $\tilde g^{\e,p}$ be the functions defined by \eqref{f^eta} and \eqref{g^eta} for $\tilde E^{\e,p}$, and let $\tilde f^0$, $\tilde g^0$  be defined as in Theorem~\ref{perturb}. Then by Theorem \ref{6.3}  
$$
\tilde f^0(x,\xi) = \tilde f'(x,\xi) = \tilde f''(x,\xi)
$$
for $\mathcal{L}^n$-a.e.\ $x\in\R^n$ and every $\xi\in \R^{m{\times}n}$, and 
$$
\tilde g^0(x,[u](x),\nu_u(x)) = \tilde g'(x,[u](x),\nu_u(x)) = \tilde g''(x,[u](x),\nu_u(x))  
$$
for every $A\in\A$, for every $u\in GSBV^p(A,\R^m)$, and for $\hs^{n-1}$-a.e. $x\in S_u$, where $\tilde f'$, $\tilde f''$, $\tilde g'$, and $\tilde g''$ are defined by \eqref{f'}-\eqref{g''}, relative to the subsequence $(E_{k_j})$. By Theorem \ref{G-convE} we then conclude that 
$E_{k_j}(\cdot, A)$ $\Gamma$-converge in $L^0(\R^n,\R^m)$, as $j\to +\infty$, to the functional 
$$
\int_A \tilde f^0 (x,\nabla u)\,dx+ \int_{S_u\cap A} \tilde g^0 (x,[u],\nu_u)\,d\mathcal H^{n-1}
$$
for every $A\in \A$. Since $E_k(\cdot,A)$ $\Gamma$-converge to $E_\infty (\cdot,A)$ by assumption, and hence so does $E_{k_j}$, we conclude that for  $\mathcal{L}^n$-a.e. $x \in \R^n$ we have
\begin{equation*}
f_\infty (x, \xi) = \tilde f' (x, \xi) = \tilde f'' (x, \xi)\quad\hbox{for every }\xi \in\R^{m{\times}n},
\end{equation*}
and 
$$
g_\infty(x,[u](x),\nu_u(x)) = \tilde g'(x,[u](x),\nu_u(x)) = \tilde g''(x,[u](x),\nu_u(x)),
$$
for every $A\in\A$, for every $u\in GSBV^p(A,\R^m)$, and for $\hs^{n-1}$-a.e. $x\in S_u$.
\end{proof}

We now show that Theorem \ref{G-convE} implies the convergence of the solutions to some minimisation problems involving  $E_k$. Other minimisation problems can be treated in a similar way.

\begin{cor}[Convergence of minimisers]\label{conv min}
Under the hypotheses of Theorem~\ref{thm:joint}, assume that conditions \eqref{Gamma L0} and \eqref{Gamma Lp} of Theorem \ref{G-convE} are satisfied for some $A\in\A$, and let $h\in L^p(A,\R^m)$.
Then
\begin{equation}\label{conv minima}
\inf_{v\in L^p(A,\R^m)} \Big( E^p_k(v,A)+\|v-h\|_{L^p(A,\R^m)}^p\Big)\longrightarrow 
\min_{v\in L^p(A,\R^m)} \Big( E^p_\infty(v,A)+\|v-h\|_{L^p(A,\R^m)}^p\Big)
\end{equation}
as $k\to+\infty$. Moreover, if $(u_k)$ is a sequence in $L^p(A,\R^m)$ such that
\begin{equation}\label{min-pb_k-g}
E^p_k(u_k,A)+\|u_k-h\|_{L^p(A,\R^m)}^p
\le \inf_{v\in L^p(A,\R^m)} \Big( E^p_k(v,A)+\|v-h\|_{L^p(A,\R^m)}^p\Big)+\e_k
\end{equation}
for some $\e_k\to 0+$, then there exists a subsequence of  $(u_k)$ which converges in $L^p(A,\R^m)$ to a solution of the minimisation problem
\begin{equation}\label{min Lp}
\min_{v\in L^p(A,\R^m)} \Big( E^p_\infty(v,A)+\|v-h\|_{L^p(A,\R^m)}^p\Big).
\end{equation}
\end{cor}

\begin{proof} Let us fix a sequence $(\e_k)$ of positive numbers, with  $\e_k\to 0+$, and let $(u_k)$  be a sequence in $L^p(A,\R^m)$ satisfying \eqref{min-pb_k-g}.  By the lower bounds $(f3)$ and $(g5)$ we have that $u_k\in GSBV^p(A,\R^m)$ and we can apply \cite[Theorem 4.36]{AFP} to deduce that there exist a subsequence of  $(u_k)$, not relabelled, and a function $u \in GSBV^p(A,\R^m)$ such that $u_k \to u $ in $L^0(A,\R^m)$ and $\mathcal{L}^n$-a.e.\ in $A$.   
Hence by the Fatou Lemma we deduce that
\begin{equation}\label{mp-1}
\|u -h\|^p_{L^p(A,\R^m)} \leq \liminf_{k\to +\infty}\|u_k-h\|^p_{L^p(A,\R^m)}.
\end{equation}
This inequality, combined with the fact that \eqref{min-pb_k-g} also ensures that $\sup_k\|u_k\|_{L^p(A,\R^m)}<+\infty$, immediately gives $u \in L^p(A,\R^m)$. 

Let us extend $u_k$ by setting $u_k=u$ on $\R^n\setminus A$.
Since $E_k(\cdot,A)$ $\Gamma$-converge to $E_\infty (\cdot,A)$ in $L^0(\R^n,\R^m)$, we have
\begin{equation*}
E^p_\infty(u ,A) = E_\infty(u ,A)\leq \liminf_{k\to+\infty}E_k(u_k,A)= \liminf_{k\to+\infty}E^p_k(u_k,A). 
\end{equation*}
This inequality, together with \eqref{min-pb_k-g} and \eqref{mp-1}, gives
\begin{align}\nonumber
E^p_\infty (u ,A)+\|u &-h\|^p_{L^p(A,\R^m)} \le \liminf_{k\to +\infty} \Big(E^p_k(u_k,A)+\|u_k-h\|^p_{L^p(A,\R^m)} \Big)
\\\label{mp:lim-inf}
&= \liminf_{k\to +\infty} \inf_{v\in L^p(A,\R^m)}\Big(E^p_k(v,A)+\|v-h\|^p_{L^p(A,\R^m)} \Big).
\end{align}

Let us fix $w\in L^p(A,\R^m)$, that we can extend to a function $w\in L^p_{\mathrm{loc}}(\R^n,\R^m)$. By \eqref{Gamma Lp} we can find a sequence $(w_k)$ in  $L^p_{\mathrm{loc}}(\R^n,\R^m)$ such that
\begin{equation*}
w_k \to w \quad \text{in}\;  L^p_{\mathrm{loc}}(\R^n,\R^m) \quad \text{and}\quad \lim_{k\to +\infty} E^p_k(w_k,A)=E^p_\infty (w,A),
\end{equation*}
hence
\begin{align}\nonumber
 \limsup_{k\to +\infty} &\inf_{v\in L^p(A,\R^m)}\Big(E^p_k(v,A)+\|v-h\|^p_{L^p(A,\R^m)} \Big)\le
\\
&\le  \lim_{k\to +\infty} \Big(E^p_k(w_k,A)+\|w_k-h\|^p_{L^p(A,\R^m)} \Big)
= E^p_\infty(w ,A)+\|w -h\|^p_{L^p(A,\R^m)}.\label{recovery-Lp}
\end{align}
Gathering \eqref{mp:lim-inf} and \eqref{recovery-Lp} gives
\begin{align*}
E^p_\infty (u& ,A)+\|u -h\|^p_{L^p(A,\R^m)} \le\liminf_{k\to +\infty} \inf_{v\in L^p(A,\R^m)}\Big(E^p_k(v,A)+\|v-h\|^p_{L^p(A,\R^m)} \Big)
\\
&\le  \limsup_{k\to +\infty} \inf_{v\in L^p(A,\R^m)}\Big(E^p_k(v,A)+\|v-h\|^p_{L^p(A,\R^m)} \Big)
\le E^p_\infty(w ,A)+\|w -h\|^p_{L^p(A,\R^m)}.
\end{align*}
Since this holds for every $w\in L^p(A,\R^m)$, we deduce that $u$ is a solution of the minimisation problem \eqref{min Lp}. 

Taking $w=u$ in the previous chain of inequalities gives \eqref{conv minima} for the subsequence selected at the beginning of the proof. Since the limit does not depend on the subsequence, \eqref{conv minima} holds for the whole sequence $(E^p_k)$.
\end{proof}

\section{Proof of Theorem~\ref{6.3} \textnormal{(a)} and \textnormal{(b)}}\label{proof volume}

We start by proving the inequality $f^0\le f'$.

\begin{proof}[Proof of Theorem~\ref{6.3} (a)]
Fix  $x\in \R^n$, $\xi\in\R^{m{\times}n}$, $\rho>0$, and $\e\in D \cap (0,1)$, where $D$ is as in Theorem~\ref{perturb}. 
By \eqref{emme} for every $k$ there exists $v_k\in L^0(\R^n,\R^m)$, with $v_k|_{Q_\rho(x)}\in W^{1,p}(Q_\rho(x),\R^m)$, such that $v_k- \ell_{\xi} \in W^{1,p}_0(Q_\rho(x),\R^m)$ and 
\begin{equation}\label{q_min-v}
E_k^{\e,p}(v_k,Q_\rho(x))=F_k(v_k,Q_\rho(x)) \leq m^{1,p}_{F_k} (\ell_{\xi}, Q_\rho(x)) + \e\, \rho^{n}.
\end{equation}

Let $k_j$ be a strictly increasing sequence of integers such that
$$
\lim_{j\to+\infty}E_{k_j}^{\e,p}(v_{k_j},Q_\rho(x))=\liminf_{k\to+\infty}E_k^{\e,p}(v_k,Q_\rho(x)).
$$
From $(f3)$, $(f4)$, and \eqref{q_min-v} we obtain
$$
c_1 \|\nabla v_k \|^p_{L^p (Q_\rho(x),\,\R^{m{\times}n})} \leq (c_2 (1 + |\xi|^p) + \e ) \rho^n.
$$
By the Poincar\'e Inequality we deduce that the sequence $(v_k)$ is bounded in 
$W^{1,p}(Q_\rho(x),\R^m)$. Therefore, up to a subsequence, 
$v_k \rightharpoonup v$ weakly in $W^{1,p}(Q_\rho(x),\R^m)$ 
for some $v \in W^{1,p}(Q_\rho(x),\R^m)$ such that $v- \ell_{\xi} \in W^{1,p}_0(Q_\rho(x),\R^m)$. 
Let $w_k$, $w\in W^{1,p}_{\mathrm{loc}}(\R^n,\R^m)$ be defined by
\begin{equation}\label{w7}
w_k:=\begin{cases}
v_k&\hbox{ in }Q_\rho(x),
\\
\ell_\xi &\hbox{ in } \R^n\setminus Q_\rho(x),
\end{cases}
\qquad\hbox{and}\qquad
w:=\begin{cases}
v&\hbox{ in }Q_\rho(x),
\\
\ell_\xi &\hbox{ in } \R^n\setminus Q_\rho(x).
\end{cases}
\end{equation}
By the Rellich Theorem $w_k\to w$ in $L^p_{\mathrm{loc}}(\R^n,\R^m)$, hence
$$
E^{\e,p} (w,Q_\rho(x))\le  \liminf_{k \to +\infty} E_k^{\e,p}(w_k,Q_\rho(x))=
 \liminf_{k \to +\infty} E_k^{\e,p}(v_k,Q_\rho(x))
$$
by the $\Gamma$-convergence of $E_k^{\e,p}(\cdot,Q_\rho(x))$ 
to $E^{\e,p}(\cdot,Q_\rho(x))$. Using this inequality, together with $(f4)$, \eqref{q_min-v}, and \eqref{w7}, we get
\begin{align*}
m_{E^{\e,p}} (\ell_{\xi}, Q_{(1+\e)\rho}(x)) &\le E^{\e,p} (w,Q_\rho(x))+
c_2(1+|\xi|^p)((1+\e)^n-1)\rho^n
\\
&\le \liminf_{k \to +\infty} E_k^{\e,p}(v_k,Q_\rho(x))+ \e n  2^{n-1}c_2(1+|\xi|^p)\rho^n
\\
&\le \liminf_{k \to+ \infty} {}m^{1,p}_{F_k} (\ell_{\xi}, Q_\rho(x)) + \e C_\xi \rho^{n}
\end{align*}
where $C_\xi:=1+ n 2^{n-1}c_2(1+|\xi|^p)$.
Dividing by $\rho^n$ and taking the limsup as $\rho \to 0+$, we obtain from \eqref{f'} and \eqref{f^eta}
\begin{align*}
(1+\e)^n f^{\e,p}(x,\xi)&= \limsup_{\rho \to 0+}  \frac{
m_{E^{\e,p}} (\ell_{\xi}, Q_{(1+\e)\rho}(x))}{\rho^n}
\\
&\le \limsup_{\rho \to 0+} \liminf_{k \to +\infty} \frac{m^{1,p}_{F_k} (\ell_{\xi}, Q_\rho(x))}{\rho^n} + \e C_\xi 
=f' (x, \xi) + \e C_\xi .
\end{align*}
Letting $\e \to 0+$, from \eqref{def f} we obtain that $f^0(x, \xi)\le f'(x,\xi)$.
\end{proof}

We now prove $(b)$. Namely, we show that $f''  \le f^0$.

\begin{proof}[Proof of Theorem~\ref{6.3} (b)]
In view of Lemma~\ref{f'' g' in F and G} we have $f''\in \mathcal{F}$, while by 
Theorem~\ref{perturb} $f^0\in\mathcal{F}$, hence in particular 
$f^0$ and $f''$ are continuous with respect to $\xi$  by $(f2)$. 
Therefore it is enough to prove that for every $\xi\in\R^{m \times n}$ we have $
 f''(x,\xi)\le f^0(x,\xi)
$ for $\mathcal{L}^n$-a.e.\ $x\in\R^n$.

We may assume that the set $D$ considered in Theorem~\ref{perturb} is contained in $(0,1)$. 
Let us fix $\xi\in\R^{m\times n}$.
Since for every $\e\in D$
\begin{equation} \label{ineqperturb}
E^{\e,p} (\ell_\xi, A)=\int_A f^{\e,p} (x,\xi)\dx \quad\hbox{for every }A\in\A,
\end{equation}
by the Lebesgue Differentiation Theorem  for every $\e\in D$ and for $\mathcal{L}^n$-a.e.\ $x\in\R^n$ we have
\begin{equation}\label{estimate:0v}
\lim_{\rho \to 0+} \frac{E^{\e,p} (\ell_\xi, Q_\rho(x))}{\rho^{n}} = f^{\e,p} (x,\xi) \le c_2(1+|\xi|^p),
\end{equation}
where the last inequality follows from the fact that $f^{\e,p}\in\mathcal{F}$ by Theorem \ref{Gamma e}.

Let $x \in\R^n$ be fixed and such that \eqref{estimate:0v} holds for every $\e\in D$. 
It follows that for every $\e\in D$ there exists $\rho_0(\e)\in (0, 1)$ such that
\begin{equation}\label{estimate:0v2}
\frac{E^{\e,p} (\ell_\xi, Q_\rho(x))}{\rho^{n}}  \le c_2(2+|\xi|^p)
\end{equation}
for every $0<\rho<\rho_0(\e)$.

Let $\e\in D$ be fixed. Since $E^{\e,p}_k(\cdot,Q(x))$ $\Gamma$-converge to $E^{\e,p}(\cdot,Q(x))$ in $L^p_{\mathrm{loc}}(\R^n,\R^m)$,
there exists $(u_k)\subset  L^p_{\mathrm{loc}}(\R^n,\R^m)$, with $u_k|_{Q(x)}\in 
SBV^p(Q(x),\R^m)\cap L^p(Q(x),\R^m)$, such that
\begin{equation}\label{recovery}
u_k \to \ell_\xi \; \text{ in }\; L^p_{\mathrm{loc}}(\R^n,\R^m)\quad\text{and} \quad \lim_{k\to +\infty}E^{\e,p}_k(u_k,Q(x))=E^{\e,p}(\ell_\xi,Q(x)).
\end{equation}
By \eqref{ineqperturb} we have $E^{\e,p} (\ell_\xi, Q(x))=E^{\e,p} (\ell_\xi, Q_\rho(x))
+ E^{\e,p} (\ell_\xi, Q(x)\setminus\overline{Q}_\rho(x))$ for all $\rho\in(0,1)$. By $\Gamma$-convergence we have also 
\begin{align*}
\liminf_{k\to +\infty}E^{\e,p}_k (u_k,Q_\rho(x))&\ge E^{\e,p}(\ell_\xi,Q_\rho(x))
\\
\liminf_{k\to +\infty}E^{\e,p}_k (u_k,Q(x)\setminus \overline{Q}_\rho(x))&\ge E^{\e,p} (\ell_\xi,Q(x)\setminus \overline{Q}_\rho(x)).
\end{align*}
From these inequalities and from \eqref{recovery} it follows that
$$
\lim_{k\to +\infty}E^{\e,p}_k(u_k,Q_\rho(x))= E^{\e,p} (\ell_\xi,Q_\rho(x)).
$$
This yields the existence of $k_0(\e,\rho)>0$ such that $
|E^{\e,p}(\ell_\xi, Q_\rho(x)) - E^{\e,p}_k(u_k, Q_\rho(x))| <\e \rho^n
$
whenever $k\ge k_0(\e,\rho)$,
hence
\begin{align}\label{estimate:1v}
\frac{E^{\e,p}_k(u_k, Q_\rho(x))}{\rho^{n}}  < \frac{E^{\e,p}(\ell_\xi, Q_\rho(x))}{\rho^{n}}+\e. 
\end{align}

In the remaining part of the proof we modify the sequence $(u_k)$ to construct a competitor for the minimisation problem $m_{F_k}^{1,p}(\ell_\xi, Q_\rho(x))$, which appears in the definition of $f''$. To this end, for every $y\in Q:=Q(0)$ we set 
\begin{align*} 
&u_k^\rho (y):= \frac{u_k(x+\rho y) - u_k(x)}{\rho},
\\
&f_k^\rho(y,\cdot):=f_k(x+\rho y,\cdot).
\end{align*}
Note that $u_k^\rho\in SBV^p(Q,\R^m)\cap L^p(Q,\R^m)$ and $f_k^\rho\in\mathcal{F}$.

We fix $\lambda> |\xi| \sqrt{n}/2$ and $h$, $\alpha$, $\psi_1,\dots,\psi_h$, and $\mu$ as in Lemma \ref{estimate truncations} with $\eta=\e$.  
By \eqref{estimate for F} for every $k$ there exists $i_k\in\{1,\dots,h\}$ such that
\begin{equation} \label{estimate F_k^rho}
F_k^\rho(\psi_{i_k}\!(u^\rho_k),Q)\le (1+\e) F_k^\rho(u^\rho_k,Q)+c_2\mathcal{L}^n(Q\cap\{|u^\rho_k|\ge\lambda\}),
\end{equation}
where $F_k^\rho$ is defined as in \eqref{Effe}, with $f$ replaced by $f_k^\rho$.

We define
\begin{equation} \label{vtroncata}
 v_k^\rho:= \psi_{i_k}\!(u^\rho_k).
\end{equation}
Then $ v_k^\rho=u^\rho_k$ in $Q\cap\{|u^\rho_k|<\lambda\}$ and $|v_k^\rho|\le \mu$ in $Q$. Since $u_k\to \ell_\xi$ in $L^p(Q_\rho(x),\R^m)$, we have $u^\rho_k\to \ell_\xi$ in $L^p(Q,\R^m)$, and since $| \ell_\xi  | \leq |\xi| \sqrt{n}/2< \lambda$ in $Q$, it follows that $v_k^\rho \to \ell_\xi$ in $L^p(Q,\R^m)$ and that  $\mathcal{L}^n(Q\cap\{|u^\rho_k|\ge\lambda\})\to0$ as $k\to+\infty$. Therefore, there exist $k_1(\e,\rho)\ge k_0(\e,\rho)$ such that
\begin{equation}\label{serve}
\| v_k^\rho - \ell_{\xi} \|_{L^p(Q,\R^m)}<\rho \quad \text{ and } 
\quad \mathcal{L}^n(Q\cap\{|u^\rho_k|\ge\lambda\}) < \rho
\quad  \text{ for every }k\ge k_1(\e,\rho).
\end{equation}

Using $(f3)$, $(g5)$, \eqref{estimate F_k^rho}-\eqref{serve}, and a change of variables we obtain the two following estimates
\begin{align}
& c_1\int_Q |\nabla v^\rho_k(y)|^pdy \le \int_Q f_k(x+\rho y, \nabla v^\rho_k(y))dy \le \frac{1+\e}{\rho^n} \int_{Q_\rho(x)} \!\!\!\!  \!\!\!\!  f_k( y, \nabla u_k(y))dy+  c_2 \rho,\label{servedavvero} \\
&\frac{c_4}{\rho} \hs^{n-1} (S_{v_k^{\rho}} \cap Q)  
\leq \frac{c_4}{\rho^n} \hs^{n-1} (S_{u_k} \cap Q_\rho(x))
\leq \frac{1}{\rho^n}\int_{S_{u_k}\cap Q_\rho(x)}  \!\!\!\!  \!\!\!\!  \!\!\!\!  \!\!\!\!  \!\!\!\!   g^\e_k (y,[u_k], \nu_{u_k} ) d\mathcal{H}^{n-1}, \label{serve:1}
\end{align}
for every $k\ge k_1(\e,\rho)$,  where $g^\e_k$ is defined in \eqref{gek}.

From \eqref{estimate:0v2}, \eqref{estimate:1v}, and \eqref{serve:1}, we deduce that there exists $M>0$, independent of $k$, $\rho$, and $\e$, such that
\begin{equation}\label{f:stime}
\|\nabla v_k^\rho\|_{L^p(Q,\R^{m\times n})}\le M \quad \textrm{and} \quad 
\hs^{n-1} (S_{v_k^{\rho}} \cap Q) \le M \rho,
\end{equation}
whenever $\e\in D$, $0<\rho<\rho_0(\e)$, and $k\ge k_1(\e,\rho)$.
Since $|[v_k^{\rho}]| \leq 2\mu$
$\hs^{n-1}$-a.e. on $S_{v_k^{\rho}}$ by \eqref{vtroncata}, from \eqref{f:stime} we obtain also that
\begin{equation}\label{f:stimenuove}
|D^sv_k^\rho|(Q)\le 2\mu M \rho.
\end{equation}
We now regularise $v_k^\rho$ in order to obtain a function $w_k^\rho\in W^{1,p}(Q,\R^m)$ such that
$$ 
\int_{Q} f_k (x +\rho y,\nabla w_k^\rho (y) ) dy \le \int_{Q} f_k (x +\rho y,\nabla v_k^\rho (y)) dy + \e
$$
for a suitable choice of $\rho$ and $k$.
We follow the procedure introduced in \cite[Lemma 2.1]{Larsen}, which we now illustrate in detail for the readers' convenience.
\medskip

\noindent
\textit{Step 1: Regularisation of $v_k^\rho$.} Let $t>0$; we define the sets
\begin{eqnarray*}
&\ds R_k^t:= \Big\{y \in Q: \frac{|Dv_k^\rho|(\overline{B_r(y)})}{\mathcal{L}^n(B_r(y))} \leq t\ \hbox{ for every } r>0 \hbox{ with }\overline{B_r(y)}\subset Q\Big\},
\\
&\ds S_k^t:= S_{v_k^\rho}\cup \Big\{y \in Q: |\nabla v_k^\rho(y)|\geq \frac{t}{2}\Big\}.
\end{eqnarray*}
For every $k$, by the Vitali Covering Lemma (see, e.g., \cite[Section 1.5.1]{EvansGariepy}), there exists a sequence of disjoint closed balls $\overline {B_{r_j}\!(y_j)}\subset Q$, with centres $y_j$ in $Q\setminus R_k^t$, such that
\begin{equation}\label{6.10bis}
\frac{|Dv_k^\rho|(\overline{B_{r_j}\!(y_j)})}{\mathcal{L}^n(B_{r_j}\!(y_j))}>t \quad \hbox{for every } j \hbox{ and }
Q\setminus R_k^t \subset \bigcup_{j=1}^\infty \overline{B_{5r_j}\!(y_j)}.
\end{equation}
Hence 
\begin{align}\label{f:22}
t \mathcal{L}^n \Big( \bigcup_{j=1}^\infty B_{r_j}\!(y_j) \Big) = t \sum_{j=1}^\infty \mathcal{L}^n (B_{r_j}\!(y_j)) < \sum_{j=1}^\infty |Dv_k^\rho|\big( \overline{B_{r_j}\!(y_j)}\big)
=  |Dv_k^\rho|\Big( \bigcup_{j=1}^\infty \overline{B_{r_j}\!(y_j)} \Big).
\end{align}
On the other hand 
\begin{align}\label{f:star}
 |Dv_k^\rho|\Big( \bigcup_{j=1}^\infty \overline{B_{r_j}\!(y_j)} \Big) =  |Dv_k^\rho| \Big(S_k^t \cap \bigcup_{j=1}^\infty \overline{B_{r_j}\!(y_j)} \Big) + 
 |Dv_k^\rho| \Big((Q\setminus  S_k^t) \cap \bigcup_{j=1}^\infty \overline{B_{r_j}\!(y_j)} \Big).
\end{align}
We are going to estimate the two terms in the right-hand side of \eqref{f:star} separately.
We observe that 
\begin{equation}\label{f:smile}
 |Dv_k^\rho| \Big((Q\setminus  S_k^t) \cap \bigcup_{j=1}^\infty \overline{B_{r_j}\!(y_j)} \Big) = \int_{(Q\setminus  S_k^t) \cap \cup_{j=1}^\infty B_{r_j}\!(y_j)}|\nabla v_k^\rho| dy 
\leq \frac{t}2\mathcal{L}^n\Big(\bigcup_{j=1}^\infty B_{r_j}\!(y_j) \Big).
\end{equation}
By \eqref{f:22} we have, using also \eqref{f:star} and  \eqref{f:smile},
$$
 t \mathcal{L}^n\Big(\bigcup_{j=1}^\infty B_{r_j}\!(y_j) \Big) < 
 |Dv_k^\rho|\Big( S_k^t\cap \bigcup_{j=1}^\infty \overline{B_{r_j}\!(y_j)} \Big) +
 \frac{t}2
 \mathcal{L}^n\Big(\bigcup_{j=1}^\infty B_{r_j}\!(y_j) \Big).
$$
This implies that 
\begin{equation}\label{f:uno}
\mathcal{L}^n\Big(\bigcup_{j=1}^\infty B_{r_j}\!(y_j) \Big) \leq \frac{2}t |Dv_k^\rho|\Big( S_k^t\cap \bigcup_{j=1}^\infty \overline{B_{r_j}\!(y_j)} \Big).
\end{equation}
By \eqref{6.10bis} and \eqref{f:uno} we have
\begin{align}\label{f:due-b}
\mathcal{L}^n(Q\setminus R_k^t) &\leq \sum_{j=1}^\infty \mathcal{L}^n(B_{5r_j}\!(y_j)) = 5^n\sum_{j=1}^\infty \mathcal{L}^n(B_{r_j}\!(y_j)) = 5^n \mathcal{L}^n\Big(\bigcup_{j=1}^\infty B_{r_j}\!(y_j)\Big)\nonumber\\
&\leq \frac{2{\cdot}5^n}{t} |Dv_k^\rho|\Big( S_k^t\cap \bigcup_{j=1}^\infty \overline{B_{r_j}\!(y_j)} \Big) 
\leq \frac{2{\cdot}5^n}{t} \Big(|D^s v_k^\rho|(Q) + \int_{S_k^t} |\nabla v_k^\rho| dy\Big)\nonumber\\
& \leq \frac{2{\cdot}5^n}{t} \Big(|D^s v_k^\rho|(Q) + \Big(\int_{S_k^t} |\nabla v_k^\rho|^p dy\Big)^{\frac1p}(\mathcal{L}^n(S_k^t))^{1-\frac1p}\Big).
\end{align}
Now, by the definition of $S_k^t$ and by \eqref{f:stime} we have that 
$$
\mathcal{L}^n(S_k^t) \Big(\frac{t}2\Big)^p\leq \int_{S_k^t} |\nabla v_k^\rho|^p dy\leq M^p
$$
whenever $\e\in D$, $0<\rho< \rho_0(\e)$, and $k\ge k_1 (\e,\rho)$.
It then follows that 
$
\mathcal{L}^n(S_k^t)  \leq {2^p M^p}/{t^p}$,
which, combined with \eqref{f:stime} and \eqref{f:due-b}, gives
\begin{align*}
\mathcal{L}^n(Q\setminus R_k^t)  \leq \frac{2{\cdot}5^n}{t} \Big(|D^s v_k^\rho|(Q) + \Big(\int_{S_k^t} |\nabla v_k^\rho|^p dy\Big)^{\frac1p}\frac{2^{p-1} M^{p-1} }{t^{p-1}}\Big)
\leq \frac{2{\cdot}5^n}t |D^s v_k^\rho|(Q) + \frac{2^p5^nM^p}{t^p}.
\end{align*}
Hence we can conclude that 
\begin{equation}\label{f:fc}
t^p \mathcal{L}^n (Q\setminus R_k^t) \leq 2{\cdot}5^n t^{p-1} |D^s v_k^\rho|(Q) + 2^p5^nM^p
\end{equation}
whenever $\e\in D$, $0<\rho< \rho_0(\e)$, and $k\ge k_1 (\e,\rho)$.

Now we choose $t_{k,\rho}>0$ such that 
$
t_{k,\rho}^{p-1} |D^s v_k^\rho|(Q) = 1
$.
By  \eqref{f:stimenuove} this implies
\begin{equation*}
t_{k,\rho}^{p-1}  = \frac{1}{|D^s v_k^\rho|(Q)} \geq \frac{1}{2\mu M\rho},
\end{equation*}
whenever $\e\in D$, $0<\rho< \rho_0(\e)$, and $k\ge k_1 (\e,\rho)$. Then, from \eqref{f:fc} we obtain
\begin{equation*}
t_{k,\rho}^p \mathcal{L}^n (Q\setminus R_k^{t_{k,\rho}}) \leq 2{\cdot}5^n  + 2^p5^nM^p=:M_1,
\end{equation*}
which gives in particular that
\begin{equation}\label{f:cc}
\mathcal{L}^n (Q\setminus R_k^{t_{k,\rho}}) \le 
\frac{M_1}{t_{k,\rho}^p} \le M_2 \rho^q,
\end{equation}
with $q:=p/(p-1)$ and $M_2 :=M_1 (2\mu M)^q$.

By \cite[Section 3.1.1 (Theorem 1) and Section 6.6.2 (Claim \#2 of Theorem 2)]{EvansGariepy} there exist a constant $c_n$, depending only on $n$, and Lipschitz functions  $z_k^\rho$ on $Q$, with Lipschitz constant bounded by $c_nt_{k,\rho}$, such that that $z_k^\rho = v_k^\rho$ $\mathcal{L}^n$-a.e.\ in $R_k^{t_{k,\rho}}$. 
Note that, since $| v_k^\rho | \leq \mu$  $\mathcal{L}^n$-a.e.\ in $Q$, it is not restrictive to assume that $| z_k^\rho | \leq \mu$ in $Q$. 
By \eqref{f:stime} and \eqref{f:cc} we have also
$$
\int_Q | \nabla z_k^\rho|^p dy \leq \int_{R_k^{t_{k,\rho}}} |\nabla v_k^\rho|^p dy + c_n^pt_{k,\rho}^p \mathcal{L}^n(Q\setminus R_k^{t_{k,\rho}}) \leq M^p+ c_n^pM_1.
$$
Therefore the sequence $(z_k^\rho)_k$ is bounded in $W^{1,p}(Q,\R^m)$. 

By \eqref{f''} there exists a decreasing sequence $\rho_j \to 0+$, with $0<\rho_j <\rho_0(\e)$, such that
\begin{equation} \label{limsuprho}
f'' (x, \xi) = \lim_{j \to +\infty} \limsup_{k \to +\infty} \frac{1}{\rho_j^n} 
m^{1,p}_{F_k} (\ell_\xi, Q_{\rho_j} (x)).
\end{equation}
By applying \cite[Lemma 1.2]{FonsPedr} to the double sequence $(z_k^{\rho_j})_{j,k}$
we find a double sequence $(w_k^{\rho_j})_{j,k}$ in  $W^{1,p}(Q,\R^m)$ such that 
$|\nabla w_k^{\rho_j}|^p$ is equi-integrable, uniformly with respect to $j$ and $k$, and 
\begin{equation*}
\mathcal{L}^n(\{w_k^{\rho_j} \neq z_k^{\rho_j} \}) \to 0 \quad \textrm{as } k+j \to +\infty.
\end{equation*}
Note that, since $| z_k^{\rho_j} | \leq \mu$ in $Q$, it is not restrictive 
to assume that $| w_k^{\rho_j} | \leq \mu$ $\mathcal{L}^n$-a.e.\ in $Q$. 
By \eqref{serve} and \eqref{f:cc} these properties imply that for every $j$
there exists $k_2(\e,j)\ge k_1(\e,\rho_j)$ such that for every $k\ge k_2(\e,j)$ we have
\begin{equation}\label{serve2}
\mathcal{L}^n(\{  w_k^{\rho_j} \neq v_k^{\rho_j} \}) \le M_2 \rho^q_j \quad\hbox{and}\quad
\|w_k^{\rho_j} - \ell_{\xi} \|_{L^p(Q,\R^m)}\le \rho_j + 4 \mu M_2^{1/p} \rho_j^{q/p}=:r_j.
\end{equation}
Moreover, 
$$
\int_{Q} f_k (x +\rho_j y,\nabla w_k^{\rho_j} (y) ) dy \leq  
\int_{Q} f_k ( x +\rho_j y ,\nabla v_k^{\rho_j} (y)) dy
+ \int_{\{w_k^{\rho_j} \neq z_k^{\rho_j} \}} f_k (x +\rho_j y,\nabla w_k^{\rho_j} (y)) dy. 
$$
By the equi-integrability of $|\nabla w_k^{\rho_j}|^p$, by the upper bound $(f4)$,
and by \eqref{serve2} we can conclude that for every $\e\in D$ there exists
$j_0 (\e)$, with $\rho_{j_0(\e)} \leq \rho_0(\e)$, such that
$$
\int_{\{w_k^{\rho_j} \neq z_k^{\rho_j} \}} f_k (x +\rho_j y,\nabla w_k^{\rho_j}(y)) dy < \e
$$
for every $j \geq j_0(\e)$ and every $k$, hence
\begin{equation}\label{serve:2} 
\int_{Q} f_k (x +\rho_j y,\nabla w_k^{\rho_j} (y) ) dy \le \int_{Q} f_k (x +\rho_j y,\nabla v_k^{\rho_j} (y)) dy+ \e,
\end{equation}
for every $j \geq j_0(\e)$ and every $k\ge k_2(\e,j)$.

\medskip

\noindent
\textit{Step 2: Attainment of the boundary datum.} We now modify $w_k^{\rho_j}$ 
so that it attains the linear boundary datum $\ell_{\xi}$, which appears in the definition of $f'' (x,\xi)$. To this end, we will apply the Fundamental Estimate to the functionals $F_k^{\rho_j}$ corresponding to the integrands $f_k^{\rho_j}(y,\cdot):=f_k(x+\rho_j y,\cdot)$. Let $Q_{1-\e} := Q_{1-\e}(0)$. By \cite[Theorem 19.1]{DM93} there exists a constant $C_\e>0$ and a \textit{finite} family of cut-off functions 
$(\varphi_i)_{1\le i\le N}\subset C_c^\infty(Q)$, with  $0\leq \varphi_i\leq 1$  in $Q$ and $\varphi_i=1$ in $Q_{1-\e}$, such that
\begin{equation*}
F_k^{\rho_j}(\tilde w_k^{\rho_j}, Q) \leq (1+\e) \big(F_k^{\rho_j}(w_k^{\rho_j}, Q) + F_k^{\rho_j}(\ell_{\xi}, Q\setminus \overline{Q}_{1-\e})\big) + C_\e \| w_k^{\rho_j} - \ell_{\xi}\|^p_{L^p(Q)}+\e,
\end{equation*}
where $\tilde w_k^{\rho_j}:= \varphi_{i_{k,j}} w_k^{\rho_j} + (1-\varphi_{i_{k,j}})\ell_{\xi}$  for a suitable $i_{k,j}\in \{1,\dots,N\}$. Clearly $\tilde w_k^{\rho_j}$ attains the boundary datum $\ell_{\xi}$ in a neighbourhood of $\partial Q$. Since $\mathcal{L}^n(Q\setminus Q_{1-\e}) < n\e$,  by $(f4)$ and \eqref{serve2} it follows that 
\begin{equation}\label{serve:3}
F_k^{\rho_j} (\tilde w_k^{\rho_j}, Q) 
\leq (1+\e) F_k^{\rho_j} (w_k^{\rho_j}, Q) + \e(1+\e) n c_2(1+|\xi|^p)+ C_\e r_j^p+ \e .
\end{equation}
Combining \eqref{estimate:1v}, \eqref{servedavvero}, \eqref{serve:2}, and \eqref{serve:3}, and setting $B_\xi:=7+2nc_2(1+|\xi|^p)$, we have the bound
\begin{align}\label{est:tildew}
\limsup_{k\to +\infty} 
\int_{Q} f_k (x +{\rho_j} y ,\nabla \tilde w_k^{\rho_j}(y) ) dy\le 
(1+\e)^2\frac{E^{\e,p} (\ell_\xi, Q_{\rho_j}(x ))}{{\rho_j}^{n}} + B_\xi \e + C_\e r_j^p+ 2 c_2\rho_j,
\end{align}
whenever $\e\in D$, $j\ge j_0(\e)$, and $k\ge k_2(\e,j)$.

Finally, we perform a change of variables in order to relate the left-hand side of \eqref{est:tildew} with the minimisation problems on $Q_{\rho_j}(x )$, appearing in \eqref{limsuprho}. For $y \in Q_{\rho_j}(x)$, define
$$
\tilde v_k^{\rho_j}(y):= {\rho_j} \,\tilde w_k^{\rho_j} \Big(\frac{y-x }{{\rho_j}}\Big) + \ell_{\xi}(x ).
$$
Clearly $\tilde v_k^{\rho_j}\in W^{1,p}(Q_{\rho_j}(x ))$, $\tilde v_k^{\rho_j} = \ell_{\xi}$ 
in a neighbourhood of  $\partial Q_{\rho_j}(x )$, and 
\begin{align*}
&\int_{Q} f_k (x +{\rho_j} y,\nabla \tilde w_k^{\rho_j}(y)) dy\, =\, 
\frac{1}{{\rho_j}^n}\int_{ Q_{\rho_j}(x )} f_k \left(y,\nabla \tilde v_k^{\rho_j}(y)\right) dy
\geq  \frac{1}{{\rho_j}^n} m^{1,p}_{F_k}\big(\ell_{\xi}, Q_{\rho_j}(x )\big)\, .
\end{align*}
Therefore, from \eqref{est:tildew} we conclude that
\begin{align*}
\limsup_{k \to +\infty} \frac{1}{{\rho_j}^n} m^{1,p}_{F_k}\big(\ell_{\xi}, Q_{\rho_j}(x )\big)\le 
(1+\e)^2\frac{E^{\e,p} (\ell_\xi, Q_{\rho_j}(x ))}{{\rho_j}^{n}} + B_\xi \e + C_\e r_j^p+  2 c_2\rho_j.
\end{align*}
Since $r_j\to0$ by \eqref{serve2}, taking the limit as $j \to +\infty$, by \eqref{estimate:0v} and  \eqref{limsuprho} we obtain the estimate
$$
f'' (x,\xi)  \le (1+\e)^2 f^{\e,p}(x ,\xi) + B_\xi \e
$$
for every $\e\in D$. Taking the limit as $\e\to0+$, from \eqref{def f} we obtain $f'' (x,\xi)  \le  f^0(x ,\xi)$.
\end{proof}

\section{Proof of Theorem~\ref{6.3} \textnormal{(c)} and \textnormal{(d)}}\label{estimate surface}

We start by proving the inequality $g^0 \le g' $.

\begin{proof}[Proof of Theorem~\ref{6.3} (c)]
Fix $x\in \R^n$, $\zeta\in \R^m_0$, $\nu \in \Sph^{n-1}$, $\rho>0$, and $\e\in D\cap (0,1)$, where $D$ is as in Theorem~\ref{perturb}. By the definition of $m^{\mathrm{pc}}_{G_k}$, for every $k$ there exists $u_k \in L^0(\R^n,\R^m)$, with $u_k|_{Q^{\nu }_\rho( x)}\in SBV_{\mathrm{pc}}(Q^{\nu }_\rho( x), \R^m)$, such that $u_k=u_{x,\zeta ,\nu }$ in a neighbourhood of $\partial  Q^{\nu }_\rho( x)$ and
\begin{equation}\label{q_min}
G_k(u_k,Q^{\nu }_\rho( x)) \leq m^{\mathrm{pc}}_{G_k}(u_{x,\zeta ,\nu },Q^{\nu }_\rho( x))+ \e\, \rho^{n-1}.
\end{equation}
Now fix $\lambda> |\zeta|$ and $h$, $\alpha$, $\psi_1,\dots,\psi_h$, and $\mu$ as in Lemma \ref{estimate truncations}. 
Then by \eqref{estimate for E} for every $k$ there exists $i_k\in\{1,\dots,h\}$ such that
\begin{equation*}
E_k(\psi_{i_k}\!(u_k),Q^{\nu }_\rho( x))\le (1+\e) E_k(u_k,Q^{\nu }_\rho( x))+c_2\mathcal{L}^n(Q^{\nu }_\rho( x)\cap\{|u_k|\ge\lambda\}).
\end{equation*}

By \eqref{bound mu} and \eqref{equality lambda} we have $\psi_{i_k}\!(u_k)=u_{x,\zeta ,\nu }$ in a neighbourhood of $\partial  Q^{\nu }_\rho( x)$ and $|\psi_{i_k}\!(u_k)|\le \mu$ in $\R^n$. Moreover, the chain rule gives  $\nabla (\psi_{i_k}(u_k))=0$ 
$\mathcal{L}^n$-a.e.\ in $Q^{\nu }_\rho( x)$. Therefore the functions $v_k$ defined as
\begin{equation} \label{def vk 10}
v_k:= \begin{cases}
\psi_{i_k}\!(u_k) & \text{in }\; Q^{\nu }_\rho( x)
\cr
u_{x,\zeta ,\nu } & \text{in }\; \R^n\setminus Q^{\nu }_\rho( x)
\end{cases}
\end{equation}
satisfy $v_k|_A\in SBV_{\mathrm{pc}}(A, \R^m)$ for every $A\in\A$. 

By definition we also have
\begin{equation}\label{vk bounded 2}
|v_k|\le \mu\quad\hbox{in }\R^n.
\end{equation}

Since $\nu_{v_k} = \nu_{u_k}$ and, by \eqref{lip 1}, $|[v_k]|\leq|[u_k]|$ $\mathcal{H}^{n-1}$-a.e. in $S_{v_k} \cap Q^{\nu }_\rho( x) \subset S_{u_k} \cap Q^{\nu }_\rho( x)$, by using $(g3)$, $(g5)$, and  $(g6)$ we get
\begin{equation}\nonumber
c_4\mathcal{H}^{n-1}(S_{v_k}\cap Q_\rho^\nu(x ))  
\leq G_k(v_k,Q^{\nu }_\rho(x ))\le c_3  G_k(u_k,Q^{\nu }_\rho(x )).
\end{equation}
Therefore, appealing to \eqref{q_min} we conclude that for every $k$
\begin{equation}\label{c:boundedness}
\mathcal{H}^{n-1}(S_{v_k}\cap Q_\rho^\nu(x))  \leq M_\zeta \rho^{n-1},
\end{equation}
where $M_\zeta:= c_3(c_5(1+ |\zeta|)+ 1)/c_4$.

Since $v_k \in SBV_{\mathrm{pc}}(Q^{\nu }_\rho(x), \R^m)$, by combining \eqref{vk bounded 2} and \eqref{c:boundedness} we can invoke \cite[Theorem 4.8]{AFP} to deduce the existence of a function $v\in SBV_{\mathrm{pc}}(Q_\rho^\nu(x),\R^m)\cap L^\infty(Q_\rho^\nu(x),\R^m)$ and a subsequence, not relabelled, such that
$v_k \to v$ in $L^0(Q_\rho^\nu(x),\R^m)$.
We extend $v$ to $\R^n$ by setting $v=u_{x,\zeta ,\nu }$ in $\R^n\setminus Q_\rho^\nu(x)$ and observe that $v|_A\in SBV_{\mathrm{pc}}(A, \R^m)$ for every $A\in\A$. 
By the definitions of $v_k$ and $v$ and by \eqref{vk bounded 2}, the convergence in $L^0(Q_\rho^\nu(x),\R^m)$ 
also implies that 
\begin{align}
& v_k \to v\; \text{ in }\; L^p_{\mathrm{loc}}(\R^n,\R^m), 
\label{vk to v Lp}
\\
\label{bound v}
&|v|\le \mu\ \;\mathcal{L}^n\text{-a.e.\ in }\; \R^n.
\end{align}

Since $v|_{Q_{(1+\e)\rho}^\nu(x)}\in SBV_{\mathrm{pc}}(Q_{(1+\e)\rho}^\nu(x),\R^m)$ and $v=u_{x,\zeta ,\nu }$ in $Q_{(1+\e)\rho}^\nu(x)\setminus Q_\rho^\nu(x)$, we have
\begin{equation}\label{m<=E(v)}
m_{E^{\e,p}}(u_{x,\zeta ,\nu },Q_{(1+\e)\rho}^{\nu }(x ))\le E^{\e,p}(v,Q^{\nu }_{(1+\e)\rho}(x)).
\end{equation}
Using the  $\Gamma$-convergence of $E_k^{\e,p}(\cdot, Q_{(1+\e)\rho}^\nu(x))$   to $E^{\e,p}(\cdot,Q_{(1+\e)\rho}^\nu(x))$ in $L^p_{\mathrm{loc}}(\R^n,\R^m)$, we deduce from \eqref{vk to v Lp} that
\begin{equation*}
E^{\e,p}(v,Q_{(1+\e)\rho}^\nu(x))\le  \liminf_{k\to+\infty} {} E_k^{\e,p}(v_k,Q_{(1+\e)\rho}^\nu(x)).
\end{equation*}

Since  $v_k=u_{x,\zeta ,\nu }$ in a neighbourhood of $\partial  Q^{\nu }_\rho( x)$, we have $\hs^{n-1}(S_{v_k}\cap \partial  Q^{\nu }_\rho( x))=0$.
Therefore, from \eqref{def vk 10} and \eqref{c:boundedness} we obtain
$$
\mathcal{H}^{n-1}(S_{v_k}\cap Q_{(1+\e)\rho}^\nu(x))  \leq M_\zeta \rho^{n-1}+((1+\e)^{n-1}-1)\rho^{n-1}
\le N_\zeta\rho^{n-1},
$$
where $N_\zeta:=M_\zeta+2^{n-1}$. By \eqref{Eepsdelta} and \eqref{vk bounded 2}, this inequality leads to the estimate
\begin{equation}\label{estimate 31}
E_k^{\e,p}(v_k,Q_{(1+\e)\rho}^\nu(x ))\le E_k(v_k,Q_{(1+\e)\rho}^\nu(x ))+ 2\e\mu N_\zeta\rho^{n-1}.
\end{equation}
Gathering $(f4)$, $(g6)$,  \eqref{q_min}-\eqref{def vk 10}, and \eqref{estimate 31} we obtain
\begin{align*}
E_k^{\e,p}(v_k&,Q_{(1+\e)\rho}^\nu(x )) \le E_k(v_k,Q_\rho^\nu(x )) +E_k(u_{x,\zeta ,\nu },Q_{(1+\e)\rho}^\nu(x )\setminus \overline Q_\rho^\nu(x ))+ 2\e\mu N_\zeta \rho^{n-1}
\\
 &\le (1+\e)E_k(u_k,Q_\rho^\nu(x )) + (1+2^n)c_2\rho^n + G_k(u_{x,\zeta ,\nu },Q_{(1+\e)\rho}^\nu(x )\setminus \overline Q_\rho^\nu(x ))+ 2\e\mu N_\zeta \rho^{n-1}
\\
&\le (1+\e)G_k(u_k,Q_\rho^\nu(x ))  + (3+2^n)c_2\rho^n 
+ \e (C_\zeta + 2 \mu N_\zeta)\rho^{n-1}
\\
&\le  (1+\e)\, m^{\mathrm{pc}}_{G_k}(u_{x,\zeta ,\nu },Q^{\nu }_\rho(x))  + (3+2^n)c_2\rho^n +\e(2 + C_\zeta+ 2 \mu N_\zeta ) \rho^{n-1}
\end{align*}
where $C_\zeta:= c_5(1+|\zeta|)(n-1) 2^{n-2}$. This inequality, together with \eqref{m<=E(v)}-\eqref{estimate 31}, gives
$$
m_{E^{\e,p}}(u_{x,\zeta ,\nu },Q_{(1+\e)\rho}^\nu(x))\le  (1+\e)\liminf_{k\to+\infty} m^{\mathrm{pc}}_{G_k}(u_{x,\zeta ,\nu },Q^{\nu }_\rho(x)) + (3+2^n)c_2\rho^n +\e
K_\zeta\rho^{n-1},
$$
where $K_\zeta:=2 + C_\zeta +   2  \mu N_\zeta$. 
Hence dividing by $\rho^{n-1}$, taking the limsup as $\rho\to 0+$, and recalling \eqref{g'} and \eqref{g^eta}, we obtain 
$$
(1+\e)^{n-1} g^{\e,p}(x,\zeta,\nu)
\le   (1+\e)g'(x,\zeta ,\nu)+\e K_\zeta.
$$
Eventually, by taking the limit as $\e\to 0+$ and appealing to \eqref{def g} we get
$$
g^0(x,\zeta,\nu)
\le g'(x,\zeta ,\nu),
$$
which concludes the proof.
\end{proof}

We are now ready to conclude the proof of Theorem~\ref{6.3}. 
\begin{proof}[Proof of Theorem~\ref{6.3} (d)]
We divide the proof into several intermediate steps. In the first four steps we prove the claimed inequality for functions $u$ which belong to $SBV^p(A,\R^m)\cap L^\infty(A,\R^m)$, while the general case of functions in $GSBV^p(A,\R^m)$ is treated in Step 5.
\smallskip

We may assume that the set $D$ introduced in Theorem~\ref{perturb} is contained in $(0,1)$.
Let $A\in\A$, $u\in SBV^p(A,\R^m)\cap L^\infty(A,\R^m)$, and $\e\in D$ be fixed. For every $x\in \R^n$ and every $\rho>0$ we set
\begin{equation}
Q^{\nu,\e}_{\rho}(x):= x+R_{\nu}\Big(\big(-\frac\rho2,\frac\rho2\big)^{n-1}{\times} \big(-\frac{\e\rho}2, \frac{\e\rho}2\big)\Big),
\end{equation}
where $R_{\nu}$ is the orthogonal  matrix introduced in (k) Section \ref{Notation}.
We fix $x\in S_u$ such that, by setting $\zeta:=[u](x)$ and $\nu:=\nu_u(x)$, we have
\begin{eqnarray}
\label{z neq 0}
&\ds \zeta\neq 0,
\\
\label{u:salto}
&\ds 
 \lim_{\rho \to 0+} \frac{1}{\rho^n}\int_{Q_\rho^{\nu,\e}(x)}|u(y)-u_{x,\zeta,\nu}(y)|^pdy = 0 ,
\\
\label{estimate:0}
&\ds g^{\e,p}(x,\zeta,\nu) = \lim_{\rho \to 0+} \frac{E^{\e,p}(u, Q_\rho^{\nu,\e}(x))}{\rho^{n-1}}.
\end{eqnarray}
Note that \eqref{z neq 0} and \eqref{u:salto} are satisfied for $\hs^{n-1}$-a.e.\ $x\in S_u$ (see, e.g., \cite[Definition 3.67 and Theorem~3.78]{AFP}). The same property holds for \eqref{estimate:0}, thanks to a generalized version of the Besicovitch Differentiation Theorem (see \cite{Mor} and \cite[Sections~1.2.1-1.2.2]{FonLeo}).

We extend $u$ to $\R^n$ by setting $u=0$ on $\R^n\setminus A$. By the $\Gamma$-convergence of $E_k^{\e,p} (\cdot,A)$ to $E^{\e,p} (\cdot,A)$ there exists a sequence $(u_k)$ converging to $u$ in $L^p_{\mathrm{loc}}(\R^n,\R^m)$ such that
$$
 \lim_{k\to +\infty}E_k^{\e,p} (u_k,A)=E^{\e,p} (u,A).
$$

Since $E^{\e,p}(u,\cdot)$ is a finite Radon measure, we have that
$E^{\e,p}(u,\partial Q_\rho^{\nu,\e}(x))=0$
for all $\rho>0$ such that $Q_\rho^{\nu,\e}(x)\subset A$, except for a countable set. As a consequence $(u_k)$ is a recovery sequence for $E^{\e,p}(u,\cdot)$ also in $Q_\rho^{\nu,\e}(x)$\ie
\begin{equation}\label{e:bordo}
\lim_{k\to +\infty}E_k^{\e,p}(u_k,Q_\rho^{\nu,\e}(x))=E^{\e,p}(u,Q_\rho^{\nu,\e}(x)),
\end{equation}
for all $\rho>0$ except for a countable set. 

We now fix $\lambda> \max\{ \|u\|_{L^\infty(\R^n,\R^m)}, |\zeta| \}$ and $h$, $\alpha$,  $\psi_1,\dots,\psi_h$, and $\mu$ as in Lemma \ref{estimate truncations}. We also fix $\rho$ satisfying \eqref{e:bordo}.
By \eqref{estimate for E} for every $k$ there exists $i_k\in\{1,\dots, h\}$ such that
\begin{equation} \nonumber
E_k^{\e,p}(\psi_{i_k}\!(u_k),Q_\rho^{\nu,\e}(x))\le (1+\e) E_k^{\e,p}(u_k,Q_\rho^{\nu,\e}(x))+c_2\mathcal{L}^n(Q_\rho^{\nu,\e}(x)\cap\{|u_k|\ge\lambda\}).
\end{equation}
Let $v_k:=\psi_{i_k}\!(u_k)$. By \eqref{bound mu} and \eqref{equality lambda} we deduce that $v_k\to u$ in $L^p_{\mathrm{loc}}(\R^n,\R^m)$ as well as
\begin{equation*} 
 |v_k|\le \mu\ \hbox{ in }\R^n,\quad 
 \limsup_{k\to +\infty} E_k^{\e,p}(v_k,Q_\rho^{\nu,\e}(x))\le (1+\e) E^{\e,p} (u,Q_\rho^{\nu,\e}(x)).
\end{equation*}
Hence there exists $k_0(\rho)>0$ such that whenever $k \ge  k_0(\rho)$
\begin{align}\label{estimate:1}
 E_k^{\e,p}(v_k, Q_\rho^{\nu,\e}(x))
\le (1+\e) E^{\e,p} (u,Q_\rho^{\nu,\e}(x))+ \rho^n.
\end{align}
We now start a multi-step modification of $v_k$ in order to obtain a function $z_k$ which is an admissible competitor in the $k$-th minimisation problem defining $g''(x,\zeta,\nu)$.

\medskip

\noindent\textit{Step 1. Attainment of the boundary datum for a blow-up of $u_k$.}
The blow-up function $v_k^\rho$ at $x$  is defined by
\begin{equation*}
v_k^\rho(y):= v_k(x+\rho y)\quad\hbox{for }y \in Q^{\nu,\e}:=Q_{1}^{\nu,\e}(0).
\end{equation*}
We now modify  $v_k^\rho$ so that it agrees with
$u_{0,\zeta,\nu}$ in a neighbourhood of $\partial Q^{\nu,\e}$. To this end, we consider the class $\A(Q^{\nu,\e}):=\{A\in\A:A\subset Q^{\nu,\e}\}$ and apply the Fundamental Estimate to the functionals $E_{k,\rho}^{\e,p}\colon \big(SBV^p(Q^{\nu,\e}, \R^m)\cap L^p(Q^{\nu,\e}, \R^m)\big){\times} \A(Q^{\nu,\e})\to[0,+\infty)$ defined as \begin{equation}\label{Ekrhoep}
E_{k,\rho}^{\e,p}(v,A):= \int_A f_k (x+\rho y,\nabla v(y) ) dy 
+ \int_{S_v\cap A} g^\e_k (x+\rho y,[v](y),\nu_v(y) ) d\mathcal{H}^{n-1}(y),
\end{equation}
where $g^\e_k$ is defined in \eqref{gek}.

Let $K_\e \subset Q^{\nu,\e}$ be a compact set such that
\begin{equation}\label{ass:1}
 c_2\mathcal{L}^{n}(Q^{\nu,\e}\setminus K_\e) + (c_5(1+|\zeta|)+\e|\zeta|)\,\mathcal{H}^{n-1}(\Pi^\nu_0\cap(Q^{\nu,\e}\setminus K_\e)) < \e.
\end{equation}
We can appeal to \cite[Proposition 3.1]{BDfV} to deduce the existence of a constant $M_\e>0$ and a \textit{finite} family of cut-off functions 
$\phi_1,\dots, \phi_N\in C_c^\infty(Q^{\nu,\e})$ such that $0\leq \phi_i\leq 1$ in $Q^{\nu,\e}$, $\phi_i=1$ in a neighbourhood of $K_\e$, and 
\begin{align}\nonumber
E_{k,\rho}^{\e,p}(\hat{v}^\rho_k,Q^{\nu,\e}) \leq {}&(1+\e)\big(E_{k,\rho}^{\e,p}(v^\rho_k, Q^{\nu,\e}) + E_{k,\rho}^{\e,p}(u_{0,\zeta,\nu},Q^{\nu,\e}\setminus K_\e)\big)
\\
&+ M_\e\|v^\rho_k - u_{0,\zeta,\nu}\|^p_{L^p(Q^{\nu,\e},\R^m)} + \e,\label{stima:fondamentale}
\end{align}   
where  $\hat{v}^\rho_k:= \phi_{i_k}v^\rho_k + (1-\phi_{i_k})u_{0,\zeta,\nu}$ for a suitable $i_ k\in \{1,\dots,N\}$. Clearly
\begin{equation}\label{hat vk bounded}
|\hat{v}^\rho_k| \le \mu\quad\hbox{in }\,Q^{\nu,\e}
\end{equation}
and  $\hat{v}^\rho_k=u_{0,\zeta,\nu}$ in a neighbourhood of $\partial Q^{\nu,\e}$. 
By $(f4)$ and $(g6)$ we have that 
\begin{align*}
E_{k,\rho}^{\e,p}(u_{0,\zeta,\nu},Q^{\nu,\e}\setminus K_\e)
&= \int_{Q^{\nu,\e}\setminus K_\e} \!\!\!\!\!  \!\!\!\!\!\!\!\!\!\! f_k (x+\rho y,0 ) dy + 
\int_{\Pi^\nu_0\cap(Q^{\nu,\e}\setminus K_\e)} \!\!\!\!\! \!\!\!\!\!  \!\!\!\!\!  \!\!\!\!\!\!\!\!\!\!g^\e_k (x+\rho y,\zeta,\nu ) d\mathcal{H}^{n-1}(y)\\
&\leq c_2\mathcal{L}^{n}(Q^{\nu,\e}\setminus K_\e) + (c_5(1+|\zeta|)+\e|\zeta|)\,\mathcal{H}^{n-1}(\Pi^\nu_0\cap(Q^{\nu,\e}\setminus K_\e)) < \e,
\end{align*}
where the last inequality follows from \eqref{ass:1}. 
Since $v_k \to u$ in $L^p(Q_{\rho}^{\nu,\e}(x),\R^m)$, it follows that 
\begin{equation} \label{f:dinner}
v^\rho_k (\cdot) = v_k(x+\rho\,\cdot) \to u(x+\rho\,\cdot) \quad \textrm{in  } \,\,L^p(Q^{\nu,\e},\R^m) \  \textrm{ as } \,\, k \to +\infty.
\end{equation}
Hence, from \eqref{stima:fondamentale} and \eqref{f:dinner} we have 
\begin{align}\nonumber
\limsup_{k\to +\infty}E_{k,\rho}^{\e,p}(\hat{v}^\rho_k,Q^{\nu,\e}) \leq{}& (1+\e)\Big(\limsup_{k\to +\infty}E_{k,\rho}^{\e,p}(v^\rho_k, Q^{\nu,\e}) + \e\Big)
\\
& + M_\e  \|u(x+\rho\,\cdot) - u_{0,\zeta,\nu}(\cdot) \|^p_{L^p(Q^{\nu,\e},\R^m)} + \e.
\label{c:17:51}
\end{align}

\medskip

\noindent\textit{Step 2. Estimate for $\nabla \hat{v}_k^\rho$.} 
We now show that $\nabla \hat{v}_k^\rho$ is small in $L^p$-norm for $k$ large and $\rho$ small. By the definition of $\hat{v}^\rho_k$ we have
\begin{align}\nonumber
\|\nabla \hat{v}_k^\rho\|_{L^p(Q^{\nu,\e}, \R^{m \times n})} &\leq \|\nabla \phi_{i_k}\|_{L^\infty(Q^{\nu,\e}, \R^n)} 
\| v_k^\rho-u_{0,\zeta,\nu}\|_{L^p(Q^{\nu,\e}, \R^m)} 
\\
 \label{v:hat}
&\qquad +
 \|\phi_{i_k}\|_{L^\infty(Q^{\nu,\e})} \|\nabla v_k^\rho\|_{L^p(Q^{\nu,\e}, \R^{m\times n})}\\
\nonumber
 & \leq C_\e \| v_k^\rho - u_{0,\zeta,\nu}\|_{L^p(Q^{\nu,\e}, \R^m)} +  \|\nabla v_k^\rho\|_{L^p(Q^{\nu,\e}, \R^{m\times n})},
\end{align}
for a suitable constant $C_\e>0$. We now estimate separately the two terms in the right-hand side of \eqref{v:hat}. 

As for the first term, note that
by \eqref{f:dinner} we can find $k_1(\rho) \geq  k_0(\rho)$ such that 
$$
\|v^\rho_k (\cdot)- u(x+\rho\,\cdot)\|_{L^p(Q^{\nu,\e}, \R^m)} \leq \rho \quad \textrm{for } \quad k \geq k_1(\rho).
$$
Hence from \eqref{u:salto} we deduce that for $k \geq k_1(\rho)$
\begin{align}\nonumber
\| v^\rho_k &- u_{0,\zeta,\nu}\|_{L^p(Q^{\nu,\e}, \R^m)}
\\
\label{v:tilde}
&\leq  \|v^\rho_k(\cdot) - u(x+\rho\,\cdot)\|_{L^p(Q^{\nu,\e},\R^m)} + \|u(x+\rho\,\cdot) - u_{0,\zeta,\nu}(\cdot)\|_{L^p(Q^{\nu,\e}, \R^m)}
\leq \omega_1 (\rho),
\end{align}
where $\omega_1(\rho)$ is independent of $k$ and $\omega_1(\rho) \to 0$ as $\rho \to 0+$.
 
For the second term in \eqref{v:hat}, by the definition of $v_k^\rho$, $(f3)$, and the positivity of $g_k$, we have that 
\begin{align}\label{star}
\int_{Q^{\nu,\e}}|\nabla v_k^\rho|^p dy &= 
\rho^{p-n} \int_{Q_\rho^{\nu,\e}(x)}|\nabla v_k |^p dy 
\leq \frac{\rho^{p-n}}{c_1} \int_{Q_\rho^{\nu,\e}(x)} f_k (y,\nabla v_k ) dy  \nonumber\\
& \leq \frac{\rho^{p-1}}{c_1}\Big(\frac{1}{\rho^{n-1}} E_k^{\e,p}(v_k, Q_\rho^{\nu,\e}(x))\Big).
\end{align}
By \eqref{estimate:0}  there exists $\rho_0>0$ such that $E^{\e,p}(u,Q_\rho^{\nu,\e}(x))/\rho^{n-1}<g^{\e,p}(x,\zeta,\nu)+1$ 
for every $0<\rho<\rho_0$. Therefore, for every $0<\rho<\rho_0$ satisfying \eqref{e:bordo} there exits $k_2(\rho)\ge  k_1(\rho)$ such that
$$
\frac{1}{\rho^{n-1}} \,E_k^{\e,p}(u_k, Q_\rho^{\nu,\e}(x)) <  g^{\e,p}(x,\zeta,\nu)+1,
$$
for  every $k \geq k_{2}(\rho)$. This inequality, together with \eqref{star}, gives
\begin{equation}\label{c:nabla-tilde}
\int_{Q^{\nu,\e}}|\nabla v_k^\rho|^p dy  \leq \frac{\rho^{p-1}}{c_1}(g^{\e,p}(x,\zeta,\nu)+1),
\end{equation}
for  every $k \geq k_{2}(\rho)$. Finally, putting together \eqref{v:hat}, \eqref{v:tilde}, and \eqref{c:nabla-tilde} yields 
\begin{equation}\label{grad:hat_v}
\|\nabla \hat{v}_k^\rho\|_{L^p(Q^{\nu,\e}, \R^{m \times n})} \leq \omega_2(\rho)
\end{equation}
for every $0<\rho<\rho_0$ satisfying \eqref{e:bordo} and every $k \geq k_{ 2 }(\rho)$, 
where $\omega_2(\rho)$ is independent of $k$ and $\omega_2(\rho) \to 0$ as $\rho \to 0+$.

\medskip

\noindent\textit{Step 3. Modification of $\hat{v}_{k}^\rho$ to make it piecewise constant.} 
On account of estimate \eqref{grad:hat_v}, we now further modify $\hat{v}_k^\rho$ using the same construction as in {\cite[page 332]{BDfV}}. Let $\zeta_1,\dots,\zeta_m$ be the coordinates of $\zeta$.
By \eqref{z neq 0} for every $0<\rho<\rho_0$ satisfying \eqref{e:bordo} there exists an integer $N_{\rho}>0$, with $2\sqrt{m}/N_\rho< \mu$ and $1/N_\rho<  |\zeta_i|$ for every $i$ with $\zeta_i\neq 0$, such that,
\begin{equation}\label{Nrho}
N_{\rho} \to +\infty \quad\text{and}\quad \omega_2(\rho) \, N_\rho \to 0+ \quad \textrm{as } \, \rho \to {0+}.
\end{equation}
Note that, by \eqref{hat vk bounded}, we have $|\hat v_k^\rho|<2\mu-(1/N_\rho)$ in $Q^{\nu,\e}$.
Let $\hat v_{k,1}^\rho,\dots, \hat v_{k,m}^\rho$ be the coordinates of $\hat v_k^\rho$.
Since $\hat v_{k,i}^\rho\in SBV(Q^{\nu,\e})$ for $i=1,\dots,m$, by the Coarea Formula the set $\{\hat v_{k,i}^\rho >t\}$
has finite perimeter in $Q^{\nu,\e}$ for $\mathcal{L}^1$-a.e. $t \in \R$ and 
\begin{align*}
\int_{Q^{\nu,\e}}|\nabla \hat v_{k,i}^\rho| dy =  |D\hat v_{k,i}^\rho|(Q^{\nu,\e}\setminus S_{\hat v_k^\rho})
= \int_{-2\mu}^{2\mu} \mathcal{H}^{n-1} \big( (Q^{\nu,\e}\setminus S_{\hat v_k^\rho}) \cap \partial^*\{\hat v_{k,i}^\rho >t\}\big) dt,
\end{align*}
where $\partial^*$ denotes the reduced boundary in $Q^{\nu,\e}$.

To simplify the exposition we assume that  $\mu$ is an integer.
From the Mean Value Theorem, for every integer $\ell$, with $-2N_\rho \mu \leq \ell < 2N_\rho \mu$, there exists $t^i_\ell\in\R$, with $\ell/N_\rho<t^i_\ell<(\ell +1)/N_\rho$, such that 
$\{\hat v_{k,i}^\rho >t^i_\ell\}$ has finite perimeter in $Q^{\nu,\e}$ and
\begin{equation}\label{l:Coa}
\int_{Q^{\nu,\e}}|\nabla \hat v_{k,i}^\rho| dy \geq \frac{1}{N_\rho} \sum_{\ell=-2N_\rho \mu}^{2N_\rho \mu -1} \mathcal{H}^{n-1} \!\big( (Q^{\nu,\e}\setminus S_{\hat v_k^\rho}) \cap \partial^*\{\hat v_{k,i}^\rho >t^i_\ell\}\big).
\end{equation}
We now define
$$
Z^i_\ell := \{ y\in Q^{\nu,\e}: t^i_\ell \le \hat v_{k,i}^\rho(y) < t^i_{\ell+1}\},
$$
and note that $Z^i_\ell$ has finite perimeter in $Q^{\nu,\e}$.
Moreover, since $|\hat v_k^\rho|<2\mu-(1/N_\rho)$ in $Q^{\nu,\e}$, 
the sets $Z^i_\ell$, $-2N_\rho \mu \leq \ell < 2N_\rho \mu$, form a partition of $Q^{\nu,\e}$.

We finally define the piecewise constant function $w_{k,i}^\rho \colon Q^{\nu,\e} \to \R$ as 
\begin{equation*}
w_{k,i}^\rho |_{Z_\ell} = \begin{cases}
0 \quad & \textrm{if } \,\, t^i_\ell \leq 0 < t^i_{\ell+1},\\
\zeta_i \quad & \textrm{if } \,\, t^i_\ell \leq \zeta_i  < t^i_{\ell+1},\\
t^i_\ell \quad & \textrm{otherwise.}
\end{cases}
\end{equation*}
Note that $w_{k,i}^\rho$ is well defined, since $|\zeta_i| > 1/N_\rho$ when $\zeta_i\neq 0$,
and therefore in this case $0$ and $\zeta_i$ cannot belong to the same interval $[t^i_\ell , t^i_{\ell+1} )$. 
Moreover, $w_{k,i}^\rho \in SBV_{\mathrm{pc}}(Q^{\nu,\e})$ since each set $Z^i_\ell$ 
has finite perimeter. Then the function  $w_k^\rho:=(w_{k,1}^\rho,\dots,w_{k,m}^\rho)$ belongs to
$SBV_{\mathrm{pc}}(Q^{\nu,\e},\R^m)$. 

We now claim that for every $0<\rho<\rho_0$ satisfying \eqref{e:bordo} and for every $k \geq k_{ 2 }(\rho)$
the following properties hold:
\begin{align}
&w_k^\rho = u_{0,\zeta,\nu} \text{ in a neighbourhood of  }\partial Q^{\nu,\e},\label{(1)} \\
&\|w_k^\rho - \hat v_k^\rho\|_{L^\infty(Q^{\nu,\e},\R^m)} \leq \frac{2\sqrt m}{N_\rho} < \mu, \label{(2)} \\
&\|w_k^\rho \|_{L^\infty(Q^{\nu,\e},\R^m)} \leq  2\mu, \label{(2')} \\
&\mathcal{H}^{n-1}((S_{w_k^\rho}\setminus S_{\hat v_k^\rho})\cap Q^{\nu,\e}) \leq \omega_3 (\rho), \label{(3)}
\vphantom{ \frac{2\sqrt m}{N_\rho}}
\end{align}
where $\omega_3(\rho)$ is independent of $k$ and  $\omega_3(\rho)\to 0+$ as $\rho\to 0+$. 

Property \eqref{(1)} follows from the definition of $w_k^\rho$. As for \eqref{(2)} we just note that 
$\|w_{k,i}^\rho - \hat v_{k,i}^\rho\|_{L^\infty(Q^{\nu,\e})} = 
\max_\ell \|w_{k,i}^\rho - \hat v_{k,i}^\rho\|_{L^\infty(Z^i_\ell)}\leq 2/N_\rho$.
Inequality \eqref{(2')} follows from \eqref{hat vk bounded} and \eqref{(2)}.
To prove \eqref{(3)} we observe that, up to $\mathcal{H}^{n-1}$-negligible sets, 
$S_{w_k^\rho}\subset \cup_i \cup_\ell \partial^* Z^i_\ell$, and since $Z^i_\ell = 
\{\hat v_{k,i}^\rho > t^i_\ell\} \setminus \{\hat v_{k,i}^\rho > t^i_{\ell+1}\}$, it follows that
$\partial^* Z^i_\ell \subset \partial^*\{\hat v_{k,i}^\rho > t^i_\ell\} \cup \partial^*\{\hat v_{k,i}^\rho > t^i_{\ell+1}\}$,
and hence 
$$
S_{w_k^\rho} \cap Q^{\nu,\e} \subset \bigcup_{i=1}^m \bigcup_{\ell =
 -2N_\rho \mu}^{2N_\rho \mu -1} (\partial^*\{\hat v_{k,i}^\rho > t^i_\ell\}\cap Q^{\nu,\e}).
$$
This inclusion implies that, by \eqref{grad:hat_v} and \eqref{l:Coa},
\begin{align*}
\mathcal{H}^{n-1}((S_{w_k^\rho}\setminus S_{\hat v_k^\rho})\cap Q^{\nu,\e}) 
&\leq \sum_{i=1}^m \sum_{\ell = -N_\rho \lambda}^{N_\rho \lambda -1} 
\mathcal{H}^{n-1}\left((Q^{\nu,\e}\setminus S_{\hat v_k^\rho}) \cap \partial^*\{\hat v_{k,i}^\rho > t^i_\ell\}\right) \nonumber\\
&\leq m N_\rho \int_{Q^{\nu,\e}}|\nabla \hat v_k^\rho| dy \leq m N_\rho \|\nabla \hat v_k^\rho\|_{L^p(Q^{\nu,\e},\, \R^{m \times n})} \leq \omega_3 (\rho) \, 
\end{align*}
where $\omega_3 (\rho): =m \omega_2 (\rho) N_\rho \to 0+$   as $\rho \to 0+$ 
by \eqref{Nrho}.

\medskip

\noindent \textit{Step 4. Conclusion of the proof for bounded functions.} We first note that by \eqref{Ekrhoep} and \eqref{c:17:51} we have 
\begin{align}\label{F:E}
\limsup_{k\to +\infty} & \int_{S_{\hat v_k^\rho}\cap Q^{\nu,\e}} 
g^{\e}_k (x+\rho y,[\hat v_k^\rho] (y),\nu_{\hat v_k^\rho} (y) )
d\mathcal{H}^{n-1} (y)\nonumber\\
&\leq (1+\e)\Big(\limsup_{k\to +\infty}E_{k,\rho}^{\e,p}(v^\rho_k, Q^{\nu,\e}) + \e\Big)
+ M_\e  \|u(x+\rho\,\cdot) - u_{0,\zeta,\nu} (\cdot) \|^p_{L^p(Q^{\nu,\e},\R^m)} + \e.
\end{align}
Further, by $(f4)$ and \eqref{c:nabla-tilde}, we can control the volume integral in \eqref{F:E} as follows:
\begin{equation*} 
\int_{Q^{\nu,\e}}f_k (x+\rho y, \nabla v_k^\rho (y) ) dy \leq c_2\int_{Q^{\nu,\e}} (1+ |\nabla v_k^\rho|^p) dy \leq c_2\Big(\e + \frac{\rho^{p-1}}{c_1}(g^{\e,p}(x,\zeta,\nu)+1) \Big)
\end{equation*}
for every $0<\rho<\rho_0$ satisfying \eqref{e:bordo}  and every $k \geq k_{ 2 }(\rho)$.

By \eqref{Ekrhoep}, this inequality and \eqref{F:E} imply in particular that 
\begin{align}\nonumber
\limsup_{k\to +\infty} & \int_{S_{\hat v_k^\rho}\cap Q^{\nu,\e}} 
g^{\e}_k (x+\rho y,[\hat v_k^\rho] (y),\nu_{\hat v_k^\rho} (y))d\mathcal{H}^{n-1} (y)\\\nonumber
&\leq  (1+\e) \limsup_{k \to + \infty}\int_{S_{ v_k^\rho}\cap Q^{\nu,\e}}
g^{\e}_k (x+\rho y,[v_k^\rho] (y),\nu_{v_k^\rho} (y))d\mathcal{H}^{n-1} (y) \\\label{c:nuova-label}
& + 2 c_2\Big(\e + \frac{\rho^{p-1}}{c_1}(g^{\e,p}(x,\zeta,\nu)+1) \Big) 
+ M_\e  \|u(x+\rho\,\cdot) - u_{0,\zeta,\nu} (\cdot) \|^p_{L^p(Q^{\nu,\e},\R^m)} + 3 \e.
\end{align}
Since 
\begin{align*}
\int_{S_{ v_k^\rho}\cap Q^{\nu,\e}}\!\!\!\!\!  \!\!\!\!\! \!\!\!\!\! 
g^{\e}_k (x+\rho y,[v_k^\rho] (y),\nu_{v_k^\rho} (y)\Big)d\mathcal{H}^{n-1} (y) 
= \frac{1}{\rho^{n-1}}\int_{S_{v_k}\cap Q_\rho^{\nu,\e}(x)} \!\!\!\!\! \!\!\!\!\! \!\!\!\!\! \!\!\!\!\!
g^\e_k (y,[v_k](y),\nu_{v_k}(y))d\mathcal{H}^{n-1}(y),
\end{align*}
gathering \eqref{estimate:1} and \eqref{c:nuova-label}  gives  
\begin{align}
&\limsup_{k \to +\infty} \int_{S_{\hat v_k^\rho}\cap Q^{\nu,\e}} 
g^{\e}_k (x+\rho y,[\hat v_k^\rho] (y),\nu_{\hat v_k^\rho} (y)) d\mathcal{H}^{n-1} (y)
\nonumber \\
&\leq (1+\e)^2 \frac{1}{\rho^{n-1}}\,E^{\e,p}(u, Q_\rho^{\nu,\e}(x))
+ 2 \rho   + 2 c_2\Big(\e + \frac{\rho^{p-1}}{c_1}(g^{\e,p}(x,\zeta,\nu)+1) \Big)
\label{estimate getaeps} \\
&+ M_\e  \|u(x+\rho\,\cdot) - u_{0,\zeta,\nu} (\cdot) \|^p_{L^p(Q^{\nu,\e},\R^m)} + 3 \e. \nonumber
\end{align}
We now estimate the left-hand side in \eqref{estimate getaeps}. We have 
\begin{align}\label{almost:there1}
&\int_{S_{\hat v_k^\rho}\cap Q^{\nu,\e}}
g^\e_k (x+\rho y,[\hat v_k^\rho](y),\nu_{\hat v_k^\rho}(y) )d\mathcal{H}^{n-1}(y)\nonumber\\
&\geq \int_{(S_{\hat v_k^\rho}\cap S_{w_k^\rho})\cap Q^{\nu,\e}}
g_k (x+\rho y,[\hat v_k^\rho ] (y),\nu_{\hat v_k^\rho}(y))d\mathcal{H}^{n-1}(y) \nonumber \\
&= \int_{S_{w_k^\rho}\cap Q^{\nu,\e}}g_k (x+\rho y,[w_k^\rho](y),\nu_{w_k^\rho}(y))d\mathcal{H}^{n-1}(y)\nonumber\\
&+ \int_{(S_{\hat v_k^\rho}\cap S_{w_k^\rho})\cap Q^{\nu,\e}} (g_k (x+\rho y,[\hat v_k^\rho](y),\nu_{\hat v_k^\rho}(y)) - g_k (x+\rho y,[w_k^\rho](y),\nu_{w_k^\rho}(y)) ) \,d\mathcal{H}^{n-1}(y)\nonumber\\
&  - \int_{(S_{w_k^\rho}\setminus S_{\hat v_k^\rho})\cap Q^{\nu,\e}}g (x+\rho y,[w_k^\rho](y),\nu_{w_k^\rho}(y))d\mathcal{H}^{n-1}(y) =: I_1 + I_2 - I_3. 
\end{align}
We now claim that 
\begin{equation}\label{I_2 I_3}
| I_2|\le  \omega_4(\rho)\quad\hbox{and}\quad  |I_3 | \leq \omega_5(\rho)\
\end{equation}
for $k\ge k_{ 2 }(\rho)$, 
where  $\omega_4(\rho)$ and $\omega_5(\rho)$ are independent of $k$ and tend to  $0+$ as $\rho \to 0+$.

Thanks to the symmetry condition $(g7)$, for the term $I_2$ we may choose the orientations of $\nu_{\hat v_k^\rho}$
and $\nu_{w_k^\rho}$ so that $\nu_{\hat v_k^\rho} = \nu_{w_k^\rho}$
$\hs^{n-1}$-a.e.\ on $S_{\hat v_k^\rho}\cap S_{w_k^\rho}$.
Thus, by assumptions $(g2)$ and $(g6)$,
\begin{align*}
& |g_k (x+\rho y,[\hat v_k^\rho](y),\nu_{\hat v_k^\rho}(y)) - g_k (x+\rho y,[w_k^\rho](y),\nu_{w_k^\rho}(y))| \\
&\leq \sigma_2(|[\hat v_k^\rho] (y) -[w_k^\rho] (y)|)
\big(g_k (x+\rho y,[\hat v_k^\rho] (y),\nu_{\hat v_k^\rho} (y)) 
+ g_k (x+\rho y,[w_k^\rho] (y),\nu_{w_k^\rho} (y))\big) \\
&\leq 2 c_5 \sigma_2(2\|\hat v_k^\rho - w_k^\rho\|_{L^{\infty}(Q^{\nu,\e}\!,\R^m)})(1+ \|\hat v_k^\rho\|_{L^{\infty}(Q^{\nu,\e}\!,\R^m)} + \|w_k^\rho\|_{L^{\infty}(Q^{\nu,\e}\!,\R^m)}), 
\end{align*}
for $\hs^{n-1}$-a.e. $y \in S_{\hat v_k^\rho}\cap S_{w_k^\rho}$.
Therefore, using \eqref{hat vk bounded},  \eqref{(2)}, and  \eqref{(2')} we obtain
\begin{align*}
| I_2 | \leq 2 c_5 (1 + 3\mu)\, \sigma_2(4 \sqrt m/ N_\rho )\, \mathcal{H}^{n-1}(S_{\hat v_k^\rho}\cap Q^{\nu,\e})
\end{align*}
for every $k\ge k_{ 2 }(\rho)$.

Now recall that, by the definition of $\hat v_k^\rho$,
$$
S_{\hat v_k^\rho}\cap Q^{\nu,\e} \subset \big(S_{ v_k^\rho}\cap Q^{\nu,\e}\big) \cup \big( \Pi^\nu_0\cap (Q^{\nu,\e}\setminus K_\e)\big),
$$
hence by \eqref{ass:1},
$$
\mathcal{H}^{n-1}(S_{\hat v_k^\rho}\cap Q^{\nu,\e}) \leq \mathcal{H}^{n-1}(S_{ v_k^\rho}\cap Q^{\nu,\e}) + \frac{\e}{c_5}
 =\frac{1}{\rho^{n-1}} \mathcal{H}^{n-1}(S_{ v_k}\cap Q^{\nu,\e}_{\rho}(x)) + \frac{\e}{c_5}.
$$
In terms of the functions $v_k$, by \eqref{estimate:1}, this implies that 
$$
\mathcal{H}^{n-1}(S_{\hat v_k^\rho}\cap Q^{\nu,\e}) 
 \leq \frac{1+\e}{c_4}\frac{1}{\rho^{n-1}} E^{\e,p}(u, Q^{\nu,\e}_{\rho}(x)) 
+ \frac{\rho}{c_4} + \frac{\e}{c_5}
$$
for every $k\ge k_{ 2 }(\rho)$.
Hence, for the term $I_2$ we have 
$$
|I_2| \leq  2 c_5 (1 + 3 \mu)\, \sigma_2(4\sqrt m / N_\rho ) 
\Big(\frac{1+\e}{c_4\,\rho^{n-1}} E^{\e,p}(u, Q^{\nu,\e}_{\rho}(x)) 
+ \frac{\rho}{c_4} + \frac{\e}{c_5} \Big).
$$
Since $\sigma_2(t)\to 0+$ as $t\to 0+$, 
by \eqref{estimate:0} we obtain that 
$|I_2| \leq \omega_4(\rho)$ for every $k\ge k_{ 2 }(\rho)$, where $\omega_4(\rho)$ is independent of $k$ and $\omega_4(\rho)\to 0+$ as $\rho \to 0+$.

As for the term $I_3$, proceeding as above and  using \eqref{(2')} we get
\begin{align*}
|I_3| \leq c_5 (1+4 \mu) \,\mathcal{H}^{n-1}\big((S_{w_k^\rho}\setminus S_{\hat v_k^\rho}) \cap Q^{\nu,\e}\big),
\end{align*}
which, by  \eqref{(3)}, implies that $|I_3| \leq \omega_5(\rho)$ for every $k\ge k_{  2 }(\rho)$, where $\omega_5(\rho) := c_5 (1+4\mu) \omega_3(\rho) \to 0+$ as $\rho \to 0+$. This concludes the proof of \eqref{I_2 I_3}.
\medskip

By combining \eqref{estimate getaeps}, \eqref{almost:there1}, and \eqref{I_2 I_3} we deduce that 

\begin{align*}
\limsup_{k \to +\infty} & \int_{S_{w_k^\rho}\cap Q^{\nu,\e}} 
g_k (x+\rho y,[w_k^\rho] (y),\nu_{w_k^\rho} (y))d\mathcal{H}^{n-1} (y)\nonumber\\
&\leq  (1+\e)^2 \frac{1}{\rho^{n-1}}\,E^{\e,p}(u, Q_\rho^{\nu,\e}(x))
 + 2 \rho  + \omega_4(\e,\rho)+  \omega_5(\e,\rho) \nonumber\\
& + 2 c_2\Big(\e + \frac{\rho^{p-1}}{c_1}(g^{\e,p}(x,\zeta,\nu)+1) \Big) 
+ M_\e  \|u(x+\rho\,\cdot) - u_{0,\zeta,\nu} (\cdot) \|^p_{L^p(Q^{\nu,\e},\R^m)} + 3 \e.
\end{align*}
We now define $z_k^\rho(y):= w_k^\rho((y-x)/\rho)$ for every $y \in Q^{\nu, \e}_{\rho} (x)$. Note that $z_k^\rho \in SBV_{\mathrm{pc}} (Q^{\nu, \e}_{\rho} (x),\R^m)$
and $z_k^\rho = u_{x, \zeta, \nu}$ in a neighbourhood of $\partial Q^{\nu, \e}_{\rho} (x)$.
In terms of the functions $z_k^\rho$ the previous estimate gives

\begin{align*} 
\limsup_{k \to +\infty} \frac{1}{\rho^{n-1}} &m^{\mathrm{pc}}_{G_k}(u_{x,\zeta,\nu}, Q_\rho^{\nu}(x)) 
\le \limsup_{k \to +\infty} \frac{1}{\rho^{n-1}} m^{\mathrm{pc}}_{G_k}(u_{x,\zeta,\nu}, Q_\rho^{\nu,\e}(x)) \nonumber\\
& \leq \limsup_{k \to +\infty} \frac{1}{\rho^{n-1}}\int_{S_{z_k^\rho}\cap Q_\rho^{\nu,\e}(x)} 
g_k (y,[z_k^\rho] (y),\nu_{z_k^\rho} (y))d\mathcal{H}^{n-1} (y)\nonumber\\
&\leq  (1+\e)^2 \frac{1}{\rho^{n-1}}\,E^{\e,p}(u, Q_\rho^{\nu,\e}(x))
 + 2 \rho  + \omega_4(\e,\rho) + \omega_5(\e,\rho) \nonumber\\
& + 2 c_2\Big(\e + \frac{\rho^{p-1}}{c_1}(g^{\e,p}(x,\zeta,\nu)+1) \Big) 
+ M_\e  \|u(x+\rho\,\cdot) - u_{0,\zeta,\nu} (\cdot) \|^p_{L^p(Q^{\nu,\e},\R^m)} + 3 \e.
\end{align*}
Finally, taking the limsup as $\rho \to 0+$ and invoking \eqref{g''}, 
\eqref{u:salto}, and \eqref{estimate:0}, 
 we obtain
$$
g'' (x,\zeta,\nu) \leq (1+\e)^2  g^{\e,p} (x,\zeta,\nu)  + C \e,
$$
with $C:= 2 c_2 +3$.
Recalling the definition of $\zeta$ and $\nu$, we obtain that
\begin{equation*} 
g'' (x,[u](x),\nu_{u}(x)) \leq (1+\e)^2 g^{\e,p} (x,[u](x),\nu_u(x))  
+ C \e 
\end{equation*}
holds true for $\hs^{n-1}$-a.e.\ $x\in S_u \cap A$. Taking the limit as $\e\to 0+$ and using \eqref{def g}
we get 
\begin{equation*} 
g'' (x,[u](x),\nu_{u}(x)) \leq g^0(x,[u](x),\nu_u(x)) 
\end{equation*}
for $\hs^{n-1}$-a.e.\ $x\in S_u \cap A$,
thus proving \eqref{g0>=ghom} for $u\in SBV^p(A,\R^m)\cap L^{\infty}(A,\R^m)$.

\medskip

\textit{Step 5. Extension to unbounded functions in $GSBV^p$.} 
Let $A\in\A$ and $u\in GSBV^p(A,\R^m)$. For every integer $k\ge 1$ we define $z_k:=\alpha_k(u)$, where 
$\alpha_k\in C^1_c(\R^m,\R^m)$ satisfies $\alpha_k(\zeta)=\zeta$ for every $\zeta\in\R^m$ with
$|\zeta|\le k$. By (h) in Section \ref{Notation} we have that $z_k\in SBV^p(A,\R^m)\cap L^\infty(A,\R^m)$.
Let $\Sigma_k:=\{x\in S_u\cap A:|u^\pm(x)|<k\}$. By the definition of $u^\pm(x)$ as approximate limits, it is easy to see that for $\hs^{n-1}$-a.e.\ $x\in \Sigma_k$ we have either
$z^\pm_k(x)=u^\pm(x)$ and $\nu_{z_k}(x)=\nu_u(x)$ or $z^\pm_k(x)=u^\mp(x)$ and $\nu_{z_k}(x)=-\nu_u(x)$ (see \cite[Remark 4.32]{AFP}). On the other hand, by 
 the previous steps in the proof we have that 
 \begin{equation*} 
g'' (x,[z_k](x),\nu_{z_k}(x)) \leq g^0(x,[z_k](x),\nu_{z_k}(x)) 
\end{equation*}
 for $\hs^{n-1}$-a.e.\ $x\in \Sigma_k$. By $(g7)$ this implies that
  \begin{equation} \label{stima:integrale}
g'' (x,[u](x),\nu_{u}(x)) \leq g^0(x,[u](x),\nu_{u}(x)) 
\end{equation}
 for $\hs^{n-1}$-a.e.\ $x\in \Sigma_k$. Since the integer $k$ is arbitrary, \eqref{stima:integrale} holds for 
$\hs^{n-1}$-a.e.\ $x\in S_u$.
\end{proof}


\section*{Appendix}
\renewcommand{\theequation}{A.\arabic{equation}}
\renewcommand{\thethm}{A.\arabic{thm}}

\noindent In this section we collect some technical results that we have used throughout the paper. We begin with an example of a family of orthogonal matrices $R_\nu$ satisfying all assumptions of (k) of Section~\ref{Notation}.

\begin{example}\label{Rnu}
Let $\phi_\pm\colon\Sph^{n-1}\setminus\{\pm e_n\}\to\R^{n-1}$ be the stereographic projection from $\pm e_n$ into the plane $x_n=0$ and let $\psi_\pm\colon \R^{n-1}\to \Sph^{n-1}\setminus\{\pm e_n\}$ be its inverse function. For every $\nu\in \widehat{\Sph}^{n-1}_\pm$ we consider the vectors $\xi_i(\nu):=\partial_i\psi_\mp(\phi_\mp(\nu))$, $i=1, \dots,n-1$, which are tangent to $\Sph^{n-1}$ at $\nu$, and hence satisfy $\xi_i(\nu){\,\cdot\,}\nu=0$. Since $\psi_\mp$ are conformal maps, we have $\xi_i(\nu){\,\cdot\,}\xi_j(\nu)=0$ for $i\neq j$. Let $\nu_i(\nu):=\xi_i(\nu)/|\xi_i(\nu)|$. Then the vectors $\nu_1(\nu),\nu_2(\nu),\dots, \nu_{n-1}(\nu), \nu$ form an orthonormal basis of $\R^n$, therefore they are the columns of an orthogonal matrix, denoted by $R_\nu$. It is clear from the construction that $R_\nu e_n=\nu$ and that the restriction of $\nu\mapsto R_\nu$ to $\widehat{\Sph}^{n-1}_\pm$ is continuous. Moreover, since $\phi_+(-\nu)=-\phi_-(\nu)$ for every $\nu\in \Sph^{n-1}\setminus\{e_n,-e_n\}$, we have $\psi_+(-y)=-\psi_-(y)$ for every $y\in\R^{n-1}\setminus\{0\}$. It follows that $\xi_i(-\nu)=\xi_i(\nu)$, hence $\nu_i(-\nu)=\nu_i(\nu)$  for every $\nu\in\Sph^{n-1}\setminus\{e_n,-e_n\}$.
This property is clearly true also for $\nu=\pm e_n$, since $\nu_i(\pm e_n)=e_i$. It follows that $R_{-\nu}Q(0)=R_\nu Q(0)$ for every $\nu\in\Sph^{n-1}$. 
\end{example}

The following remark will be used in \cite{ergodic}.

\begin{rem}
From the formulas defining the stereographic projections $\phi_\pm$ it follows that $\nu\in (\Sph^{n-1}\cap \Q^n)\setminus\{e_n,-e_n\}$ if and only if $\phi_\pm(\nu)\in \Q^{n-1}\setminus\{0\}$. Therefore $\Sph^{n-1}\cap \Q^n$ is dense in $\Sph^{n-1}$. Moreover, the explicit formulas for  $\partial_i\psi_\pm$ show that $\nu_i(\nu)\in \Sph^{n-1}\cap \Q^n$ for every $\nu\in \Sph^{n-1}\cap \Q^n$, hence  $R_\nu\in \Q^{n{\times} n}$ for every $\nu\in \Sph^{n-1}\cap \Q^n$.
\end{rem}

The rest of this section is devoted to some technical lemmas needed to prove some of the properties satisfied by the functions $f'$, $f''$, $g'$, and $g''$ introduced in \eqref{f'}-\eqref{g''} and by the functions $f^{\e,p}$ and $g^{\e,p}$ introduced in \eqref{f^eta} and \eqref{g^eta}.

\begin{lem}[Upper semicontinuity]\label{upper semicontinuity}
Let $X$ be either $L^0(\R^n,\R^m)$ or $L^p_{\mathrm{loc}}(\R^n,\R^m)$, and let  $H\colon X{\times}\A\to[0,+\infty]$ be a functional such that
\begin{itemize}
\item[$(h1)$] (locality) $H(u,A)=H(v,A)$ if $u$, $v\in X$, $A\in\A$, and $u=v$ $\mathcal{L}^n$-a.e.\ in $A$,
\item[$(h2)$] (measure) for every $u\in X\cap SBV_{\mathrm{loc}}(\R^n,\R^m)$ the function $H(u,\cdot)$ is the restriction to $\A$ of a countably additive function defined on the $\sigma$-algebra of the Borel subsets of $\R^n$,
\item[$(h3)$] (upper bound) for every $u\in X\cap SBV_{\mathrm{loc}}(\R^n,\R^m)$ and every $A\in \A$
\begin{equation*}
H(u,A)\le c_2\int_A(1+|\nabla u|^p)dx +c_5\int_{S_u\cap A}(1+|[u]|)\,d\hs^{n-1}.
\end{equation*}
\end{itemize}
Let $m^{1,p}_H$,  $m^{\mathrm{pc}}_H$, and $m_H$ be as in \eqref{emme}-\eqref{mE0eta}, and let $\rho>0$. Then 
\begin{itemize}
\item[(a)]the functions 
$$
(x,\xi)\mapsto m_H(\ell_\xi,Q_\rho(x))\quad\hbox{and}\quad (x,\xi)\mapsto m^{1,p}_H(\ell_\xi,Q_\rho(x))
$$
are upper semicontinuous in $\R^n{\times}\R^{m{\times}n}$;
\item[(b)]the restrictions of the function
$$
(x,\zeta,\nu)\mapsto m_H(u_{x,\zeta,\nu},Q^\nu_\rho(x))
$$ 
to the sets $\R^n{\times}\R^m_0{\times}\widehat\Sph^{n-1}_+$ and $\R^n{\times}\R^m_0{\times}\widehat\Sph^{n-1}_-$ are upper semicontinuous;
\item[(c)]for every $\zeta_0\in \R^m_0$ the restrictions of the function
$$
(x,\nu)\mapsto m^{\mathrm{pc}}_H(u_{x,\zeta_0,\nu},Q^\nu_\rho(x))
$$
to the sets $\R^n{\times}\widehat\Sph^{n-1}_+$ and $\R^n{\times}\widehat\Sph^{n-1}_-$ are upper semicontinuous.
\end{itemize}
\end{lem}

\begin{proof} In the proof of (a) we only deal with $m_H$, the proof of the upper semicontinuity of $m^{1,p}_H$ being similar.

Fix $x_0\in\R^n$, $\xi_0\in\R^{m{\times}n}$, and $\e>0$. By the definition of $m_H$ there exist $u_0\in X$, with $u_0|_{Q_\rho(x_0)}\in SBV^p(Q_\rho(x_0),\R^m)$, and $\delta_0\in(0,\rho/3)$ such that
\begin{align}\label{u0=ell}
&u_0=\ell_{\xi_0}\quad\mathcal{L}^n\hbox{-a.e.\ in }Q_\rho(x_0)\setminus Q_{\rho-3\delta_0}(x_0),
\\
&H(u_0,Q_\rho(x_0))<m_H(\ell_{\xi_0},Q_\rho(x_0))+\e.
\label{quasi minimality f}
\end{align}
Now fix $\delta\in (0,\delta_0)$, $x\in Q_\delta(x_0)$, $\xi\in\R^{m{\times}n}$ with $|\xi-\xi_0|<\delta$, and $\varphi\in C^\infty_c(\R^n)$ with
$\mathrm{supp}\,\varphi\subset Q_\rho(x)$, $\varphi=1$ in $Q_{\rho-\delta}(x)$, $0\le \varphi\le 1$ in $\R^n$, and $|\nabla\varphi|\le 3/\delta$ in $\R^n$.
We define
$u_1\in SBV^p_{\mathrm{loc}}(\R^n,\R^m)$ by
\begin{equation*}\label{def v f}
u_1:=
\begin{cases}
u_0&\hbox{in }Q_{\rho-\delta}(x),
\\
\varphi\,\ell_{\xi_0}+ (1-\varphi)\,\ell_\xi&\hbox{in }\R^n\setminus Q_{\rho-2\delta}(x).
\end{cases}
\end{equation*}
Since $x\in Q_\delta(x_0)$, we have $Q_{\rho-\delta}(x)\setminus Q_{\rho-2\delta}(x) \subset\subset Q_\rho(x_0)\setminus Q_{\rho-3\delta_0}(x_0)$. Therefore $u_1$ is well defined, since, by \eqref{u0=ell}, both formulas  give the same value in  the overlapping set $Q_{\rho-\delta}(x)\setminus Q_{\rho-2\delta}(x)$.
Moreover $u_1=\ell_\xi$ 
in a neighbourhood of $\partial  Q_\rho(x)$, hence $m_H(\ell_\xi,Q_\rho(x))\le H(u_1,Q_\rho(x))$. Therefore, using $(h1)$-$(h3)$, we obtain
\begin{equation}\label{mF<=F}
m_H(\ell_\xi,Q_\rho(x))\le H(u_0,Q_{\rho-\delta}(x))+c_2\int_{Q_\rho(x)\setminus Q_{\rho-2\delta}(x)}
\!\!\!\!\!  \!\!\!\!\! \!\!\!\!\! \!\!\!\!\! \!\!\!\!\! \!\!\!\!\! (1+|\nabla u_1|^p)\,dy.
\end{equation}
Since $\nabla u_1=\varphi\xi_0+(1-\varphi)\xi  + (\ell_{\xi_0}-\ell_\xi){\otimes}\nabla\varphi$ in $Q_\rho(x)\setminus Q_{\rho-2\delta}(x)$, by convexity we have $|\nabla u_1|^p\le  3^{p-1}  (|\xi_0|^p+|\xi|^p  + |\xi_0-\xi|^pC_1|\nabla\varphi|^p)$, where $C_1:=\sup\{|y|^p:y\in Q_{\rho+\delta_0}(x_0)\}$. 

Therefore \eqref{mF<=F}, together with the estimates for $|\xi_0-\xi|$ and $|\nabla\varphi|$, yields
$$
m_H(\ell_\xi,Q_\rho(x))\le H(u_0,Q_\rho(x))+C_2(\rho^n-(\rho-2\delta)^n),
$$
where  $C_2:=c_2\big(1+3^{2p-1}(|\xi_0|^p+\delta_0^p+C_1)\big)$. Combining this inequality with 
\eqref{quasi minimality f} we get
$$
m_H(\ell_\xi,Q_\rho(x))\le m_H(\ell_{\xi_0},Q_\rho(x_0))+\e+2nC_2\rho^{n-1}\delta.
$$
Therefore, if $0<\delta<\min\{\delta_0,\e/(2nC_2\rho^{n-1})\}$, $x\in Q_\delta(x_0)$, and $|\xi-\xi_0|<\delta$, then
$$
m_H(\ell_\xi,Q_\rho(x))\le m_H(\ell_{\xi_0},Q_\rho(x_0))+2\e.
$$
This proves the upper semicontinuity of $(x,\xi)\mapsto m_H(\ell_\xi,Q_\rho(x))$ at $(x,\xi)=(x_0,\xi_0)$.

To prove (b), we fix three points $x_0\in\R^n$, $\zeta_0\in\R^m_0$, $\nu_0\in \widehat\Sph^{n-1}_+$, three sequences $(x_j)\subset \R^n$, $(\zeta_j)\subset \R^m_0$, $(\nu_j) \subset \widehat\Sph^{n-1}_+$, with $x_j\to x_0$, $\nu_j\to\nu_0$, $\zeta_j\to \zeta_0$, and a constant $\e>0$.  
By definition there exist $v_0\in X$, with $v_0|_{Q^{\nu_0}_\rho(x_0)}\in SBV^p(Q^{\nu_0}_\rho(x_0),\R^m)$, and $\delta_0\in(0,\rho/3)$ such that
\begin{align}\label{v0=ux0}
&v_0=u_{x_0,\zeta_0,\nu_0}\quad\mathcal{L}^n\hbox{-a.e.\ in }Q^{\nu_0}_\rho(x_0)\setminus Q^{\nu_0}_{\rho-3\delta_0}(x_0),
\\
&H(v_0,Q^\nu_\rho(x_0))<m_H(u_{x_0,\zeta_0,\nu_0},Q^{\nu_0}_\rho(x_0))+\e.
\label{quasi minimality}
\end{align}
Let us fix $\delta\in (0,\delta_0/2)$. There exists an integer $i_\delta$ such that $Q^{\nu_0}_{\rho-\delta}(x_j) \subset Q^{\nu_0}_\rho(x_0)$, $Q^{\nu_0}_{\rho+\delta}(x_j) \subset Q^{\nu_0}_{\rho+2\delta}(x_0)$, and $ Q^{\nu_0}_{\rho-3\delta_0}(x_0)\subset Q^{\nu_0}_{\rho-5\delta}(x_0)\subset Q^{\nu_0}_{\rho-4\delta}(x_j)$ for every $j\ge i_\delta$.

By (k) in Section \ref{Notation} the function $\nu\mapsto R_\nu$ is continuous on 
$\widehat\Sph^{n-1}_+$. Consequently there exists an integer $j_\delta\ge i_\delta$ such that
$Q^{\nu_j}_{\rho-2\delta}(x)\subset Q^{\nu_0}_{\rho-\delta} (x)$, $Q^{\nu_j}_\rho(x)\subset Q^{\nu_0}_{\rho+\delta} (x)$,
and $Q^{\nu_0}_{\rho-4\delta} (x)\subset Q^{\nu_j}_{\rho-3\delta}(x)$
for every $j\geq j_\delta$ and every $x\in\R^n$. Therefore the previous inclusions imply that
\begin{align}\label{cubes}
Q^{\nu_j}_{\rho-2\delta}(x_j)\setminus Q^{\nu_j}_{\rho-3\delta}(x_j) &\subset Q^{\nu_0}_\rho(x_0)\setminus Q^{\nu_0}_{\rho-5\delta}(x_0)\subset Q^{\nu_0}_\rho(x_0)\setminus Q^{\nu_0}_{\rho-3\delta_0}(x_0),
\\
Q^{\nu_j}_\rho(x_j)\setminus  Q^{\nu_j}_{\rho-3\delta}(x_j) &\subset Q^{\nu_0}_{\rho+2\delta}(x_0)\setminus Q^{\nu_0}_{\rho-5\delta}(x_0),
\label{cubes3}
\end{align}
for every $j\geq j_\delta$.

Let $\psi_j\in C^\infty_c(\R^n)$ be such that
$\mathrm{supp}\,\psi_j\subset Q^{\nu_j}_\rho(x_j)$, $\psi_j=1$ in $Q^{\nu_j}_{\rho-\delta}(x_j)$, $0\le \psi_j\le 1$ in $\R^n$, and $|\nabla\psi_j|\le 3/\delta$ in $\R^n$.
We define
$v_j\in SBV^p_{\mathrm{loc}}(\R^n,\R^m)$ by
\begin{equation*}
v_j:=
\begin{cases}
v_0&\hbox{in }Q^{\nu_j}_{\rho-2\delta}(x_j),
\\
\psi_ju_{x_0,\zeta_0,\nu_0}+(1-\psi_j)u_{x_j,\zeta_j,\nu_j}&\hbox{in }\R^n\setminus Q^{\nu_j}_{\rho-3\delta}(x_j).
\end{cases}
\end{equation*}
By \eqref{v0=ux0} and \eqref{cubes} the function $v_j$ is well defined, since both formulas  give the same value in the overlapping set
$Q^{\nu_j}_{\rho-2\delta}(x_j)\setminus Q^{\nu_j}_{\rho-3\delta}(x_j)$. Moreover $v_j=u_{x_j,\zeta_j,\nu_j}$ 
in a neighbourhood of $\partial  Q^{\nu_j}_\rho(x_j)$, hence $m_H(u_{x_j,\zeta_j,\nu_j},Q^{\nu_j}_\rho(x_j))\le H(v_j,Q^{\nu_j}_\rho(x_j))$. So, using $(h1)$-$(h3)$ and setting $A_j:=Q^{\nu_j}_\rho(x_j)\setminus Q^{\nu_j}_{\rho-3\delta}(x_j)$, we obtain
\begin{align}\nonumber
m_H(u_{x_j,\zeta_j,\nu_j}&,Q^{\nu_j}_\rho(x_j))\le H(v_0,Q^{\nu_j}_{\rho-2\delta}(x_j))
\\
&+c_2\int_{A_j} (1+|\nabla v_j|^p)\,dy +
c_5\int_{S_{v_j}\cap A_j}
\!\!\!\!\!  \!\!\!\!\!   (1+|[ v_j]|)\,d\hs^{n-1}.\label{mH<=H}
\end{align}

Since $|\nabla v_j|\le |\nabla\psi_j||u_{x_0,\zeta_0,\nu_0}-u_{x_j,\zeta_j,\nu_j}|$ on $A_j$, we have $|\nabla v_j|\le (3/\delta)|u_{x_0,\zeta_0,\nu_0}-u_{x_j,\zeta_j,\nu_j}|$ on $A_j$. It follows that
\begin{equation*}
\int_{A_j} (1+|\nabla v_j|^p)\,dy\le \rho^n-(\rho-3\delta)^n+ \frac{3^p}{\delta^p}\eta_j\le
3n\delta\rho^{n-1}+\frac{3^p}{\delta^p}\eta_j,
\end{equation*}
where $\eta_j:=\int_{A_j}|u_{x_0,\zeta_0,\nu_0}-u_{x_j,\zeta_j,\nu_j}|^pdy\to0+$, as $j\to +\infty$.

On the other hand by \eqref{cubes3} we have $S_{v_j}\cap A_j\subset \big(\Pi^{\nu_0}_{x_0}\cap  Q^{\nu_0}_{\rho+2\delta}(x_0)\setminus Q^{\nu_0}_{\rho-5\delta}(x_0)\big)\cup \big(\Pi^{\nu_j}_{x_j}\cap  Q^{\nu_j}_{\rho}(x_j)\setminus Q^{\nu_j}_{\rho-3\delta}(x_j)\big)$. Moreover there exists a constant $M_1>0$ such that $|[v_j]|\le M_1$ $\hs^{n-1}$-a.e.\ on $S_{v_j}\cap A_j$ for every $j\ge j_\delta$. Therefore
\begin{equation}\label{S vj}
\int_{S_{v_j}\cap A_j}
\!\!\!\!\!  \!\!\!\!\!   (1+|[ v_j]|)\,d\hs^{n-1}\le 2(1+M_1)\big((\rho+2\delta)^{n-1}-(\rho-5\delta)^{n-1})
\le 14\delta(1+M_1)(n-1)(2\rho)^{n-2}.
\end{equation}
From \eqref{quasi minimality} and \eqref{mH<=H}-\eqref{S vj} it follows that  for every $j\ge j_\delta$
$$
m_H(u_{x_j,\zeta_j,\nu_j},Q^{\nu_j}_\rho(x_j))\le m_H(u_{x_0,\zeta_0,\nu_0},Q^{\nu_0}_\rho(x_0)) +\e +M_2\delta + c_2\frac{3^p}{\delta^p}\eta_j,
$$
where $M_2:=3nc_2\rho^{n-1}+
14c_5(1+M_1)(n-1)(2\rho)^{n-2}$. Taking the limit as $j\to+\infty$ we get
$$
\limsup_{j\to+\infty} m_H(u_{x_j,\zeta_j,\nu_j},Q^{\nu_j}_\rho(x_j))\le m_H(u_{x_0,\zeta_0,\nu_0},Q^{\nu_0}_\rho(x_0)) +\e +M_2\delta.
$$
Since $\e>0$ and $\delta\in(0,\delta_0/2)$ are arbitrary, we obtain
$$
\limsup_{j\to+\infty} m_H(u_{x_j,\zeta_j,\nu_j},Q^{\nu_j}_\rho(x_j))\le m_H(u_{x_0,\zeta_0,\nu_0},Q^{\nu_0}_\rho(x_0)),
$$
which proves the upper semicontinuity of the restriction of $(x,\zeta,\nu)\mapsto m_H(u_{x,\zeta,\nu},Q^\nu_\rho(x))$ to  $\R^n{\times}\R^m_0{\times}\widehat\Sph^{n-1}_+$. The same proof holds for $\R^n{\times}\R^m_0{\times}\widehat\Sph^{n-1}_-$.

To prove (c), we fix $x_0$, $\zeta_0$, $\nu_0$, $(x_j)$, $(\nu_j)$, and $\e>0$ as in the proof of (b). By definition there exist $w_0\in X$, with $w_0|_{Q^{\nu_0}_\rho(x_0)}\in SBV_{\mathrm{pc}}(Q^{\nu_0}_\rho(x_0),\R^m)$, and $\delta_0\in(0,\rho/3)$ such that
\begin{align}\label{w0=ux0}
&w_0=u_{x_0,\zeta_0,\nu_0}\quad\mathcal{L}^n\hbox{-a.e.\ in }Q^{\nu_0}_\rho(x_0)\setminus Q^{\nu_0}_{\rho-3\delta_0}(x_0),
\\
&H(w_0,Q^\nu_\rho(x_0))<m^{\mathrm pc}_H(u_{x_0,\zeta_0,\nu_0},Q^{\nu_0}_\rho(x_0))+\e.
\label{quas minim}
\end{align}
Fix $\delta\in (0,\delta_0/2)$ and let $j_\delta$ be an integer such that \eqref{cubes} and \eqref{cubes3} are satisfied for every $j\ge j_\delta$. We define
$w_j\in SBV^p_{\mathrm{loc}}(\R^n,\R^m)$ by
\begin{equation}\label{def zj}
w_j:=
\begin{cases}
w_0&\hbox{in }Q^{\nu_j}_{\rho-2\delta}(x_j),
\\
u_{x_j,\zeta_0,\nu_j}&\hbox{in }\R^n\setminus Q^{\nu_j}_{\rho-2\delta}(x_j).
\end{cases}
\end{equation}
Then $w_j|_{Q^{\nu_j}_\rho(x_j)}\in SBV_{\mathrm{pc}}(Q^{\nu_j}_\rho(x_j),\R^m)$ and $w_j=u_{x_j,\zeta_0,\nu_j}$ in a neighbourhood of $\partial Q^{\nu_j}_\rho(x_j)$, hence $m^{\mathrm{pc}}_H(u_{x_j,\zeta_0,\nu_j},Q^{\nu_j}_\rho(x_j))\le H(w_j,Q^{\nu_j}_\rho(x_j))$. Therefore, using $(h1)$-$(h3)$ and setting $A_j:=Q^{\nu_j}_\rho(x_j)\setminus Q^{\nu_j}_{\rho-3\delta}(x_j)$, we obtain
\begin{equation}\label{mpcH<=H}
m^{\mathrm{pc}}_H(u_{x_j,\zeta_0,\nu_j},Q^{\nu_j}_\rho(x_j))\le H(w_0,Q^{\nu_j}_{\rho-2\delta}(x_j))
 +
c_5\int_{S_{w_j}\cap A_j}
\!\!\!\!\!  \!\!\!\!\!   (1+|[ w_j]|)\,d\hs^{n-1}.
\end{equation}
By \eqref{cubes} and \eqref{w0=ux0} we have $w_j=u_{x_0,\zeta_0,\nu_0}$ on $Q^{\nu_j}_{\rho-2\delta}(x_j)\setminus Q^{\nu_j}_{\rho-3\delta}(x_j)$ for every $j\ge j_\delta$, 
while by \eqref{def zj} we have 
$w_j=u_{x_j,\zeta_0,\nu_j}$ in $Q^{\nu_j}_\rho(x_j)\setminus Q^{\nu_j}_{\rho-2\delta}(x_j)$.
Therefore $S_{w_j}\cap A_j\subset \big(\Pi^{\nu_0}_{x_0}\cap  Q^{\nu_j}_{\rho-2\delta}(x_j)\setminus Q^{\nu_j}_{\rho-3\delta}(x_j)\big)\cup \Sigma_j\cup \big(\Pi^{\nu_j}_{x_j}\cap  Q^{\nu_j}_{\rho}(x_j)\setminus Q^{\nu_j}_{\rho-2\delta}(x_j)\big)\subset \big(\Pi^{\nu_0}_{x_0}\cap  Q^{\nu_0}_\rho(x_0)\setminus Q^{\nu_0}_{\rho-5\delta}(x_0)\big)\cup \Sigma_j\cup \big(\Pi^{\nu_j}_{x_j}\cap  Q^{\nu_j}_{\rho}(x_j)\setminus Q^{\nu_j}_{\rho-2\delta}(x_j)\big)$, where $ \Sigma_j$ is the set of points $y\in \partial Q^{\nu_j}_{\rho-2\delta}(x_j)$ such that $(y-x_j){\,\cdot\,}\nu_j$ and $(y-x_0){\,\cdot\,}\nu_0$ have opposite sign.
Moreover $|[w_j]|= |\zeta_0|$ $\hs^{n-1}$-a.e.\ on $S_{w_j}\cap A_j$ for every $j\ge j_\delta$ and $\sigma_j:=\hs^{n-1}(\Sigma_j)\to 0$ as $j\to +\infty$. 
Therefore
\begin{equation}\label{S wj}
\int_{S_{w_j}\cap A_j}
\!\!\!\!\!  \!\!\!\!\!   (1+|[ w_j]|)\,d\hs^{n-1}\le 2\big(1+|\zeta_0|)\big(\rho^{n-1}-(\rho-5\delta)^{n-1}+\sigma_j\big)
\le 2(1+|\zeta_0|)\big((n-1)\rho^{n-2}\delta+\sigma_j\big).
\end{equation}
From \eqref{quas minim}, \eqref{mpcH<=H}, and \eqref{S wj} it follows that  for every $j\ge j_\delta$
$$
m_H(u_{x_j,\zeta_0,\nu_j},Q^{\nu_j}_\rho(x_j))\le m_H(u_{x_0,\zeta_0,\nu_0},Q^{\nu_0}_\rho(x_0)) +\e +2 c_5(1+|\zeta_0|)\big((n-1)\rho^{n-2}\delta+\sigma_j\big).
$$ 
Since $\e>0$ and $\delta\in(0,\delta_0/2)$ are arbitrary and $\sigma_j\to 0$, we obtain
$$
\limsup_{j\to+\infty} m_H(u_{x_j,\zeta_0,\nu_j},Q^{\nu_j}_\rho(x_j))\le m_H(u_{x_0,\zeta_0,\nu_0},Q^{\nu_0}_\rho(x_0)),
$$
which proves the upper semicontinuity of the restriction of $(x,\zeta,\nu)\mapsto m^{\mathrm{pc}}_H(u_{x,\zeta_0,\nu},Q^\nu_\rho(x))$ to  $\R^n{\times}\widehat\Sph^{n-1}_+$. The same proof holds for $\R^n{\times}\widehat\Sph^{n-1}_-$.
\end{proof}

\begin{lem}[Monotonicity in $\rho$]\label{monotonicity mG}
Let $x\in\R^n$, $\xi\in \R^{m{\times}n}$, $\zeta\in\R^m_0$, and $\nu\in \Sph^{n-1}$. Under the assumptions of Lemma \ref{upper semicontinuity}
the functions 
\begin{align*}
&\rho\mapsto m_H(\ell_\xi,Q_\rho(x))-c_2(1+|\xi|^p)\rho^n &{} &\rho\mapsto m^{1,p}_H(\ell_\xi,Q_\rho(x))-c_2(1+|\xi|^p)\rho^n,
\\
&\rho\mapsto m_H(u_{x,\zeta,\nu},Q^\nu_\rho(x))-c_5(1+|\zeta|)\rho^{n-1} &{}
&\rho\mapsto m^{\mathrm{pc}}_H(u_{x,\zeta,\nu},Q^\nu_\rho(x))-c_5(1+|\zeta|)\rho^{n-1}
\end{align*}
are nonincreasing in $(0,+\infty)$.
\end{lem}

\begin{proof}
Let $\rho_2>\rho_1>0$ and $\e>0$ be fixed. By the definition of $m_H$ there exist $u_1\in X$, with $u_1|_{Q_{\rho_1}\!(x)}\in SBV^p(Q_{\rho_1}\!(x),\R^m)$, and $\rho'\in (0,\rho_1)$, such that $u_1=\ell_\xi$ $\mathcal{L}^n$-a.e.\ in $Q_{\rho_1}\!(x)\setminus Q_{\rho'}\!(x)$ and
\begin{equation}\label{q min u1}
H(u_1,Q_{\rho_1}\!(x))<m_H(\ell_\xi,Q_{\rho_1}\!(x))+\e.
\end{equation}
Let  $u_2$ be defined by
\begin{equation*}\label{def u2}
u_2:=
\begin{cases}
u_1&\hbox{in }Q_{\rho_1}\!(x),
\\
\ell_\xi&\hbox{in }\R^n\setminus Q_{\rho_1}\!(x).
\end{cases}
\end{equation*}
Then $u_2=\ell_\xi$ in a neighbourhood of $\partial Q_{\rho_2}\!(x)$, hence $m_H(\ell_\xi,Q_{\rho_2}\!(x))\le H(u_2,Q_{\rho_2}\!(x))$. Let us fix $\rho''\in (\rho',\rho_1)$. Using $(h1)$-$(h3)$, from the previous inequality  we obtain
\begin{equation*}
m_H(\ell_\xi,Q_{\rho_2}\!(x))\le H(u_1,Q_{\rho_1}\!(x))+H(\ell_\xi, Q_{\rho_2}\!(x)\setminus Q_{\rho''}\!(x))
\le H(u_1,Q_{\rho_1}\!(x))
+c_2(1+|\xi|^p)(\rho_2^{n}-(\rho'')^{n})
\end{equation*}
Taking the limit as $\rho''\to \rho_1-$, from \eqref{q min u1} we obtain
$$
m_H(\ell_\xi,Q_{\rho_2}\!(x))\le m_H(\ell_\xi,Q_{\rho_1}\!(x))+\e+ c_2(1+|\xi|^p)(\rho_2^{n}-\rho_1^{n}).
$$
Taking the limit as $\e\to 0+$ we obtain
$$
m_H(\ell_\xi,Q_{\rho_2}\!(x))- c_5(1+|\xi|^p)\rho_2^{n}\le m_H(\ell_\xi,Q_{\rho_1}\!(x))- c_2(1+|\xi|^p)\rho_1^{n},
$$
which proves the monotonicity of $\rho\mapsto m_H(\ell_\xi,Q_\rho(x))-c_2(1+|\xi|^p)\rho^n$. The same proof holds for  $\rho\mapsto m^{1,p}_H(\ell_\xi,Q_\rho(x))-c_2(1+|\xi|^p)\rho^n$.

We now consider $m^{\mathrm{pc}}_H$. By definition there exist $v_1\in X$, with $v_1|_{Q^\nu_{\rho_1}\!(x)}\in SBV_{\mathrm{pc}}(Q^\nu_{\rho_1}\!(x),\R^m)$, and $\rho'\in (0,\rho_1)$ such that $v_1=u_{x,\zeta,\nu}$ $\mathcal{L}^n$-a.e.\ in $Q_{\rho_1}\!(x)\setminus Q_{\rho'}\!(x)$  and
\begin{equation}\label{q min v1}
H(v_1,Q^\nu_{\rho_1}\!(x))<m^{\mathrm{pc}}_H(u_{x,\zeta,\nu},Q^\nu_{\rho_1}\!(x))+\e.
\end{equation}
Let 
$v_2$ be defined by
\begin{equation*}\label{def v2}
v_2:=
\begin{cases}
v_1&\hbox{in }Q^\nu_{\rho_1}\!(x),
\\
u_{x,\zeta,\nu}&\hbox{in }\R^n\setminus Q^\nu_{\rho_1}\!(x).
\end{cases}
\end{equation*}
Then $v_2=u_{x,\zeta,\nu}$ in a neighbourhood of $\partial Q^\nu_{\rho_2}\!(x)$, hence 
$m^{\mathrm{pc}}_H(u_{x,\zeta,\nu},Q^\nu_{\rho_2}\!(x))\le H(v_2,Q^\nu_{\rho_2}\!(x))$.  Let us fix $\rho''\in (\rho',\rho_1)$. Using $(h1)$-$(h3)$, from the previous inequality we obtain
\begin{align*}
m^{\mathrm{pc}}_H(u_{x,\zeta,\nu},Q^\nu_{\rho_2}\!(x_0))&\le H(v_1,Q^\nu_{\rho_1}\!(x))+H(u_{x,\zeta,\nu}, Q^\nu_{\rho_2}\!(x)\setminus Q^\nu_{\rho''}\!(x))
\\
&\le H(v_1,Q^\nu_{\rho_1}\!(x))+c_5(1+|\zeta|)(\rho_2^{n-1}-(\rho'')^{n-1}).
\end{align*}
Taking the limit as $\rho''\to \rho_1-$, from \eqref{q min v1} we obtain
$$
m^{\mathrm{pc}}_H(u_{x,\zeta,\nu},Q^\nu_{\rho_2}\!(x))\le m^{\mathrm{pc}}_H(u_{x,\zeta,\nu},Q^\nu_{\rho_1}\!(x))+\e + c_5(1+|\zeta|)(\rho_2^{n-1}-\rho_1^{n-1}).
$$
Taking the limit as $\e\to 0+$ we obtain
$$
m^{\mathrm{pc}}_H(u_{x,\zeta,\nu},Q^\nu_{\rho_2}\!(x_0))- c_5(1+|\zeta|)\rho_2^{n-1}\le m^{\mathrm{pc}}_H(u_{x,\zeta,\nu},Q^\nu_{\rho_1}\!(x_0))- c_5(1+|\zeta|)\rho_1^{n-1},
$$
which proves the monotonicity of $\rho\mapsto m^{\mathrm{pc}}_H(u_{x,\zeta,\nu},Q^\nu_\rho(x))-c_5(1+|\zeta|)\rho^{n-1}$. The same proof holds for $\rho\mapsto m_H(u_{x,\zeta,\nu},Q^\nu_\rho(x))-c_5(1+|\zeta|)\rho^{n-1}$.
\end{proof}

\begin{lem}[Borel measurability]\label{measurability}
Let $(f_k)$ be a sequence in $\mathcal{F}$ and let $(g_k)$ be a sequence in $\mathcal{G}$. Then for every $\e>0$ the functions $f'$, $f''$, 
 $f^{\e,p}$, and $g^{\e,p}$ defined in \eqref{f'}, \eqref{f''}, \eqref{f^eta}, and \eqref{g^eta} are  Borel measurable.
Moreover, for every $\zeta_0\in\R^m_0$ the functions 
$$
(x,\nu)\mapsto g'(x,\zeta_0,\nu)\quad\hbox{and}\quad (x,\nu)\mapsto g''(x,\zeta_0,\nu)
$$
defined in \eqref{g'} and \eqref{g''} are Borel measurable in $\R^n{\times}\Sph^{n-1}$.
\end{lem}

\begin{proof} We prove the result only for $f'$, the proof for $f''$, $f^{\e,p}$, $g^{\e,p}$, $g'$, and $g''$ being analogous. For every $x\in\R^n$, $\xi\in\R^{m{\times}n}$, and $\rho>0$ we set 
$$
\psi(x,\xi,\rho):= \liminf_{k \to + \infty}
m_{F_k}^{1,p}(\ell_\xi, Q_\rho(x)). 
$$
By Lemma \ref{monotonicity mG} for every $x\in\R^n$ and every $\xi\in\R^{m{\times}n}$ the function $\rho\mapsto \psi(x,\xi,\rho)
-c_2(1+|\xi|^p)\rho^{n}$ is nonincreasing on $(0,+\infty)$. It follows that
$$
\lim_{\rho'\to\rho-}\psi(x,\xi,\rho')\ge \psi(x,\xi,\rho)\ge \lim_{\rho'\to\rho+}\psi(x,\xi,\rho')\quad\hbox{for every }x\in\R^n,\ \xi\in\R^{m{\times}n}, \hbox{ and }\rho>0.
$$
Therefore, if $D$ is a countable dense subset of $(0,+\infty)$, we have
$$
\limsup_{\rho \to 0+}  \frac{1}{\rho^{n-1}} \psi(x,\xi,\rho)=
\limsup_{\rho \to 0+,\,\rho\in D}  \frac{1}{\rho^{n-1}} \psi(x,\xi,\rho),
$$
hence 
$$
f' (x, \xi) = \limsup_{\rho \to 0+,\,\rho\in D} \liminf_{k \to + \infty} \frac{1}{\rho^{n-1}} 
m_{F_k}^{1,p}(\ell_\xi, Q_\rho(x))
$$
for every  $x\in\R^n$ and $\xi\in\R^{m{\times}n}$. The conclusion follows now from Lemma \ref{upper semicontinuity}.
\end{proof}

The next lemma provides all properties of the functions $f'$ and $f''$.

\begin{lem}\label{f'' g' in F and G}
Let $(f_k)$ be  a sequence in $\mathcal{F}$ and let $f'$ and $f''$  be as in 
\eqref{f'} and \eqref{f''}. 
Then $f'$, $f'' \in \mathcal{F}$.
\end{lem}

\begin{proof}
Property $(f1)$ for $f'$ and $f''$ is proved in Lemma \ref{measurability}. The proof of $(f2)$ for $f'$ and $f''$ can be easily obtained by adapting the proof of the same property for $f^{\e,p}$ established in Theorem~\ref{Gamma e}. In fact it is enough to deduce from \eqref{new *} that \eqref{new **} holds, with $m_{E^{\e,p}}$ replaced by $m^{1,p}_{F_k}$. The conclusion then follows from  \eqref{f'} and \eqref{f''}, passing to the limit first as $k\to+\infty$ and then as $\rho\to 0+$.

We now prove $(f3)$ for $f'$ and $f''$.
Let $x$, $\xi \in\R^{m{\times}n}$ be fixed. 
By $(f3)$ for $f_k$ for any $\rho > 0$ and $u \in W^{1,p} (Q_{\rho} (x),\R^m)$
with $u=\ell_\xi$ near $\partial Q_{\rho} (x)$ we have
\begin{align*}
\frac{1}{\rho^n} F_k (u, Q_\rho (x)) \geq 
\frac{c_1}{\rho^n} \int_{Q_\rho (x)} | \nabla u |^p \, dy \geq
c_1 \Big| \frac{1}{\rho^n} \int_{Q_\rho (x)} \nabla u  \, dy \Big|^p
= c_1 |\xi|^p,
\end{align*}
where we used Jensen's inequality and the boundary conditions for $u$.
By letting $k \to +\infty$ and then $\rho \to 0+$,  the lower bounds for $f'$ and $f''$ follow from  \eqref{f'} and \eqref{f''}.

Since  $f_k$ satisfies $(f4)$, for any $\rho > 0$ we also have
\begin{align*}
\frac{1}{\rho^n} m^{1,p}_{F_k} (\ell_\xi, Q_\rho (x))
\leq \frac{1}{\rho^n} F_k (\ell_\xi, Q_\rho (x)) \leq c_2 (1 + |\xi|^p).
\end{align*}
By letting $k \to +\infty$ and then $\rho \to 0+$ we obtain the upper bounds for $f'$ and $f''$.
\end{proof}

The next lemma provides all properties of the functions $g'$ and $g''$.

\begin{lem}\label{g' in G}
Let $(g_k)$ be a sequence in $\mathcal{G}$, and let $g'$ and $g''$ be as in \eqref{g'} and \eqref{g''}. Then $g' , g'' \in \mathcal{G}$. 
\end{lem}

\begin{proof} We prove $(g1)$--$(g7)$ only for $g'$, the proof for $g''$ being similar.

We start by proving $(g2)$. To this end fix $x\in\R^n$, $\zeta_1$, $\zeta_2 \in \R^m_0$, $\nu\in \Sph^{n-1}$,  $k\in\N$, $\rho>0$.  There exists $u_1\in L^0(\R^n,\R^m)$, with $u_1|_{Q_\rho ^\nu(x)}\in SBV_{\mathrm{pc}}(Q_\rho ^\nu(x),\R^m)$ and
 $u_1=u_{x,\zeta_1,\nu}$ in a neighbourhood of $\partial Q^\nu_\rho (x)$, such that
\begin{equation} \label{quasimin}
G_k(u_1,Q_\rho ^\nu(x)) \leq m^{\mathrm{pc}}_{G_k}(u_{x,\zeta_1,\nu}, Q_\rho ^\nu(x)) 
+ \e\, \rho^{n-1}.
\end{equation}
Let $E:=\{y\in Q_\rho ^\nu(x): u_1(y)=\zeta_1\}$ and let $\chi_E$ be its characteristic function. Then $\chi_E\in BV(Q_\rho ^\nu(x))$ and $S_{\chi_E}\cap Q_\rho ^\nu(x)\subset S_{u_1}\cap Q_\rho ^\nu(x)$ (see \cite[Theorem 4.23]{AFP}). 

Let $u_2:=u_1+(\zeta_2-\zeta_1)\chi_E$. Then $u_2|_{Q_\rho ^\nu(x)}\in SBV_{\mathrm{pc}}(Q_\rho ^\nu(x),\R^m)$ and
 $u_2=u_{x,\zeta_2,\nu}$ in a neighbourhood of $\partial Q_\rho ^\nu(x)$. Moreover
 $S_{u_2}\subset S_{u_1}$ and  $[u_2]=[u_1]$ $\hs^{n-1}$-a.e.\ on $S_{u_1}\setminus S_{\chi_E}$, while $[u_2]=[u_1]+\zeta_2-\zeta_1$ $\hs^{n-1}$-a.e.\ on $S_{\chi_E}\cap S_{u_2}\cap Q_\rho ^\nu(x)$. 
 By $(g2)$ we have
\begin{equation*}\label{continuity3}
G_k(u_2,Q_\rho ^\nu(x))\le G_k(u_1,Q_\rho ^\nu(x))+\sigma_2(|\zeta_1-\zeta_2|)\big(G_k(u_1,Q_\rho ^\nu(x))+G_k(u_2,Q_\rho ^\nu(x))\big)
\end{equation*}
hence
$$
\big(1-\sigma_2(|\zeta_1-\zeta_2|)\big) G_k(u_2,Q_\rho ^\nu(x))\le \big(1+\sigma_2(|\zeta_1-\zeta_2|)\big)G_k(u_1,Q_\rho ^\nu(x)).
$$
Assume that $\sigma_2(|\zeta_1-\zeta_2|)<1$. 
Then the previous inequality together with \eqref{quasimin} yield
$$
\big(1-\sigma_2(|\zeta_1-\zeta_2|)\big) m^{\mathrm{pc}}_{G_k}(u_{x,\zeta_2,\nu}, Q_\rho ^\nu(x))
\le \big(1+\sigma_2(|\zeta_1-\zeta_2|)\big)\big(m^{\mathrm{pc}}_{G_k}(u_{x,\zeta_1,\nu}, Q_\rho ^\nu(x))+ \e\, \rho^{n-1}\big).
$$
Dividing by $ \rho^{n-1}$ and taking the liminf as $k\to+\infty$, then the limsup as $\rho\to0^+$, and finally the limit as $\e\to 0+$ we obtain
$$
\big(1-\sigma_2(|\zeta_1-\zeta_2|)\big) g'(x,\zeta_2,\nu)
\le \big(1+\sigma_2(|\zeta_1-\zeta_2|)\big)g'(x,\zeta_1,\nu)
$$
hence 
\begin{equation}\label{continuity4}
g'(x,\zeta_2,\nu)
\le g'(x,\zeta_1,\nu) + \sigma_2(|\zeta_1-\zeta_2|)\big(g'(x,\zeta_1,\nu)+
g'(x,\zeta_2,\nu)\big).
\end{equation}
Inequality \eqref{continuity4} is trivial if  $\sigma_2(|\zeta_1-\zeta_2|)\ge 1$. Then $(g2)$ can be obtained from \eqref{continuity4} by interchanging the roles of $\zeta_1$ and $\zeta_2$. 

We now observe that the Borel measurability of $g'$ on $\R^n{\times} \R^m_0{\times}\Sph^{n-1}$ follows from Lemma \ref{measurability} and from  the continuity estimate $(g2)$. This concludes the proof of $(g1)$.

To prove $(g3)$ for $g'$, let us fix $x\in\R^n$, $\zeta_1$, $\zeta_2\in\R^m_0$, with $|\zeta_1|\le |\zeta_2|$, $\nu\in\Sph^{n-1}$, and  a rotation $R$ on $\R^m$ such that $aR\zeta_2=\zeta_1$, where $a:= |\zeta_1|/|\zeta_2|\le 1$. For every $k$ the functions $g_k$ satisfy $(g3)$, thus for every $\rho>0$ and every $u\in SBV_{\mathrm{pc}}(Q^\nu_\rho(x),\R^m)$ we have
$$
\int_{S_u\cap Q^\nu_\rho(x)}g_k(y, aR[u](y), \nu_u(y))d\hs^{n-1}(y)
\le  c_3 \int_{S_u\cap Q^\nu_\rho(x)}g_k(y, [u](y), \nu_u(y))d\hs^{n-1} 
$$
Since $aR\zeta_2=\zeta_1$, this inequaliy implies that
$$
m_{G_k}^{\mathrm{pc}}(u_{x,\zeta_1,\nu},Q^\nu_\rho(x))= m_{G_k}^{\mathrm{pc}}(u_{x,aR\zeta_2,\nu},Q^\nu_\rho(x))\le c_3 m_{G_k}^{\mathrm{pc}}(u_{x,\zeta_2,\nu},Q^\nu_\rho(x)).
$$
Using \eqref{g'} we obtain
$
g'(x,\zeta_1,\nu)\le   c_3\, g'(x,\zeta_2,\nu)
$,
which proves $(g3)$.

To prove $(g4)$  for $g'$, let us fix $x\in\R^n$,  $\zeta_1$, $\zeta_2\in\R^m_0$, with $c_3|\zeta_1|\le |\zeta_2|$,  $\nu\in\Sph^{n-1}$, and   a rotation $R$ on $\R^m$ such that $aR\zeta_2=\zeta_1$, where $a:= |\zeta_1|/|\zeta_2|\le 1/ c_3\le 1$.
For every $k$ the functions $g_k$ satisfy $(g4)$, thus for every $\rho>0$ and every $u\in SBV_{\mathrm{pc}}(Q^\nu_\rho(x),\R^m)$ we have
$$
\int_{S_u\cap Q^\nu_\rho(x)}g_k(y, aR[u](y), \nu_u(y))d\hs^{n-1}(y)
\le \int_{S_u\cap Q^\nu_\rho(x)}g_k(y, [u](y), \nu_u(y))d\hs^{n-1} 
$$
Since $aR\zeta_2=\zeta_1$, this inequaliy implies that
$$
m_{G_k}^{\mathrm{pc}}(u_{x,\zeta_1,\nu},Q^\nu_\rho(x))= m_{G_k}^{\mathrm{pc}}(u_{x,aR\zeta_2,\nu},Q^\nu_\rho(x))\le m_{G_k}^{\mathrm{pc}}(u_{x,\zeta_2,\nu},Q^\nu_\rho(x)).
$$
Using \eqref{g'} we obtain
$
g'(x,\zeta_1,\nu)\le g'(x,\zeta_2,\nu)
$,
which proves $(g4)$.

To prove $(g5)$ for $g'$, let us fix $x\in\R^n$, $\zeta\in \R^m_0$, $\nu\in \Sph^{n-1}$, $k\in\N$, and $\rho>0$. Since $(g5)$ holds for $g_k$, for every $u\in L^0(\R^n,\R^m)$, with $u|_{Q_\rho ^\nu(x)}\in SBV_{\mathrm{pc}}(Q_\rho ^\nu(x),\R^m)$ we have $G_k(u, Q_\rho ^\nu(x))\ge c_4\hs^{n-1}(S_u)$. If $u$ agrees with $u_{x,\zeta,\nu}$ in a neighbourhood of $\partial Q_\rho ^\nu(x)$, each straight line intersecting $Q_\rho ^\nu(x)$ and parallel to $\nu$ meets $S_u$ (see \cite[Theorem 3.108]{AFP}). This implies that $\hs^{n-1}(S_u)\ge \rho^{n-1}$, which, together with the previous estimate, gives $G_k(u, Q_\rho ^\nu(x))\ge c_4 \rho^{n-1}$. Taking the infimum with respect to $u$ we obtain $m^{\mathrm{pc}}_{G_k}(u_{x,\zeta,\nu},Q_\rho ^\nu(x))\ge c_4 \rho^{n-1}$. By \eqref{g'} this implies $(g5)$ for $g'$.

On the other hand, appealing to $(g6)$ for $g_k$ we have 
$$
m^{\mathrm{pc}}_{G_k}(u_{x,\zeta,\nu},Q_\rho ^\nu(x))\le G_k(u_{x,\zeta,\nu}, Q_\rho ^\nu(x))\le c_5(1+|\zeta|) \rho^{n-1}.
$$ 
Then the latter leads to $(g6)$ for $g'$ by \eqref{g'}.

To prove the symmetry condition $(g7)$, we observe that $u_{x,-\zeta,-\nu}=u_{x,\zeta,\nu}-\zeta$
for every $x\in\R^n$, $\zeta\in \R^m_0$, $\nu\in \Sph^{n-1}$, and $t>0$. Therefore $u\in SBV_{\mathrm{pc}}(Q_\rho ^\nu(x),\R^m)$ satisfies $u=u_{x,-\zeta,-\nu}$ in a neighbourhood of $\partial Q_\rho ^\nu(x)$ if and only if $u=v-\zeta$ for some $v\in SBV_{\mathrm{pc}}(Q_\rho ^\nu(x),\R^m)$ satisfying $v=u_{x,\zeta,\nu}$ in a neighbourhood of $\partial Q_\rho ^\nu(x)$. Since $Q_\rho ^{-\nu}(x)=Q_\rho ^{\nu}(x)$ by (k) and (l) in Section~\ref{Notation}, it follows that $m^{\mathrm{pc}}_{G_k}(u_{x,-\zeta,-\nu},Q_\rho ^{-\nu}(x))=m^{\mathrm{pc}}_{G_k}(u_{x,\zeta,\nu},Q_\rho ^\nu(x))$ for every $k$. By \eqref{g'} this implies that $g'(x,\zeta,\nu)=g'(x,-\zeta,-\nu)$, which proves $(g7)$ for $g'$.
\end{proof}

\section*{Acknowledgments}
The authors wish to acknowledge the hospitality of the University of Bath, where part of this work was carried out. 
F. Cagnetti was supported by the EPSRC under the Grant EP/P007287/1 ``Symmetry of Minimisers in Calculus of Variations''.
The research of G. Dal Maso was partially funded by the European Research Council under
Grant No. 290888 ``Quasistatic and Dynamic Evolution Problems in Plasticity
and Fracture''. G. Dal Maso is a member of the Gruppo Nazionale per l'Analisi Matematica, la Probabilit\`a e le loro Applicazioni (GNAMPA) of the Istituto Nazionale di Alta Matematica (INdAM).
L. Scardia acknowledges support by the EPSRC under the Grant EP/N035631/1 ``Dislocation patterns 
beyond optimality''.


\end{document}